\documentclass{amsart}
\usepackage[T1]{fontenc}

\usepackage{color}
\usepackage{verbatim}
\usepackage{array}
\usepackage{graphicx}
\usepackage[colorlinks=true,linkcolor=blue,citecolor=blue, urlcolor=blue, pdfhighlight =/O]{hyperref}
\usepackage{tikz-cd}
\usepackage{amsmath}
\usepackage{amssymb}
\usepackage{mathtools}
\usepackage{stmaryrd}

\DeclareMathOperator{\Frac}{\normalfont{Frac}}

\DeclareMathOperator{\ord}{\normalfont{ord}}

\DeclareMathOperator{\type}{\normalfont{type}}
\DeclareMathOperator{\GL}{GL}
\DeclareMathOperator{\SL}{SL}
\DeclareMathOperator{\PGL}{PGL}

\newcommand{\slashbar}[2]{\big|_{#1,#2}}

\newtheorem{theorem}{Theorem}[subsection]
\newtheorem{proposition}[theorem]{Proposition}
\newtheorem{lemma}[theorem]{Lemma}
\newtheorem{corollary}[theorem]{Corollary}
\newtheorem{conjecture}[theorem]{Conjecture}
\newtheorem*{theorem*}{Theorem}
\newtheorem*{proposition*}{Proposition}

\theoremstyle{remark}
\newtheorem*{remark*}{Remark}

\theoremstyle{definition}
\newtheorem{definition}[theorem]{Definition}
\newtheorem{example}[theorem]{Example}
\newtheorem{remark}[theorem]{Remark}

\title[Non-Archimedean Poincaré series]
      {Non-Archimedean Poincaré series and geodesics on the Bruhat--Tits tree}

\author{Milan Berger-Guesneau}
\address{Department of Mathematics, Pennsylvania State University, University Park, Pennsylvania, United States of America}
\email{mbb6409@psu.edu}

\author{Mihran Papikian}
\address{Department of Mathematics, Pennsylvania State University, University Park, Pennsylvania, United States of America}
\email{papikian@psu.edu}

\thanks{The second author was supported in part by the Simons Foundation, award number MPS-TSM-00008093.} 

\date{}


\begin{document}

\begin{abstract}
Adapting a construction of Kurihara in the setting of Drinfeld modular forms, we define Poincaré series on the Drinfeld half-plane $\Omega$. These series are built from products of meromorphic $1$-forms that are naturally associated to geodesics on the Bruhat--Tits tree. We establish convergence under a finiteness condition on the geodesics, verify this condition in several cases, and give sufficient conditions for the resulting cusp forms to be nonzero. For the principal congruence subgroup $\Gamma(\mathfrak{n})$ of $\GL_2(\mathbb{F}_q[T])$, we construct explicit linearly independent families of Poincaré series by lifting certain $k$-forms from the components of the analytic reduction of $\Gamma(\mathfrak{n})\backslash\Omega$. We formulate conjectures
on the vanishing orders at cusps, and we prove the first of them for an explicit family of Poincaré series by computing the corresponding expansions at the cusps; as an application, we obtain the Drinfeld modular forms $h$ and $\Delta$ as Poincaré series (up to a sign). Finally, for cocompact groups attached to quaternion algebras over $\mathbb{F}_q(T)$ that split at $\infty$, we show that the Poincaré series span the whole space of modular forms of given weight and type.
\end{abstract}

\maketitle

\section{Introduction}

\subsection*{Motivation and background}

Let $\mathcal{H}\coloneqq \{z \in \mathbb{C} : \operatorname{Im}(z) > 0\}$
denote the complex upper half-plane. The group $\SL_2(\mathbb{R})$
acts on $\mathcal{H}$ by linear fractional transformations. For
$\gamma=\begin{pmatrix} a & b \\ c & d \end{pmatrix}
\in\SL_2(\mathbb{R})$, let $j_\gamma(z)\coloneqq cz+d$. Let $\Gamma$ be a
finite-index subgroup of $\SL_2(\mathbb{Z})$. A standard way to
produce modular forms for $\Gamma$ is to average a function
$g\colon\mathcal{H}\to\mathbb{C}$ over the $\Gamma$-orbit of $z$,
with a correction factor ensuring the correct transformation law.
For integers $k\geqslant 3$ and $m\geqslant 1$, the
\textit{classical Poincaré series}
\begin{equation*}
    \sum_{[\gamma]\in\Gamma_\infty\backslash\Gamma}
    \frac{e^{2\pi im\gamma z/h}}{j_\gamma(z)^{k}}
\end{equation*}
converges uniformly on compact subsets of $\mathcal{H}$ and defines
a cusp form of weight $k$, where $\Gamma_\infty$ is the stabilizer
of $\infty$ in $\Gamma$ and $h$ is the width of the cusp $\infty$.
Using the Petersson inner product, one shows that these series for
$m\geqslant 1$ span the entire space of cusp forms of weight $k$ \cite[Chapter 3]{Gunning}. 

\noindent One can also associate Poincaré-style modular forms to geodesics in
$\mathcal{H}$: for two distinct points
$a,b\in\mathbb{P}^1(\mathbb{R})$, there exists a unique geodesic in
$\mathcal{H}$ joining them. Katok \cite{Katok1985} associates to
each closed geodesic fixed by a hyperbolic element $\gamma_0$ of
$\Gamma$ a \textit{relative Poincaré series} $\theta_{k,\gamma_0}$
for $k\geqslant 4$ even, and shows that these span the space
$S_k(\Gamma)$ of cusp forms. As for classical Poincaré series, the proof
that these series span relies crucially on the Petersson inner
product.

\noindent An analogous result holds in the function field setting.  Using the notation defined below, let $\Gamma$ be a subgroup of $\GL_2(F_\infty)$
commensurable with $G$ (i.e., $\Gamma\cap G$ has finite index in both
$\Gamma$ and $G$), and let $\widetilde{\Gamma}\coloneqq \Gamma/Z(\Gamma)$. For $\alpha\in\Gamma$ and $a\in\mathbb{P}^1(F)\cup\Omega$, Gekeler \cite[Proposition 2.3]{Gekeler1997} showed that the
theta function
\begin{equation*}
    u_\alpha(z)\coloneqq \prod_{[\gamma]\in\widetilde{\Gamma}}
    \frac{z-\gamma a}{z-\gamma(\alpha a)}
\end{equation*}
converges locally uniformly on $\Omega$ and defines a rigid analytic
function there, invertible and nonvanishing at the cusps
\cite[Proposition 2.6]{Gekeler1997}. This function does not depend on the choice of
$a$ \cite[Corollary 2.7]{Gekeler1997}, and depends only on the class of
$\alpha$ in $\overline{\Gamma}\coloneqq \Gamma^{\mathrm{ab}}/
(\Gamma^{\mathrm{ab}})_{\mathrm{tors}}$
\cite[Corollary 2.11]{Gekeler1997}. Products of this shape appear in
\cite[p.~47]{GP1980} and \cite[(5.2)]{GekelerReversat1996}, where the
parameter $a$ is taken in $\Omega$. Since $u_\alpha$ does not depend on $a\in\mathbb{P}^1(F)\cup\Omega$, we may take $a\in\mathbb{P}^1(F)$ while still
appealing to the results of \cite{GekelerReversat1996}. The logarithmic
derivative of $u_\alpha$ can be written as a series over $\widetilde{\Gamma}$:
\begin{equation}\label{equationDlogUalpha}
    \mathrm{dlog}(u_\alpha)\coloneqq \frac{u_\alpha'}{u_\alpha}=
    \sum_{[\gamma]\in\widetilde{\Gamma}}
    \left(\frac{1}{z-\gamma a}-\frac{1}{z-\gamma(\alpha a)}\right)
\end{equation}
Since $u_\alpha$ is invertible at the cusps, $\mathrm{dlog}(u_\alpha)$
vanishes there to order at least $2$. The group $\overline{\Gamma}$ is
free of rank $g$, where $g$ is the genus of the Drinfeld modular curve
$X_\Gamma$; choosing generators
$\overline{\alpha_1},\ldots,\overline{\alpha_g}$, the
$\mathrm{dlog}(u_{\alpha_i})$ for $1\leqslant i\leqslant g$ form a basis
of the space $M_{2,1}^2(\Gamma)$ of doubly cuspidal forms of weight $2$
and type $1$ for $\Gamma$ (see Definition \ref{defDMF} for the definition of the order of vanishing of a modular form, and \cite[(6.5.4)]{GekelerReversat1996} for the result).

\noindent Poincaré series built from \emph{several} geodesics were introduced
in the $p$-adic setting by Kurihara \cite{Kurihara1994}, for
discrete torsion-free cocompact subgroups $\Gamma$ of
$\PGL_2(\mathbb{Q}_p)$. These series are holomorphic $k$-forms on
the $p$-adic Drinfeld half-plane
$\Omega_{\mathbb{C}_p}\coloneqq \mathbb{P}^1(\mathbb{C}_p)\setminus
\mathbb{P}^1(\mathbb{Q}_p)$. Kurihara proved that the space of
$\Gamma$-invariant holomorphic $k$-forms on $\Omega_{\mathbb{C}_p}$
is spanned by such Poincaré
series \cite[Proposition 2.4]{Kurihara1994}. Our goal in this paper
is to adapt and generalize Kurihara's construction to the function
field setting. Some proofs in Kurihara's paper are quite terse; we
take the opportunity to present the construction in full detail.
The crucial difference in our setting is that the group $\Gamma$ we consider need not be cocompact and may have
cusps.
\medskip

\subsection*{Notation}
Let $q=p^e$ be an odd prime power. Let $A\coloneqq \mathbb{F}_q[T]$, 
$F\coloneqq \mathbb{F}_q(T)$, and 
$G\coloneqq \GL_2(A)$. Let $\ord$ be the valuation on $F$ at
$\infty$ normalized by $\ord(a/b)=\deg(b)-\deg(a)$, let
$\pi\coloneqq 1/T$ be the uniformizer, and let
$|\cdot|\coloneqq q^{-\ord(\cdot)}$. We denote by
$F_\infty\coloneqq \mathbb{F}_q(\!(\pi)\!)$ the completion of $F$, by
$\mathcal{O}_\infty\coloneqq \mathbb{F}_q \llbracket\pi
\rrbracket$ the ring of integers of
$F_\infty$, and by $\mathbb{C}_\infty$ the completion of an
algebraic closure of $F_\infty$. The Drinfeld half-plane $\Omega\coloneqq  \mathbb{P}^1(\mathbb{C}_\infty)\setminus \mathbb{P}^1(F_\infty)=\mathbb{C}_\infty\setminus F_\infty$
carries a natural structure of rigid analytic space, and the group
$\GL_2(F_\infty)$ acts on $\Omega$ by linear fractional
transformations. For $\gamma=\begin{pmatrix}a&b\\c&d\end{pmatrix}
\in\GL_2(F_\infty)$, we write $j_\gamma(z)\coloneqq cz+d$ for the
automorphy factor.

\noindent Let $\mathcal{T}$ denote the \textit{Bruhat--Tits tree} of
$\PGL_2(F_\infty)$, and let $\mathcal{T}(\mathbb{R})$ denote its
geometric realization \cite[p.~14]{Serre1980}. The set of ends of
$\mathcal{T}$ is canonically identified with
$\mathbb{P}^1(F_\infty)$, and the \textit{building map}
$\lambda\colon\Omega\to\mathcal{T}(\mathbb{R})$ connects the rigid
analytic geometry of $\Omega$ to the combinatorics of $\mathcal{T}$
(see \S\ref{sectionTree}). For two distinct points
$a,b\in\mathbb{P}^1(F_\infty)$, we denote by $[a\to b]$ the
geodesic in $\mathcal{T}$ connecting $a$ to $b$, and we define the
holomorphic $1$-form
\begin{equation*}
    s(z;a,b)\coloneqq \left(\frac{1}{z-b}-\frac{1}{z-a}\right)\mathrm{d}z
\end{equation*}
on $\Omega$. Up to scaling, $s(z;a,b)$ is the unique meromorphic
$1$-form on $\mathbb{P}^1$ with simple poles at $a$ and $b$. One
can define the valuation of any $k$-form $\omega$ on $\Omega$ at
$z\in\Omega$ \eqref{equationOrdKForm}; we prove that the valuation
of $s(z;a,b)$ at $z$ equals the distance from $\lambda(z)$ to the
geodesic $[a\to b]$ in $\mathcal{T}(\mathbb{R})$
(Proposition \ref{propositionValuationS}). A similar formula
appeared in Kurihara \cite[Proposition 2.1]{Kurihara1994}.

\noindent Let $\Gamma$ be a discrete subgroup of $\GL_2(F_\infty)$ with
$\det(\Gamma)\subset\mathcal{O}_\infty^\times$. Given
$r\geqslant 1$ pairs of distinct points
$a_i,b_i\in\mathbb{P}^1(F_\infty)$, exponents $k_i\geqslant 1$
with $\sum_{i=1}^r k_i=k$, a subgroup $H\subset\Gamma$ fixing each
$a_i$ and $b_i$, and $l\in\mathbb{Z}$, we define the
\emph{Poincaré series} (Definition \ref{defPS})
\begin{equation*}
    P(z)\coloneqq \sum_{[\gamma]\in H\backslash\Gamma}(\det\gamma)^l
    \prod_{i=1}^r\gamma^*s(z;a_i,b_i)^{k_i},
\end{equation*}
where $\gamma^*s(z;a,b)\coloneqq \left(\frac{1}{\gamma z-b}-\frac{1}
{\gamma z-a}\right)\mathrm{d}(\gamma z)$ is the pullback of
$s(z;a,b)$ by $\gamma$. Writing $P(z)=f(z)\,(\mathrm{d}z)^k$ and
$u(z;a,b)\coloneqq (z-b)^{-1}-(z-a)^{-1}$, we also refer to the function
\begin{equation*}
    f(z)=\sum_{[\gamma]\in H\backslash\Gamma}(\det\gamma)^l
    \prod_{i=1}^r u(z;\gamma^{-1}a_i,\gamma^{-1}b_i)^{k_i}
\end{equation*}
as a Poincaré series. Note that \eqref{equationDlogUalpha} is a Poincaré series with $r=k=1$, $l=0$ and $H=Z(\Gamma)$ for the geodesic $[\alpha a\to a]$.

\subsection*{Main results}

Our first result establishes convergence under a natural finiteness
condition \textbf{(P)} on the geodesics $[a_i\to b_i]$ (see
Theorem \ref{thmConvergence} for the precise statement);  
we verify condition \textbf{(P)} in several natural situations.

\begin{theorem*}[Theorem \ref{theorem(P)}]
    The Poincaré series $f$ converges locally uniformly on
    $\Omega$, provided one of the following holds:
    \begin{enumerate}
        \item[\textnormal{(i)}] $r=1$,
            $H=\langle\delta\rangle$ with $\delta\in\Gamma$
            hyperbolic and $\{a,b\}$ its fixed points;
        \item[\textnormal{(ii)}] $r\geqslant 2$, $H=\{I\}$,
            $\bigcap_{i=1}^r\{a_i,b_i\}=\emptyset$;
        \item[\textnormal{(iii)}] $r=1$, $H=\{I\}$,
            $a,b\in\mathbb{P}^1(F)$, and $\Gamma$ commensurable
            with $G=\GL_2(A)$;
        \item[\textnormal{(iv)}] $r\geqslant 2$, $H=\{I\}$,
            $\bigcap_{i=1}^r\{a_i,b_i\}=\{c\}$ with
            $c\in\mathbb{P}^1(F)$, and $\Gamma$ commensurable
            with $G$.
    \end{enumerate}
\end{theorem*}

\begin{remark*}
    The analogous construction in the classical setting, i.e., averaging products of
$s(z;a,b)=\left(\frac{1}{z-b}-\frac{1}{z-a}\right)\mathrm{d}z$
over $\SL_2(\mathbb{Z})$, does not appear to converge in general:
the convergence proof in the non-Archimedean setting relies
essentially on the ultrametric inequality (see the proof of Theorem \ref{thmConvergence}).
\end{remark*}

\noindent We then give sufficient conditions for $f$ to be nonzero
(Theorems \ref{theoremIntersectionVertex}
and \ref{theoremIntersectionEdge}). For example, if the geodesics
$[a_i\to b_i]$ meet at a single vertex $v$ whose stabilizer
$\Gamma_v\coloneqq \{\gamma\in\Gamma:\gamma v=v\}$ is contained in $H$,
then the identity term in the Poincaré series dominates on
$\Omega_v\coloneqq \lambda^{-1}(v)$, and $f\neq 0$.

\medskip

\noindent In \S\ref{sectionModularForms}, we restrict to groups $\Gamma$ that
are either commensurable with $G$, or discrete cocompact in
$\GL_2(F_\infty)$; for example, such groups arise as
unit groups of orders in quaternion algebras over $F$ that split
at $\infty$. In this context, $f$ is
a Drinfeld cusp form (see Definition \ref{defDMF}) of weight $2k$
and type $k+l$ for the group $\Gamma$
(Corollary \ref{coroCuspForm}).

\noindent We then apply these results to the principal congruence subgroup
\begin{equation*}
    \Gamma(\mathfrak{n})\coloneqq \ker(\GL_2(A)\to\GL_2(A/\mathfrak{n}A)),
\end{equation*}
where $\mathfrak{n}\in A$ is monic with
$d\coloneqq \deg(\mathfrak{n})\geqslant 1$. The quotient graph
$\Gamma(\mathfrak{n})\backslash\mathcal{T}$ consists of a finite
graph $(\Gamma(\mathfrak{n})\backslash\mathcal{T})^0$ together with
$c$ half-lines corresponding to cusps, where $c$ is the number of
cusps of the Drinfeld modular curve $X_{\Gamma(\mathfrak{n})}$, the
compactification of the rigid analytic curve
$\Gamma(\mathfrak{n})\backslash\Omega$.

\begin{theorem*}[Theorem \ref{theoremPSspan}]
    Let $k\geqslant 2$ and $\mathfrak{n}\in A$ monic with
    $\deg(\mathfrak{n})\geqslant 1$. The Poincaré series with
    $r\geqslant 2$ span a subspace of
    $M_{2k}^1(\Gamma(\mathfrak{n}))$ of dimension at least
    $(2k-1)(g-1)+kc$, where $g$ denotes the genus of
    $X_{\Gamma(\mathfrak{n})}$.
\end{theorem*}

\noindent Since there is no analog of the Petersson inner product in this
setting, the proof proceeds differently from the classical case.
Following Kurihara, we associate linearly independent Poincaré
series to vertices and edges of the quotient graph
$\Gamma(\mathfrak{n})\backslash\mathcal{T}$. Each stable vertex
$\bar{v}$ of the finite graph gives rise to an irreducible component
$E_{\bar{v}}\cong\mathbb{P}^1_{\mathbb{F}_q}$ of the analytic
reduction of $\Gamma(\mathfrak{n})\backslash\Omega$. We define \emph{Kurihara differentials}
(Definition \ref{defKuriharaDiff}) as $k$-forms on
$\mathbb{P}^1_{\mathbb{F}_q}$ built from products of
$s(z;\alpha,\beta)$ for $\alpha,\beta\in\mathbb{P}^1(\mathbb{F}_q)$,
and show that they span the space of $k$-forms on
$\mathbb{P}^1_{\mathbb{F}_q}$ whose poles all lie in
$\mathbb{P}^1(\mathbb{F}_q)$ and have order at most $k-1$. This space has dimension $N\coloneqq k(q-1)-q$ (Lemma \ref{lemmaKuriharaSpan}). For
each stable vertex, we construct $N$ linearly independent Poincaré
series whose reductions to the corresponding component are Kurihara
differentials, and for each stable edge, one additional Poincaré
series. Linear independence is established by reduction to the residue field $\overline{\mathbb{F}_q}$ of $\mathbb{C}_\infty$.

\medskip

\noindent Numerical computations with SageMath suggest precise vanishing
orders for the Poincaré series at the cusps of
$\Gamma(\mathfrak{n})$. We formulate two conjectures.
Conjecture \ref{conjectureVanishingExact} predicts that a vertex
Poincaré series with endpoints
$a_i,b_i\in\mathbb{P}^1(\mathbb{F}_q)$ vanishes to order at least
$2k-p_\alpha$ at a cusp represented by
$\alpha\in\mathbb{P}^1(\mathbb{F}_q)$, where $p_\alpha$ is the pole
order of the underlying Kurihara differential at $\alpha$; if
$p_\alpha\leqslant q-1$ for all $\alpha$, the bound is expected to be exact. Conjecture \ref{conjectureVanishingOrders} predicts that the  Poincaré series constructed in the proof of Theorem \ref{theoremPSspan} lie in $M_{2k}^{k}(\Gamma(\mathfrak{n}))$.

\begin{proposition*}[Proposition \ref{propConditionalBasis}]
    Assume Conjecture \ref{conjectureVanishingOrders}. Then the  Poincaré series constructed in the proof of Theorem \ref{theoremPSspan} form a basis of $M_{2k}^{k}(\Gamma(\mathfrak{n}))$.
\end{proposition*}

\medskip

\noindent In \S\ref{subsectionQuaternionic} we treat cocompact groups arising from quaternion algebras. Let $B$ be a division quaternion algebra over $F$ that splits at
$\infty$, let $\mathcal{O}$ be a maximal $A$-order in $B$, and let $\Gamma$ be a
subgroup of finite index of the image of $\mathcal{O}^\times$ in
$\GL_2(F_\infty)$. Such a $\Gamma$ is discrete and cocompact, so that
$\Gamma\backslash\mathcal{T}$ is a finite graph and $\Gamma$ has no cusps. Here
the Poincaré series span the \emph{whole} space of modular forms.

\begin{theorem*}[Theorem \ref{theoremPSspanQuat}]
    Let $k\geqslant 2$, let $l$ be an integer such that $2l\equiv 0\bmod\#Z(\Gamma)$, and put $m\coloneqq k+l$. Then $M_{2k,m}(\Gamma)$ is
    spanned by Poincaré series.
\end{theorem*}

\noindent The reason we obtain generation, and not only a lower bound as in the case of $\Gamma(\mathfrak{n})$, is that $\Gamma$ has no cusps. For
$\Gamma(\mathfrak{n})$, identifying the span would require knowing the order of vanishing at the cusps of \emph{all} the Poincaré series of Theorem \ref{theoremPSspan}, which is the content of Conjectures
\ref{conjectureVanishingExact} and \ref{conjectureVanishingOrders}.

\medskip

\noindent In \S\ref{sectionExisting} we show that several modular forms
occurring in the literature are Poincaré series.
First, the \textit{Drinfeld--Poincaré series} $P_{2k,k}$ introduced
by Petrov \cite{Petrov2015} coincides, up to sign, with the cusp form associated to the Poincaré series $\sum_{\gamma\in G}\gamma^*s(z;\infty,0)^k$. Second, we compute
explicitly the expansion at the cusps of a family of
Poincaré series for $\Gamma(\mathfrak{n})$: writing
$r\coloneqq (q+1)/2$ and fixing a decomposition
$\mathbb{P}^1(\mathbb{F}_q)=\coprod_{i=1}^{r}\{a_i,b_i\}$ into
pairs, we consider
\begin{equation*}
    f_\kappa^{\mathfrak{n}}
    \coloneqq\sum_{\gamma\in\Gamma(\mathfrak{n})}
    \prod_{i=1}^{r}u(z;\gamma a_i,\gamma b_i)^{\kappa}
    \in M_{\kappa(q+1)}^{1}(\Gamma(\mathfrak{n})).
\end{equation*}

\begin{theorem*}[Theorem \ref{theoremFisC0h}]
    Let $1\leqslant\kappa\leqslant q-1$. Then $f_\kappa^{\mathfrak{n}}$
    vanishes to order exactly $\kappa q$ at every cusp of
    $\Gamma(\mathfrak{n})$ represented by an element of
    $\mathbb{P}^1(\mathbb{F}_q)$.
\end{theorem*}

\noindent This proves Conjecture \ref{conjectureVanishingExact} for this family (Remark \ref{remarkConjectureCase}). For
$\mathfrak{n}=T$ every cusp is represented by an element of
$\mathbb{P}^1(\mathbb{F}_q)$, and comparing with the dimension formula
\eqref{eqdim} we obtain that $f_\kappa^{T}$ is a
constant multiple of $h^{\kappa}$, where $h=P_{q+1,1}$ is
one of the two generators of the graded algebra of Drinfeld modular forms for $G$ (Theorem \ref{theoremFisC0h2}); in particular $h$ and $-\Delta=h^{q-1}$ are Poincaré series.

\medskip

\begin{remark*}
We work throughout with $A=\mathbb{F}_q[T]$ and the completion
$F_\infty$ of $F=\Frac(A)$ at $\infty$, but Section \ref{sectionTree},
Theorem \ref{thmConvergence}, cases (i) and (ii) of
Theorem \ref{theorem(P)}, and the nonvanishing results of
\S\ref{sectionPS} use nothing about $A$: they remain valid when
$F_\infty$ is replaced by an arbitrary non-Archimedean local field
$K$,
with $q$ replaced by the cardinality of the residue field. The same is true
of Theorem \ref{theoremPSspanQuat} for any discrete cocompact
$\Gamma\subset\GL_2(K)$ with $\widetilde{\Gamma}$ torsion-free (with some minor modifications).
\end{remark*}

\begin{remark*}
Poincaré series have a natural generalization to higher rank: the
Drinfeld half-plane is replaced by the Drinfeld symmetric space
$\Omega^{n-1}$, the geodesics of $\mathcal{T}$ by apartments of
the Bruhat--Tits building of $\GL_n(K)$, and $s(z;a,b)$ by
the logarithmic $(n-1)$-form attached to a configuration of $n$
hyperplanes. In the forthcoming work
\cite{BGP2} we construct the corresponding Poincaré series and
prove their convergence, the analogue of
Theorem \ref{thmConvergence}, based on geometric results on
intersections of apartments and their divergence. In particular,
the analogue of the condition
$\bigcap_{i=1}^r\{a_i,b_i\}=\emptyset$ in case (ii) of
Theorem \ref{theorem(P)} is that the apartments have no common
face in the building at infinity.
\end{remark*}

\subsection*{Acknowledgments}

The work on this project began during the second author's visit to
the Max Planck Institute for Mathematics in Bonn in 2023, and
continued while the two authors were visiting the National Center for
Theoretical Sciences in Taiwan in the summer of 2025; we thank
both institutes for their hospitality and excellent working
conditions. We are grateful to Fu-Tsun Wei for his hospitality
and for many helpful conversations; the computation in the proof
of Proposition \ref{propDrinfeldPoincare} is due to him.

\section{Geodesics and valuation}\label{sectionTree}

\subsection{The Bruhat--Tits tree}

An \textit{$\mathcal{O}_\infty$-lattice} in $F_\infty^2$ is a rank-two $\mathcal{O}_\infty$-submodule. We denote by $[L]$ the homothety class of a lattice $L$. 
The \textit{Bruhat--Tits tree} $\mathcal{T}$ of $\PGL_2(F_\infty)$ is the graph whose vertex set $X(\mathcal{T})$ consists of these homothety classes. 
Two vertices are adjacent if they admit representatives $L$ and $L'$ such that $\pi L\subset L'\subset L$. One shows that $\mathcal{T}$ is an undirected $(q+1)$-regular tree \cite[p. 70-72]{Serre1980}. 

Let $\{f_1,f_2\}$ be the canonical basis of the space $F_\infty^2$ of row vectors, and let $v_n$ be the vertex $[\pi^n\mathcal{O}_\infty f_1+\mathcal{O}_\infty f_2]$ for $n\in\mathbb{Z}$. The group $\GL_2(F_\infty)$ acts on $X(\mathcal{T})$ via $\gamma \cdot [L] = [L\gamma^{-1}]$, and the map $\gamma\mapsto \gamma  v_0$ induces a bijection
\begin{equation*}
    \GL_2(F_\infty)/F_\infty^\times\GL_2(\mathcal{O}_\infty)\xrightarrow{\sim} X(\mathcal{T}).
\end{equation*}
Every coset in $\GL_2(F_\infty)/F_\infty^\times\GL_2(\mathcal{O}_\infty)$ has a unique representative of the form $\begin{pmatrix}
\pi^n & u \\
0 & 1
\end{pmatrix}$ with $n\in \mathbb{Z}$ and $u\in F_\infty/\pi^n \mathcal{O}_\infty$. We choose as a representative either $u=0$ or an element $u\in F$ with $\ord(u)<n$. For example, the vertex $v_n$ is represented by the matrix 
$\begin{pmatrix} \pi^{-n} & 0 \\ 0 & 1 \end{pmatrix}$. 

Recall that an \textit{end} of $\mathcal{T}$ is an equivalence class of half-lines in $\mathcal{T}$, where two half-lines are equivalent if they eventually coincide. There exists a $\GL_2(F_\infty)$-equivariant bijection between the set of ends of $\mathcal{T}$ and the projective line $\mathbb{P}^1(F_\infty)$ \cite[(1.3.6)]{GekelerReversat1996}. 
In this paper, we follow the convention in \cite[(1.6.3)]{GekelerReversat1996}: we label the end corresponding to $(x:y)$ by $(-y:x)$. We also implicitly identify $z\in F_\infty$ with $(z:1)\in \mathbb{P}^1(F_\infty)$. Under these identifications, (the equivalence class of) the half-line $\{v_n\}_{n\geqslant 0}$ is labeled $(1:0)=\infty$, while $\{v_{-n}\}_{n\geqslant 0}$ is labeled $0$. 

We root $\mathcal{T}$ at the end $\infty\in\mathbb{P}^1(F_\infty)$,
which induces a natural orientation on the edges: for adjacent
vertices $v$ and $w$, exactly one of them, say $w$, lies on the
unique half-line starting at the other one and representing the
end $\infty$; we declare $w$ to be the initial vertex of the edge
$\{v,w\}$.
Let $v \in X(\mathcal{T})$ be represented by $\begin{pmatrix} \pi^n & u \\ 0 & 1 \end{pmatrix}$, and let $w$ be a vertex adjacent to $v$. 
If $w$ lies in the direction of $\infty$, then $w$ is represented by $\begin{pmatrix} \pi^{n-1} & \overline{u} \\ 0 & 1 \end{pmatrix}$, 
where $\overline{u}$ is the class of $u$ modulo $\pi^{n-1}\mathcal{O}_\infty$. 
Otherwise, $w$ is represented by $\begin{pmatrix} \pi^{n+1} & u + \alpha\pi^n \\ 0 & 1 \end{pmatrix}$ for a unique $\alpha \in \mathbb{F}_q$. 

For any two distinct points $a,b\in \mathbb{P}^1(F_\infty)$, there exists a unique oriented doubly infinite path in $\mathcal{T}$ connecting them, which we denote by $[a\to b]$ and refer to as the \textit{oriented geodesic} from $a$ to $b$. The oriented geodesic $[\infty\to 0]$ is called the \textit{principal geodesic}; its vertices are $v_n$ for $n\in\mathbb{Z}$.

A vertex $v$ represented by $\begin{pmatrix} \pi^n & u \\ 0 & 1 \end{pmatrix}$ lies on the principal geodesic if and only if $u=0$ (in which case $v$ coincides with $v_{-n}$). 
The combinatorial distance from $v$ to the principal geodesic is given by
\begin{equation}\label{eq:distance-geodesic}
    d(v,[\infty\to 0]) = \begin{cases}
        n-\ord(u) &\text{if } u\neq 0,\\
        0 &\text{if } u=0.
    \end{cases}
\end{equation}
Indeed, if $u\neq 0$, moving from $v$ toward $\infty$ at each step yields the unique path without backtracking from $v$ to the geodesic:
\begin{equation*}
    \begin{pmatrix}\pi^n & u \\ 0 & 1\end{pmatrix}
    \to \begin{pmatrix}\pi^{n-1} & \bar{u} \\ 0 & 1\end{pmatrix}
    \to \cdots
    \to \begin{pmatrix}\pi^{\ord(u)} & 0 \\ 0 & 1\end{pmatrix}.
\end{equation*}
The terminal vertex $v_{-\ord(u)}$ is the nearest vertex on the principal geodesic to $v$, and the path has length $n - \ord(u)$.
The same computation gives the distance to the base vertex: since the path from $v$ to $v_0$ passes through $v_{-\ord(u)}$, we have $d(v,v_0) = n-\ord(u)+|\ord(u)|$ if $u\neq 0$, and $d(v,v_0) =       |n|$ otherwise.

We denote by $\mathcal{T}(\mathbb{R})$ the geometric realization of $\mathcal{T}$, equipped with the path metric extending the combinatorial distance $d$ on $X(\mathcal{T})$.
A (non-oriented) \textit{open edge} is a subset of $\mathcal{T}(\mathbb{R})$ of the form $\{(1-t)v+tw:0<t<1\}$ where $v,w\in X(\mathcal{T})$ are adjacent. We denote by $Y^\circ(\mathcal{T})$ the set of open edges.

Let $p$ be a point in $\mathcal{T}(\mathbb{R})$. If $p$ lies on an open
edge, let $v$ and $w$ denote its initial and terminal vertices, and let
$t \in (0,1)$ be the distance from $v$ to $p$. Then $v$ and $w$ are
represented respectively by $\begin{pmatrix} \pi^{n} & u \\ 0 & 1
\end{pmatrix}$ and $\begin{pmatrix} \pi^{n+1} & u+\alpha\pi^n \\ 0 & 1
\end{pmatrix}$ for a unique $\alpha \in \mathbb{F}_q$. If $p$ is a vertex
represented by $\begin{pmatrix} \pi^{n} & u \\ 0 & 1 \end{pmatrix}$, we
set $v = p$, $t = 0$, $\alpha = 0$, and let $w$ be the adjacent vertex
represented by $\begin{pmatrix} \pi^{n+1} & u \\ 0 & 1 \end{pmatrix}$.
In both cases, $p = (1-t)v + tw$ and $t \in [0,1)$ is the distance from $v$ to $p$.

\begin{lemma}
    The distance from $p$ to the principal geodesic on $\mathcal{T}(\mathbb{R})$ is
    \begin{equation}\label{equationDistanceGeodesic}
        d(p,[\infty\to 0])=
        \begin{cases}
            n-\ord(u)+t &\text{if } u\neq 0,\\
            t &\text{if } u=0,\ \alpha\neq 0,\\
            0 &\text{if } u=0,\ \alpha=0.
        \end{cases}
    \end{equation}
\end{lemma}
\begin{proof}
    First assume $u=0$, so the vertex $v$ lies on the geodesic. In this case, $p$ lies on the principal geodesic if and only if $\alpha=0$. If $\alpha\neq 0$, then $d(p,[\infty\to 0])=d(p,v)=t$.

    \noindent If $u\neq 0$, then $v$ lies between the principal geodesic and $p$; therefore $d(p,[\infty\to 0])=d(p,v)+d(v,[\infty\to 0])=n-\ord(u)+t$.
\end{proof}

 Let $\lambda \colon \Omega \to \mathcal{T}(\mathbb{R})$ denote the \textit{building map}, whose image is the set of rational points $\mathcal{T}(\mathbb{Q}) \subset \mathcal{T}(\mathbb{R})$ \cite[(1.5.2)]{GekelerReversat1996}. 
 For any point $p\in\mathcal{T}(\mathbb{Q})$ and any open edge $e\in Y^\circ(\mathcal{T})$, define $\Omega_p\coloneqq \lambda^{-1}(p)\subset\Omega$ and $\Omega_e\coloneqq \lambda^{-1}(e)\subset\Omega$. The Drinfeld half-plane $\Omega$ decomposes as a disjoint union:
\begin{equation*}
\Omega=\left(\coprod_{v\in X(\mathcal{T})}\Omega_{v}\right)\coprod\left(\coprod_{e\in Y^\circ(\mathcal{T})}\Omega_{e}\right).
\end{equation*}
 Let $|z|_i=\inf\{|z-x|\,:\,x\in F_\infty\}$ be the \textit{imaginary absolute value} of $z\in\Omega$. Recall that for all $\gamma\in\GL_2(F_\infty)$
\cite[(1.1.5)]{GekelerReversat1996},
 \begin{equation}\label{equationImaginaryPart}
     |\gamma z|_i=\frac{|\det\gamma|}{|j_\gamma(z)|^2}|z|_i.
 \end{equation}
 One has 
\begin{equation*}
    \Omega_{v_0}=\{z\in\mathbb{C}_\infty\,:\,|z|=|z|_i=1\}.
\end{equation*}
The action of $\GL_2(F_\infty)$ on $\mathcal{T}$ induces a natural action on $\mathcal{T}(\mathbb{R})$:
\begin{equation*}
    \gamma\cdot  [(1-t)v+tw]\coloneqq (1-t)(\gamma\cdot  v)+t(\gamma\cdot  w)
\end{equation*}
This action preserves the distance, maps vertices to vertices, open edges to open edges, and satisfies $\gamma\Omega_p=\Omega_{\gamma p}$ for all $p\in \mathcal{T}(\mathbb{R})$. 

Let $p = (1-t)v + tw\in\mathcal{T}(\mathbb{Q})$ with $t \in [0,1)$ as above. Fix an element $\pi_t\in\mathbb{C}_\infty^\times$ with $\ord(\pi_t)=t$ (if $t=0$, we choose $\pi_0=1$). 

\begin{lemma}\label{lemmaPhiP}
    The map $\zeta\mapsto \pi^n\pi_t \zeta+u+\alpha\pi^n$ defines a bijection $\phi_p\colon\Omega_{v_0}\xrightarrow{\sim}\Omega_{p}$. Moreover, $|z|_i=q^{-(n+t)}$ for all $z\in \Omega_{p}$.
\end{lemma}
\begin{proof}
    We first treat the special case $v=v_0$ and $w=v_{-1}$. The matrix $\gamma\coloneqq \begin{pmatrix} \pi & 0 \\ 0 & 1 \end{pmatrix}$ maps $v_0$ to $v_{-1}$, so $\Omega_{v_{-1}}=\gamma\Omega_{v_{0}}=\{z\in\mathbb{C}_\infty\,:\,|z|=|z|_i=q^{-1}\}$.

    \noindent Since $\log |\cdot|$ and $\log |\cdot|_i$ are linear on edges \cite[(1.5.10)]{GekelerReversat1996}, we have 
    \begin{equation*}
        \Omega_{p}=\{z\in\mathbb{C}_\infty\,:\,|z|=|z|_i=q^{-t}\}.
    \end{equation*}
    We now prove the general case. Let $\delta\coloneqq \begin{pmatrix} \pi^{n} & u+\alpha\pi^n \\ 0 & 1 \end{pmatrix}$. We have $\delta v_0=v$ and $\delta v_{-1}=w$, so $\delta^{-1}p=(1-t)v_0+tv_{-1}$. By the special case above, $\Omega_{\delta^{-1}p}=\{z\in\mathbb{C}_\infty\,:\,|z|=|z|_i=q^{-t}\}$. Let $\zeta\in \Omega_{v_{0}}$; then $\pi_t \zeta\in \Omega_{\delta^{-1}p}$, and so $\phi_p(\zeta)=\delta(\pi_t \zeta)\in\Omega_{p}$. Conversely, if $z\in\Omega_{p}$, then $\delta^{-1}z\in \Omega_{\delta^{-1}p}$, so $\pi_t^{-1}(\delta^{-1}z)\in\Omega_{v_0}$. The map $z\mapsto \pi_t^{-1}(\delta^{-1}z)$ is the inverse of $\phi_p$.

    \noindent Finally, if $z=\phi_p(\zeta)=\delta(\pi_t \zeta)\in\Omega_{p}$, then 
    \begin{equation*}
        |z|_i=\frac{|\det\delta|}{|j_\delta(\pi_t\zeta)|^2}|\pi_t\zeta|_i=q^{-(n+t)}.
    \end{equation*}
\end{proof}

\subsection{Valuation of $k$-forms}

We now define the valuation of a $k$-form $\omega(z)=f(z)\,(\mathrm{d}z)^k$ on $\Omega$ at a point $z\in\Omega_p$. Set $\zeta\coloneqq \phi_p^{-1}(z)\in\Omega_{v_0}$; the pullback $\phi_p^*\omega(\zeta)=f(\phi_p(\zeta))\,(\mathrm{d}\phi_p(\zeta))^k$ can be written $g_p(\zeta)\,(\mathrm{d}\zeta)^k$ with 
\begin{equation*}
g_p(\zeta)=\phi_p^*\omega(\zeta)/(\mathrm{d}\zeta)^k=f(\phi_p(\zeta))\,\pi^{kn}\pi_t^k.
\end{equation*}
We define the \textit{valuation of $\omega$ at} $z\in\Omega_p$ to be $\ord_z(\omega)\coloneqq \ord g_p(\zeta)$. By Lemma \ref{lemmaPhiP}, one has
\begin{equation}\label{equationOrdKForm}
    \ord_z(\omega)=\ord f(z)+k(n+t)=\ord f(z)-k\log_q|z|_i,
\end{equation}
which shows that $\ord_z(\omega)$ is independent of the choice of $\pi_t$. For all $\gamma\in\GL_2(F_\infty)$, the valuation of $\gamma^*\omega(z)=f(\gamma z)\,\mathrm{d}(\gamma z)^k$ at $z$ is the valuation of $\omega(z)$ at $\gamma z$, i.e. $\ord_z(\gamma^*\omega)=\ord_{\gamma z}(\omega)$.

\begin{definition}
    Let $K$ be any field, and let $a,b$ be two distinct points in $\mathbb{P}^1(K)$. We define the following $1$-form on $\mathbb{P}^1_{K}$:
\begin{equation*}
    s(z;a,b)\coloneqq \left(\frac{1}{z-b}-\frac{1}{z-a}\right)\, \mathrm{d}z
\end{equation*}
\end{definition}
\noindent The differential form $s(z;a,b)$ is, up to scaling, the unique $1$-form on $\mathbb{P}^1_{K}$ with simple poles at $a$ and $b$ and no other poles (with the convention $1/(z-a)=0$ if $a=\infty$). Let 
\begin{equation*}
    \gamma^*s(z;a,b)=\left(\frac{1}{\gamma z-b}-\frac{1}{\gamma z-a}\right)\, \mathrm{d}(\gamma z)
\end{equation*}
be the pullback of $s(z;a,b)$ by $\gamma\in \GL_2(K)$; it satisfies
\begin{equation}\label{equationS}
    \gamma^*s(z;a,b)=s(z;\gamma^{-1}a,\gamma^{-1}b)
\end{equation}
for any $\gamma\in \GL_2(K)$.
To see \eqref{equationS}, observe that $s(z;a,b)$ is the unique rational $1$-form on $\mathbb{P}^1_{K}$ with at most simple poles, all contained in $\{a,b\}$, and residues $-1$ at $a$ and $+1$ at $b$. The difference of two such forms is a holomorphic $1$-form on
$\mathbb{P}^1_{K}$, hence zero. Since the pullback of a $1$-form under $z\mapsto\gamma z$ preserves residues,
$\gamma^\ast s(z;a,b)$ has simple poles exactly at
$\gamma^{-1}a$ and $\gamma^{-1}b$, with residues $-1$ and $+1$ respectively, and \eqref{equationS} follows.

We now specialize to the case $K=F_\infty$. Let $a,b$ be two distinct points in $\mathbb{P}^1(F_\infty)$; then $s(z;a,b)$ is a holomorphic differential form on $\Omega$.

\begin{proposition}\label{propositionValuationS}
    For all $z\in\Omega_p$, the valuation $\ord_z s(z;a,b)$ is the distance from $p$ to the geodesic $[a\to b]$ on $\mathcal{T}(\mathbb{R})$.
\end{proposition}
\begin{proof}
    Since $\lambda$ is $\GL_2(F_\infty)$-equivariant, acting by a suitable
$\gamma\in\GL_2(F_\infty)$ we may assume $a=\infty$, $b=0$. For convenience, we write $\omega(z)$ for $s(z;\infty,0)=\mathrm{d}z/z$.

\noindent Set $\zeta=\phi_p^{-1}(z)\in\Omega_{v_0}$. We have 
\begin{align*}
    \phi_p^*\omega(\zeta)&=\frac{\mathrm{d}\phi_p(\zeta)}{\phi_p(\zeta)}\\
    &=\frac{\pi^n\pi_t}{\pi^n\pi_t \zeta+u+\alpha\pi^n}\,\mathrm{d}\zeta\\
    &=\frac{1}{\zeta+\pi^{-n}\pi_t^{-1}(u+\alpha\pi^n)}\,\mathrm{d}\zeta
\end{align*}
and thus
\begin{equation*}
    \ord_z(\omega)=-\ord\left( \zeta+ \frac{u+\alpha\pi^n}{\pi^n\pi_t}\right).
\end{equation*}
\begin{itemize}
\item If $u\neq 0$, then $\ord(u)<n$ and
\begin{equation*}
    \ord \left(\frac{u+\alpha\pi^n}{\pi^n\pi_t}\right)=\ord(u)-n-t<0=\ord \zeta.
\end{equation*}
Therefore
\begin{equation*}
    \ord_z(\omega)=-\ord \left(\frac{u+\alpha\pi^n}{\pi^n\pi_t}\right)
    =n-\ord(u)+t.
\end{equation*}
    \item If $u=0$ and $\alpha\neq 0$, then $t>0$ and $\ord_z(\omega)=-\ord(\alpha/\pi_t)=t$.
    \item If $u=0$ and $\alpha=0$, then $\ord_z(\omega)=-\ord \zeta=0$.
\end{itemize}
Comparing with \eqref{equationDistanceGeodesic}, we see that $\ord_z(\omega)=d(p,[\infty\to 0])$.
\end{proof}

Let $r\geqslant 1$, $k_1,\ldots,k_r\geqslant 1$, and $a_1,b_1,\ldots,a_r,b_r\in\mathbb{P}^1(F_\infty)$ be such that $a_i\neq b_i$ and $\{a_i,b_i\}\neq\{a_j,b_j\}$ for $i\neq j$. Let $k\coloneqq \sum_{i=1}^r k_i$.

\noindent For any $p\in\mathcal{T}(\mathbb{R})$, let
$d(p,X(\mathcal{T}))\coloneqq\min_{v\in X(\mathcal{T})}d(p,v)$. The minimum is attained, and equals $\min(t,1-t)$ if
$p=(1-t)v+tw$ with $0\leqslant t<1$.

\begin{corollary}\label{corollaryValuationS}
     Let $s(z)$ be the $k$-form $\prod_{i=1}^r s(z;a_i,b_i)^{k_i}$. For any $\gamma\in \GL_2(F_\infty)$ and $z\in\Omega_p$, 
    \begin{equation*}
        \ord_z(\gamma^*s)=\sum_{i=1}^r k_i\cdot d(\gamma p,[a_i\to b_i]).
    \end{equation*}
    In particular, $\ord_z(\gamma^*s)=0$ if $\gamma p\in \bigcap_{i=1}^r [a_i \to b_i]$ and $\ord_z(\gamma^*s)\geqslant t_p > 0$ otherwise, where
         \begin{equation*}
             t_p\coloneqq \begin{cases}
                 1 &\text{if } p\in X(\mathcal{T}),\\
                 d(p,X(\mathcal{T})) &\text{otherwise},
             \end{cases}
         \end{equation*}
         depends only on $p$.
\end{corollary}
\begin{proof}
    By \eqref{equationS} and Proposition \ref{propositionValuationS},
    \begin{align*}
        \ord_z(\gamma^*s)&=\sum_{i=1}^r k_i\cdot \ord_z s(z;\gamma^{-1}a_i,\gamma^{-1} b_i) \\
        &=\sum_{i=1}^r k_i\cdot d(p,[\gamma^{-1}a_i\to\gamma^{-1} b_i])\\
        &=\sum_{i=1}^r k_i\cdot d(\gamma p,[a_i\to b_i])
    \end{align*}
    where the last equality holds because $\GL_2(F_\infty)$ preserves the distance on $\mathcal{T}(\mathbb{R})$. This sum vanishes if and only if $\gamma p\in [a_i\to b_i]$ for all $1\leqslant i\leqslant r$, which proves the first assertion.

    Now assume $\gamma p\not\in \bigcap_{i=1}^r [a_i\to b_i]$, and let $j$ be such that $\gamma p\not\in[a_j\to b_j]$. Since $k_j\geqslant 1$, one has
    \begin{equation*}
        \ord_z(\gamma^*s)\geqslant k_j\cdot d(\gamma p,[a_j\to b_j])\geqslant d(\gamma p,[a_j\to b_j]).
    \end{equation*}
    If $p$ is a vertex, then so is $\gamma p$, and since geodesics pass through vertices, $d(\gamma p,[a_j\to b_j])\geqslant 1$. If $p$ lies on an open edge, then $d(\gamma p,[a_j\to b_j])\geqslant d(\gamma p,X(\mathcal{T}))=d(p,X(\mathcal{T}))$, again because geodesics pass through vertices.
\end{proof}

\section{Poincaré series}\label{sectionPS}

Let $\Gamma$ be a discrete subgroup of $\GL_2(F_\infty)$ with
$\det(\Gamma)\subset\mathcal{O}_\infty^\times$. The group $\Gamma$
acts on $X(\mathcal{T})$ with finite stabilizers. Indeed, 
$\GL_2(F_\infty)$ acts transitively on $X(\mathcal{T})$, so it suffices to consider the stabilizer of $v_0$; if $\gamma=cg$ with $c\in
F_\infty^\times$ and $g\in\GL_2(\mathcal{O}_\infty)$ stabilizes
$v_0$, then $|c|^2=|\det\gamma|=1$, so
$\gamma\in\GL_2(\mathcal{O}_\infty)$, and
$\Gamma\cap\GL_2(\mathcal{O}_\infty)$ is finite, since $\GL_2(\mathcal{O}_\infty)$ is compact and $\Gamma$ is discrete.

\noindent The combinatorial distance $d(v,\gamma v)$ is even for $\gamma\in \Gamma$ and $v\in X(\mathcal{T})$ \cite[p.75]{Serre1980}. It follows that two adjacent vertices never lie in the same $\Gamma$-orbit, and that $\Gamma$ acts without inversion on $\mathcal{T}$.

Let $r\geqslant 1$, $k_1,\ldots,k_r\geqslant 1$, $l\in \mathbb{Z}$, and
let $H$ be a subgroup of $\Gamma$ such that $\det(h)^l=1$ for all $h\in H$.
Set $k=\sum_{i=1}^r k_i$. Let $a_1,b_1,\ldots,a_r,b_r\in\mathbb{P}^1(F_\infty)$
satisfying the following conditions:
\begin{enumerate}
    \item[(a)] $a_i\neq b_i$ for all $i$,
    \item[(b)] $\{a_i,b_i\}\neq \{a_j,b_j\}$ for $i\neq j$,
    \item[(c)] $h a_i=a_i$ and $h b_i=b_i$ for all $h\in H$ and $1\leqslant i\leqslant r$.
\end{enumerate}
Condition (a) ensures that the geodesic $[a_i\to b_i]$ is well-defined,
condition (b) guarantees that distinct pairs determine distinct geodesics, and condition (c) implies that $\gamma^*s(z)$ depends only on the right coset of $\gamma$ in $H\backslash\Gamma$.

\noindent We set $s(z)\coloneqq \prod_{i=1}^rs(z;a_i,b_i)^{k_i}$. 

\begin{definition}\label{defPS}
    We define the \textit{Poincaré series} associated to this data by
    \begin{equation*}
        P(z)=\sum_{[\gamma]\in H\backslash\Gamma}(\det\gamma)^l \cdot\gamma^*s(z).
    \end{equation*}
\end{definition}
\noindent The condition $\det(h)^l=1$ for all $h\in H$ ensures that $(\det\gamma)^l$ depends only on the coset of $\gamma$ in $H\backslash\Gamma$. Assuming convergence (see Theorems \ref{thmConvergence} and \ref{theorem(P)}), $P(z)$ is a $k$-form on $\Omega$. By \eqref{equationS}, we have 
\begin{align*}
    P(z)&=\sum_{[\gamma]\in H\backslash\Gamma}(\det\gamma)^l\prod_{i=1}^rs(z;\gamma^{-1}a_i,\gamma^{-1}b_i)^{k_i}\\
    &=\sum_{[\gamma]\in \Gamma/H}(\det\gamma)^{-l}\prod_{i=1}^rs(z;\gamma a_i,\gamma b_i)^{k_i},
\end{align*}
where the second equality follows from the substitution
$\gamma\mapsto\gamma^{-1}$, which induces a bijection
$H\backslash\Gamma\to\Gamma/H$.
Set
\begin{equation*}
    u(z;a,b)\coloneqq\frac{1}{z-b}-\frac{1}{z-a},
\end{equation*}
so that $s(z;a,b)=u(z;a,b)\, \mathrm{d}z$. Then $P(z)=f(z)\,(\mathrm{d}z)^k$ with 
\begin{align*}
    f(z)&=\sum_{[\gamma]\in H\backslash\Gamma}(\det\gamma)^l\prod_{i=1}^r u(z;\gamma^{-1}a_i,\gamma^{-1}b_i)^{k_i}\\
    &=\sum_{[\gamma]\in \Gamma/H}(\det\gamma)^{-l}\prod_{i=1}^ru(z;\gamma a_i,\gamma b_i)^{k_i}.
\end{align*}
By abuse of terminology, we also refer to $f\colon\Omega\to\mathbb{C}_\infty$ as a Poincaré series.
\subsection{Convergence}

For $v\in X(\mathcal{T})$ and $N\geqslant 0$, define
\begin{equation*}
    S_N(v)=S_N^{H\backslash\Gamma}(v)\coloneqq \{[\gamma]\in H\backslash\Gamma:
    d(\gamma v,[a_i\to b_i])\leqslant N
    \text{ for all }1\leqslant i\leqslant r\}.
\end{equation*}
Note that condition (c) implies that the distance $d(\gamma v,[a_i\to b_i])$
depends only on the coset of $\gamma$ in $H\backslash\Gamma$.

\noindent Suppose that $S_N(v)$ is finite for some $v\in X(\mathcal{T})$
and every $N\geqslant 0$. Then $S_N(w)$ is finite for every
$w\in X(\mathcal{T})$ and every $N\geqslant 0$: indeed, the
triangle inequality gives $S_N(w)\subseteq S_{N+D}(v)$, where $D\coloneqq d(v,w)$.

\begin{remark}\label{remarkSNVIntersection}
    If $\bigcap_{i=1}^r[a_i\to b_i]$ is nonempty, then
    Lemma \ref{lemmaIntersectionConvex} gives 
    \begin{equation*}
        d\left(\gamma v,\bigcap_{i=1}^r[a_i\to b_i]\right)=\max_{1\leqslant i\leqslant r}d(\gamma v,[a_i\to b_i]).
    \end{equation*}
    Therefore in this case,
    \begin{equation*}
        S_N(v)=\left\{[\gamma]\in H\backslash\Gamma:
        d\left(\gamma v,\bigcap_{i=1}^r[a_i\to b_i]\right)
        \leqslant N\right\}.
    \end{equation*}
\end{remark}

\begin{lemma}\label{lemmaIntersectionConvex}
    Let $C_1,\ldots,C_r$ be convex subsets of $\mathcal{T}(\mathbb{R})$ with nonempty
    intersection $X\coloneqq \bigcap_{i=1}^r C_i$. Then for any point $p\in\mathcal{T}(\mathbb{R})$,
    \begin{equation*}
        d(p,X)=\max_{1\leqslant i\leqslant r}d(p,C_i).
    \end{equation*}
\end{lemma}

\begin{proof}
    Since $X\subseteq C_i$, we have
    $d(p,X)\geqslant d(p,C_i)$ for all $i$, so
    $d(p,X)\geqslant\max_i d(p,C_i)$.
    For the reverse inequality, it suffices by induction on $r$ to treat the case $r=2$.
    Let $c_i\in C_i$ be the nearest point
    of $C_i$ to $p$, and let $x\in X$ be the nearest point of $X$ to $p$. 
    Every path from $p$ to a point in $C_i$ passes through $c_i$, so both $c_1$ and $c_2$ lie on the segment from $p$ to $x$. 
    Without loss of generality,
    $d(p,c_1)\leqslant d(p,c_2)$; then $c_1$ lies on the segment from $c_2$ to $x$. 
    Since $C_1$ is convex and $c_1,x\in C_1$, we have $c_2\in C_1$, hence $c_2\in C_1\cap C_2=X$. 
    Therefore $d(p,X)\leqslant d(p,c_2)=\max_i d(p,C_i)$.
\end{proof}

By Corollary \ref{corollaryValuationS}, the valuation of the term
$\gamma^*s$ at $z\in\Omega_p$ equals
$\sum_{i=1}^r k_i\cdot d(\gamma p,[a_i\to b_i])$. In the
non-Archimedean setting, a series converges if and only if its
terms tend to $0$. The terms of the Poincaré series with small
valuation at $z$ are thus indexed by the cosets $[\gamma]$ for
which all the translates $\gamma\lambda(z)$ remain close to every
geodesic $[a_i\to b_i]$, that is, by the elements of the sets $S_N(v)$ defined
above. This motivates the following finiteness condition:
\begin{equation*}
\textnormal{\textbf{(P)}} \qquad \textit{$S_N(v)$ is finite for
some (equivalently, all) $v\in X(\mathcal{T})$ and all
$N\geqslant 0$.}
\end{equation*}

\begin{theorem}\label{thmConvergence}
    The Poincaré series 
    \begin{equation*}
        f(z)=\sum_{[\gamma]\in H\backslash\Gamma}(\det \gamma)^l\prod_{i=1}^r u(z;\gamma^{-1}a_i,\gamma^{-1}b_i)^{k_i}
    \end{equation*}
    converges locally uniformly on $\Omega$ if and only if \textnormal{\textbf{(P)}} holds. In this case, $f\colon\Omega\to\mathbb{C}_\infty$ is holomorphic, and $P(z)=f(z)\,(\mathrm{d}z)^k$ is a holomorphic $k$-form on $\Omega$. Furthermore, $|f(z)|\leqslant |z|_i^{-k}$ for all $z\in\Omega$.
\end{theorem}

\noindent The proof of Theorem \ref{thmConvergence} requires two preliminary
lemmas: Lemma \ref{lemmaPhiEdge} reduces lower bounds for the
valuations of the terms $\gamma^*s$ on $\mathcal{T}(\mathbb{R})$
to their values at vertices, and Lemma \ref{lemmaImageUm} shows
that each affinoid of a suitable admissible covering of $\Omega$
only meets the fibers of $\lambda$ over a finite subtree. The proof of the theorem is given after Lemma \ref{lemmaImageUm}.

For $\gamma\in\Gamma$, define
$\Phi_\gamma\colon\mathcal{T}(\mathbb{R})\to\mathbb{R}_{\geqslant 0}$ by
\begin{equation*}
    \Phi_\gamma(p)
    \coloneqq \sum_{i=1}^r k_i\cdot d(\gamma p,[a_i\to b_i]).
\end{equation*}
Note that $\Phi_\gamma$ depends only on the right coset of $\gamma$ in $H\backslash\Gamma$. For $z\in\Omega$ and $p=\lambda(z)$, we have $\ord_z(\gamma^*s)=\Phi_\gamma(p)$ by Corollary \ref{corollaryValuationS}.

\begin{lemma}\label{lemmaPhiEdge}
    Let $p\in\mathcal{T}(\mathbb{R})$. Write $p=(1-t)v+tw$ with $v,w\in X(\mathcal{T})$ adjacent; then
    \begin{equation*}
        \Phi_\gamma(p)\geqslant \min\{\Phi_\gamma(v),\Phi_\gamma(w)\}
    \end{equation*}
    for all $[\gamma]\in H\backslash\Gamma$.
\end{lemma}
\begin{proof}
    The map $t\mapsto d(\gamma p,[a_i\to b_i])=d(p,[\gamma^{-1}a_i\to \gamma^{-1}b_i])$ is an affine map $[0,1]\to \mathbb{R}_{\geqslant 0}$; thus
    $t\mapsto \Phi_\gamma(p)$ is also an affine map $[0,1]\to \mathbb{R}_{\geqslant 0}$. Any such map  attains its minimum at $t=0$ or $t=1$.
\end{proof}

For $m\geqslant 0$, define the affinoid subset $U_m\coloneqq \{z\in\Omega\,:\,|z|\leqslant q^{m}, |z|_i\geqslant q^{-m}\}$; then $\Omega=\bigcup_{m\geqslant 1} U_m$ is an admissible covering. 

\begin{lemma}\label{lemmaImageUm}
    $\lambda(U_m)$ is contained in the geometric realization of a finite subtree.
\end{lemma}
\begin{proof}
    Let $z\in U_m$. We denote by $p=(1-t)v+tw$ the image of $z$ in $\mathcal{T}(\mathbb{Q})$, where $v$ is represented by $\begin{pmatrix} \pi^{n} & u \\ 0 & 1
\end{pmatrix}$ and $t\in [0,1)\cap\mathbb{Q}$. We have 
    \begin{equation*}
        q^{-m}\leqslant |z|_i=q^{-(n+t)}
    \end{equation*}
    and
    \begin{equation*}
        q^{-(n+t)}=|z|_i\leqslant |z|\leqslant q^m
    \end{equation*}
    so $-m\leqslant n+t\leqslant m$. As $t\in [0,1)$, it implies $-m\leqslant n\leqslant m$. 

    \noindent By Lemma \ref{lemmaPhiP}, we may write
     $z=\pi^n\pi_t\zeta+u+\alpha\pi^n$ with $\zeta\in\Omega_{v_0}$
     and $\alpha\in\mathbb{F}_q$. Suppose $u\neq 0$; then
     $\ord u<n$. Since $\ord(\pi^n\pi_t\zeta)=n+t\geqslant n$ and
     $\ord(\alpha\pi^n)\geqslant n$, while $\ord u\leqslant n-1$,
     we get $\ord u=\ord z\geqslant -m$.
    
      \noindent  There are only finitely many vertices $v$ in $X(\mathcal{T})$ represented by a matrix $\begin{pmatrix} \pi^{n} & u \\ 0 & 1 \end{pmatrix}$ with $-m\leqslant n\leqslant m$ and ($-m\leqslant\ord u<n$ or $u=0)$. Since $p$ lies on the closed edge joining such a vertex to one of its finitely many neighbors, $\lambda(U_m)$ is contained in the geometric realization of the finite subtree spanned by these vertices and their neighbors.
\end{proof}

We are now ready to prove Theorem \ref{thmConvergence}.

\begin{proof}[Proof of Theorem \ref{thmConvergence}]
    For $[\gamma]\in H\backslash\Gamma$, define 
    \begin{equation*}
        f_\gamma(z)\coloneqq (\det \gamma)^l\prod_{i=1}^r u(z;\gamma^{-1}a_i,\gamma^{-1}b_i)^{k_i}
    \end{equation*}
    so that $f(z)=\sum_{[\gamma]\in H\backslash\Gamma}f_\gamma(z)$. 
    
    \noindent We first prove the forward implication: assume that \textnormal{\textbf{(P)}} is satisfied. Let $M\geqslant 0$; we want to show $\ord f_\gamma(z)> M$ uniformly for $z\in U_m$, for all but finitely many $[\gamma]\in H\backslash\Gamma$.
     Let $p\coloneqq \lambda(z)$. By definition of the valuation of a $k$-form and since $\det \gamma\in\mathcal{O}_\infty^\times$, we have $\Phi_\gamma(p)=\ord_z(\gamma^*s)=\ord f_\gamma(z)-k\log_q|z|_i$. Therefore
    \begin{equation}\label{equationOrdfgamma}
        \ord f_\gamma(z)=\Phi_\gamma(p)+k\log_q|z|_i.
    \end{equation}
     For $z\in U_m$, we have $\log_q|z|_i\geqslant -m$. The point $p\in\lambda(U_m)$ lives in the geometric realization of a finite subtree $\mathcal{T}_m$ (Lemma \ref{lemmaImageUm}). By Lemma  \ref{lemmaPhiEdge}, for all $p\in\lambda(U_m)$, we have $\Phi_\gamma(p)\geqslant \min_{v\in X(\mathcal{T}_m)} \Phi_\gamma(v)$. Thus
    \begin{equation*}
        \ord f_\gamma(z)\geqslant \min_{v\in X(\mathcal{T}_m)} \Phi_\gamma(v)-km
    \end{equation*}
    for all $z\in U_m$. For $v\in X(\mathcal{T}_m)$, consider the set
    \begin{equation*}
        \{[\gamma]\in H\backslash\Gamma: \Phi_\gamma(v)\leqslant M+km\}.
    \end{equation*}
    Any $[\gamma]$ in this set satisfies $d(\gamma v,[a_i\to b_i])\leqslant \Phi_\gamma(v)\leqslant M+km$ for all $1\leqslant i\leqslant r$; thus this set is contained in $S_{M+km}(v)$, which is finite since \textnormal{\textbf{(P)}} holds. Since $X(\mathcal{T}_m)$ is finite, we have
    \begin{equation*}
        \ord f_\gamma(z)>  (M+km)-km=M
    \end{equation*}
    for all $z\in U_m$, for all but finitely many $ [\gamma]\in H\backslash\Gamma$.

    \noindent Conversely, suppose that $f$ converges locally uniformly on $\Omega$; then $\sup_{z\in U_m}|f_\gamma(z)|$ goes to $0$ as $[\gamma]$ ranges over $H\backslash\Gamma$.
    Since $|f_\gamma(z)| \to 0$ for $z \in \Omega_{v_0} \subseteq U_m$, we have $\Phi_\gamma(v_0) = \ord f_\gamma(z) \to \infty$, so $\{[\gamma]\in H\backslash\Gamma:\Phi_\gamma(v_0)\leqslant N\}$ is finite for all $N\geqslant 0$. 
    As $S_N(v_0)\subseteq \{[\gamma]\in H\backslash\Gamma:\Phi_\gamma(v_0)\leqslant kN\}$, condition \textnormal{\textbf{(P)}} for $v=v_0$ follows.

    \noindent It remains to show the last statement. Since $\Phi_\gamma(p)\geqslant 0$, \eqref{equationOrdfgamma} gives $\ord f_\gamma(z)\geqslant k\log_q|z|_i$, i.e., 
    $|f_\gamma(z)|\leqslant |z|_i^{-k}$. Thus $|f(z)|\leqslant |z|_i^{-k}$ by the ultrametric inequality.
    \end{proof}

\subsection{When is condition \textnormal{\textbf{(P)}} satisfied?}

Recall that a matrix $\delta\in \GL_2(F_\infty)$ is \textit{hyperbolic} if its eigenvalues $\lambda_1$ and $\lambda_2$ satisfy $|\lambda_1|\neq |\lambda_2|$. 

\begin{lemma}\label{lemmaHyperbolic}
    Let $\delta\in\GL_2(F_\infty)$ be hyperbolic with eigenvalues
    $\lambda_1,\lambda_2$. Then $\lambda_1,\lambda_2\in F_\infty$,
    the element $\delta$ has exactly two fixed points $a\neq b$ on
    $\mathbb{P}^1(F_\infty)$, and $\delta$ translates the geodesic
    $[a\to b]$ by
    $\ell(\delta)=|\ord\lambda_1-\ord\lambda_2|\geqslant 1$.
    Moreover, \[d(v,\delta v)=\ell(\delta)+2\,d(v,[a\to b])\] for
    every $v\in X(\mathcal{T})$; in particular, $\delta$ fixes
    no vertex.
\end{lemma}
\begin{proof}
    Since $|\lambda_1|\neq|\lambda_2|$, the Newton polygon of the
    characteristic polynomial of $\delta$ has two distinct slopes,
    so $\lambda_1,\lambda_2\in F_\infty$ and
    $\lambda_1\neq\lambda_2$. Thus
    $\delta=g\,\mathrm{diag}(\lambda_1,\lambda_2)\,g^{-1}$ for
    some $g\in\GL_2(F_\infty)$, the fixed points of $\delta$ on
    $\mathbb{P}^1(F_\infty)$ are $a=g\infty$ and $b=g0$, and $g$
    maps $[\infty\to 0]$ to $[a\to b]$. We may therefore assume
    $\delta=\mathrm{diag}(\lambda_1,\lambda_2)$, which maps $v_n$
    to $v_{n+\ell(\delta)}$ (after possibly interchanging
    $\lambda_1,\lambda_2$), a translation of $[\infty\to 0]$ by
    $\ell(\delta)$. For arbitrary $v\in X(\mathcal{T})$, let $x$ be the vertex of $[\infty\to 0]$
    nearest to $v$. The geodesic from $v$ to $\delta v$ passes
    through $x$ and $\delta x$, whence
    $d(v,\delta v)=d(v,x)+\ell(\delta)+d(\delta x,\delta v)
    =\ell(\delta)+2\,d(v,[\infty\to 0])$.
\end{proof}

\begin{theorem}\label{theorem(P)}
    Condition \textnormal{\textbf{(P)}} is satisfied in the following cases:
    \begin{itemize}
    \item[(i)] $r=1$, $H=\langle\delta\rangle$ with $\delta\in \Gamma$ hyperbolic,
    $\{a,b\}=\{\text{fixed points of $\delta$ in $\mathbb{P}^1(F_\infty)$}\}$,
    \item[(ii)] $r\geqslant 2$, $H=\{I\}$, $\bigcap_{i=1}^r \{a_i,b_i\}=\emptyset$. 
    \end{itemize}
Moreover if $\Gamma$ is commensurable with $G\coloneqq \GL_2(A)$, then \textnormal{\textbf{(P)}} is also satisfied in the following cases:
    \begin{itemize}
    \item[(iii)] $r=1$, $H=\{I\}$, $a,b\in \mathbb{P}^1(F)$,
    \item[(iv)] $r\geqslant 2$, $H=\{I\}$, $\bigcap_{i=1}^r \{a_i,b_i\}=\{c\}$ with $c\in \mathbb{P}^1(F)$.
    \end{itemize}
\end{theorem}

\begin{remark}
    If $r\geqslant 2$, the set $\bigcap_{i=1}^r \{a_i,b_i\}$ contains at most one element. It is empty if and only if $\bigcap_{i=1}^r [a_i\to b_i]$ contains only finitely many vertices, and consists of a single element $c\in\mathbb{P}^1(F_\infty)$ if and only if $\bigcap_{i=1}^r [a_i\to b_i]$ is a half-line toward $c$.
\end{remark}

\begin{remark}
    Assume that $Z(\Gamma)$ consists of the scalar matrices in $\Gamma$. Since $Z(\Gamma)$ acts trivially on $\mathbb{P}^1(F_\infty)$, it fixes each of the parameters $a_1,b_1,\ldots,a_r,b_r$. Therefore one can take $H=Z(\Gamma)$ whenever $(\det h)^l=1$ for all $h\in Z(\Gamma)$; this holds in particular if $l=0$. If moreover $Z(\Gamma)$ is finite, then the map $\Gamma\to Z(\Gamma)\backslash\Gamma$ induces a surjection $S_N^{\Gamma}(v)\to S_N^{Z(\Gamma)\backslash\Gamma}(v)$ with fibers of size $\#Z(\Gamma)$; it follows that \textnormal{\textbf{(P)}} holds for $H=\{I\}$ if and only if it holds for $H=Z(\Gamma)$.
\end{remark}

\noindent We now check that $\textnormal{\textbf{(P)}}$ is satisfied in cases (i)-(iv).

    \subsubsection{Verification of $\textnormal{\textbf{(P)}}$ in case \textnormal{(i)}}

    Suppose $r=1$, $H=\langle\delta\rangle$ with $\delta\in \Gamma$ hyperbolic,
    $\{a,b\}=\{\text{fixed points of $\delta$ in $\mathbb{P}^1(F_\infty)$}\}$. Fix $v\in X(\mathcal{T})$ and $N\geqslant 0$, and let
\begin{equation*}
    T_N\coloneqq \{w\in X(\mathcal{T}): d(w,[a\to b])\leqslant N\}.
\end{equation*}
We want to show that the set 
\begin{equation*}
    S_N(v)=\{[\gamma]\in \langle\delta\rangle\backslash\Gamma: \gamma v\in T_N\}
\end{equation*}
is finite. By Lemma \ref{lemmaHyperbolic}, $\delta$ translates the geodesic $[a\to b]$ by $\ell(\delta)\geqslant 1$. Hence $\langle\delta\rangle$ preserves $[a\to b]$ and acts on $T_N$. Moreover, $\langle\delta\rangle\backslash T_N$ is finite: the vertices of $[a\to b]$ fall into $\ell(\delta)$ orbits under $\langle\delta\rangle$, and every $w\in T_N$ lies within distance $N$ of $[a\to b]$. We now show that the fibers of the map $\psi\colon S_N(v)\to \langle\delta\rangle\backslash T_N$ that maps $[\gamma]$ to $[\gamma v]$ are finite.

    \begin{lemma}\label{lemmacoolcool}
        Let $[w]\in\langle\delta\rangle\backslash T_N$. Assume that the fiber $\psi^{-1}([w])$ is nonempty; then the map $\psi^{-1}([w]) \to \{\gamma_0\in\Gamma:\gamma_0 v = w\}$ that maps $[\gamma]$ to the unique representative $\gamma_0$ of $[\gamma]$ sending $v$ to $w$ is a bijection.
    \end{lemma}
    \begin{proof}
        We only show the existence and uniqueness of the representative; both injectivity and surjectivity follow easily. If $[\gamma]\in\psi^{-1}([w])$,
then $\gamma v = \delta^k w$ for some $k\in\mathbb{Z}$. The element $\gamma_0 \coloneqq  \delta^{-k}\gamma$ is a representative of $[\gamma]$
satisfying $\gamma_0 v = \delta^{-k}\gamma v = w$. Suppose $\gamma_0, \gamma_0'$
are two representatives of the same coset $[\gamma]$ satisfying
$\gamma_0 v = \gamma_0' v = w$. Then $\gamma_0' = \delta^j\gamma_0$ for some $j\in\mathbb{Z}$, and
\[
w = \gamma_0' v = \delta^j\gamma_0 v = \delta^j w.
\]
Thus $\delta^j$ fixes $w$. If $j\neq 0$, then $\delta^j$ is hyperbolic (its eigenvalues $\lambda_1^j,\lambda_2^j$ have distinct absolute values), so it fixes no vertex of $\mathcal{T}$ by Lemma \ref{lemmaHyperbolic}. Thus $j = 0$, giving $\gamma_0' = \gamma_0$.
    \end{proof}

\noindent Lemma \ref{lemmacoolcool} shows that the nonempty fibers of the map $\psi\colon S_N(v)\to \langle\delta\rangle\backslash T_N$ have size $\#\Gamma_v$. Since $\langle\delta\rangle\backslash T_N$ is finite, $S_N(v)$ is finite as well. 

\subsubsection{Verification of $\textnormal{\textbf{(P)}}$ in case \textnormal{(ii)}}

Suppose $r\geqslant 2$, $H=\{I\}$, $\bigcap_{i=1}^r \{a_i,b_i\}=\emptyset$. Fix $v\in X(\mathcal{T})$ and $N\geqslant 0$, and let
\begin{equation*}
    T_N\coloneqq \{w\in X(\mathcal{T}): d(w,[a_i\to b_i])\leqslant N \text{ for all $1\leqslant i\leqslant r$}\}.
\end{equation*}
We want to show that the set 
\begin{equation*}
    S_N(v)=\{\gamma\in \Gamma: \gamma v\in T_N\}
\end{equation*}
is finite. If $T_N$ were infinite, it would contain a half-line going to some $c\in\mathbb{P}^1(F_\infty)$. For each $i$, every vertex on this half-line is within distance $N$ of $[a_i\to b_i]$, which implies $c\in\{a_i,b_i\}$. So $c\in\bigcap_{i=1}^r \{a_i,b_i\}=\emptyset$, contradiction. Thus $T_N$ is finite. For $w\in T_N$, the set $\{\gamma\in \Gamma: \gamma v=w\}$ is either empty or a coset of $\Gamma_v$, depending on whether $w$ is in the $\Gamma$-orbit of $v$ or not. Thus $\#S_N(v)\leqslant \#T_N\cdot \#\Gamma_v$.

\subsubsection{Verification of $\textnormal{\textbf{(P)}}$ in case \textnormal{(iii)}}\label{subsubsection(P)(iii)}

Suppose $r=1$, $\Gamma$ commensurable with $G$, $H=\{I\}$, $a,b\in \mathbb{P}^1(F)$. Fix $N\geqslant 0$. We show that the set
\begin{equation*}
    S_N^\Gamma(v)=\{\gamma\in \Gamma: d(\gamma v,[a\to b])\leqslant N\}
\end{equation*}
is finite for all $v\in X(\mathcal{T})$. 

\textit{Reduction to the case $\Gamma=G$.} Suppose that $S_N^G(v)$ is finite for all $v\in X(\mathcal{T})$. The group $\Gamma_0\coloneqq \Gamma\cap G$ has finite index in $\Gamma$, so $\Gamma=\coprod_{i=1}^m\Gamma_0\sigma_i$ for some $\sigma_1,\ldots,\sigma_m\in \Gamma$. Then $S_N^\Gamma(v)=\coprod_{i=1}^mS_N^{\Gamma_0}(\sigma_iv)$ is finite since $S_N^{\Gamma_0}(\sigma_iv)\subseteq S_N^{G}(\sigma_iv)$.

\textit{Proof for $\Gamma=G$.} Acting by an element of $G$ if necessary, we may assume $a=\infty$, $b=0$. The canonical projection $\mathcal{T}\to G\backslash\mathcal{T}$ identifies the quotient $G\backslash\mathcal{T}$ with the half-line in $\mathcal{T}$ with vertices $\{v_{n}\}_{n\geqslant 0}$ \cite[2.4.1]{Serre1980}. We define the \textit{type} of $v\in X(\mathcal{T})$ as the unique $n\geqslant 0$ such that $v$ lies in the $G$-orbit of $v_n$. It satisfies $\type(v_n)=|n|$. 
If $v$ and $w$ are adjacent, then their images in $G\backslash\mathcal{T}$ are adjacent as well (since $G$ acts without inversion), and $|\type(v)-\type(w)|=1$. It follows that $|\type(v)-\type(w)|\leqslant d(v,w)$ for any $v,w\in X(\mathcal{T})$. One checks that any vertex of type $n\geqslant 1$ has $q$ neighbors of type $n-1$, and a single neighbor of type $n+1$. Moreover, the $q+1$ neighbors of a type $0$ vertex have type $1$ \cite[p.112]{Serre1980}. 

\noindent Fix $v\in X(\mathcal{T})$, and let $n\coloneqq \type(v)$. Let 
\begin{equation*}
    T_N^n\coloneqq \{w\in X(\mathcal{T}): \type(w)=n,\, d(w,[\infty\to 0])\leqslant N \}.
\end{equation*}
For $\gamma\in S_N^G(v)$, we have $\type(\gamma v)=\type(v)=n$; thus $\gamma\mapsto \gamma v$ is a well-defined map $S_N^G(v)\to T_N^n$. One shows as in case (ii) that the fibers of this map are finite: it thus suffices to show that $T_N^n$ is finite. Let $w\in T_N^n$; then $d(w,v_m)\leqslant N$ for some $m\in\mathbb{Z}$. In particular, $|n-|m||\leqslant N$, i.e., $|m|\in [n-N,n+N]$. Each vertex $w\in T_N^n$ is at distance at most $N$ from the finite set $\{v_m\}_{n-N\leqslant m\leqslant n+N}$; thus $T_N^n$ is finite.

\subsubsection{Verification of $\textnormal{\textbf{(P)}}$ in case \textnormal{(iv)}} Suppose $r\geqslant 2$, $\Gamma$ commensurable with $G$, $H=\{I\}$, $\bigcap_{i=1}^r \{a_i,b_i\}=\{c\}$ with $c\in \mathbb{P}^1(F)$. By commensurability, we reduce to $\Gamma=G$ as in case (iii). Acting by an element of $G$ if necessary, we may assume $c=\infty$. Swapping $a_i$ and $b_i$ if necessary, we can furthermore assume $a_i=\infty$ for all $1\leqslant i\leqslant r$.

\noindent Fix $N\geqslant 0$ and let $v\in X(\mathcal{T})$. Since $r\geqslant 2$ and the $b_i$ are pairwise distinct, the set $X\coloneqq \bigcap_{i=1}^r[\infty\to b_i]$ is a half-line toward $\infty$. By Remark \ref{remarkSNVIntersection}, we need to show that the set
\begin{equation*}
    S_N(v)=\{\gamma\in G: d(\gamma v,X)\leqslant N\}
\end{equation*}
is finite. Let $n\coloneqq \type(v)$ and
\begin{equation*}
    T_N^n\coloneqq \{w\in X(\mathcal{T}): \type(w)=n,\, d(w,X)\leqslant N\}.
\end{equation*}
As in case (iii), the fibers of the map $S_N(v)\to T_N^n$ are finite, so it suffices to show that $T_N^n$ is finite.
Let $w\in T_N^n$; there exists a vertex $x$ in $X$ such that $d(w,x)\leqslant N$. In particular, $|n-\type(x)|\leqslant N$, i.e., $\type(x)\in [n-N,n+N]$. The type along $X$ is eventually strictly increasing, so the set $\{x \in X : \type(x) \in [n-N, n+N]\}$ is finite. Any $w\in T_N^n$ is at distance at most $N$ from a finite set, so $T_N^n$ is finite.

\subsection{Valuation of Poincaré series}

Let $P(z)=\sum_{[\gamma]\in H\backslash\Gamma}(\det\gamma)^l \cdot\gamma^*s(z)=f(z)\,(\mathrm{d}z)^k$ be a Poincaré series with $r\geqslant 2$ satisfying \textnormal{\textbf{(P)}}. We now give sufficient conditions on the intersection of the geodesics $[a_i\to b_i]$ for $P$ (and thus $f$) not to vanish identically.

\begin{theorem}\label{theoremIntersectionVertex}
     Assume that the geodesics $[a_i\to b_i]$ for $1\leqslant i\leqslant r$ intersect at a single vertex $v\in X(\mathcal{T})$ with $\Gamma_v\subseteq H$. Then $\ord_z(P)=0$ if $z$ lies on the $\Gamma$-orbit of $\Omega_{v}$, and $\ord_z(P)\geqslant t_p>0$ otherwise, where $p\coloneqq \lambda(z)\in\mathcal{T}(\mathbb{Q})$ and $t_p$ was defined in Corollary \ref{corollaryValuationS}. In particular, $P$ is not identically zero.
\end{theorem}

\begin{proof}
    First assume that $z$ does not lie in the $\Gamma$-orbit of $\Omega_{v}$. Then for all $\gamma\in\Gamma$, we have $\gamma p\neq v$. By Corollary \ref{corollaryValuationS}, one has $\ord_z(\gamma^*s)\geqslant t_p$ for all $[\gamma]\in H\backslash\Gamma$. As $\det(\Gamma)\subset\mathcal{O}_\infty^\times$, it implies $\ord_z(P)\geqslant t_p$.

    \noindent Now assume that $z$ lies in the $\Gamma$-orbit of $\Omega_{v}$. Since $P$ satisfies $\gamma^*P=(\det\gamma)^{-l}P$, we can assume $z\in\Omega_{v}$ without loss of generality. Then $\ord_z(\gamma^*s)=0$ if and only if $\gamma$ fixes $v$, i.e. $[\gamma]=[I]$ since $\Gamma_v\subseteq H$. We can therefore write
    \begin{equation}\label{equationValuationRPS}
        P(z)=s(z)+\sum_{[\gamma]\neq [I]}(\det\gamma)^l \cdot\gamma^*s(z)
    \end{equation}
    with $\ord_z(s)=0$ and $\ord_z(\gamma^*s)\geqslant t_p>0$ for $[\gamma]\neq [I]$. Thus $\ord_z(P)=0$.
\end{proof}

\begin{theorem}\label{theoremIntersectionEdge}
     Assume that the intersection of the geodesics $[a_i\to b_i]$ for $1\leqslant i\leqslant r$ consists of an open edge $e$ and its two endpoints $v,w\in X(\mathcal{T})$. Assume further that the stabilizer of one of the endpoints (say $v$) in $\Gamma$ is contained in $H$. Then $\ord_z(P)=0$
     if $z$ lies on the $\Gamma$-orbit of $\Omega_{v}\cup\Omega_{e}$, and $\ord_z(P)\geqslant t_p>0$ if $z$ does not lie on the $\Gamma$-orbit of $\Omega_{v}\cup\Omega_{e}\cup\Omega_{w}$, where $p\coloneqq \lambda(z)\in\mathcal{T}(\mathbb{Q})$ 
     and $t_p$ was defined in Corollary \ref{corollaryValuationS}. In particular, $P$ is not identically zero. If moreover $\Gamma_w\subseteq H$, then also $\ord_z(P)=0$ on the $\Gamma$-orbit of $\Omega_{w}$.
\end{theorem}

\begin{proof} The proof is similar to that of Theorem \ref{theoremIntersectionVertex}. If $z$ does not lie in the $\Gamma$-orbit of $\Omega_{v}\cup\Omega_{e}\cup\Omega_{w}$, then $\gamma p$ does not lie in the intersection of the geodesics. By Corollary \ref{corollaryValuationS}, one has $\ord_z(\gamma^*s)\geqslant t_p$ for all $[\gamma]\in H\backslash\Gamma$, and thus $\ord_z(P)\geqslant t_p$.

    \noindent Now assume that $z$ lies in the $\Gamma$-orbit of $\Omega_{v}\cup\Omega_{e}$. Since $P$ satisfies $\gamma^*P=(\det\gamma)^{-l}P$, we can assume $z\in\Omega_{v}\cup\Omega_{e}$ without loss of generality. If $z\in\Omega_{v}$, then $\ord_z(\gamma^*s)=0$ if and only if $\gamma$ fixes $v$ or $\gamma v=w$. The second case does not occur, since the distance from $\gamma v$ to $v$ is even. If $z\in\Omega_{e}$, then $\ord_z(\gamma^*s)=0$ if and only if $\gamma$ fixes both $v$ and $w$ (since $\Gamma$ acts without inversion). Since $\Gamma_v\subseteq H$, we have $\ord_z(\gamma^*s)=0$ if and only if $[\gamma]=[I]$, and one can conclude as in the proof of Theorem \ref{theoremIntersectionVertex}. The case $z\in\Omega_w$ with $\Gamma_w\subseteq H$ is analogous.
\end{proof}

\begin{remark}
    A similar argument to the proof of Theorem \ref{theoremIntersectionVertex} shows that if the intersection of the geodesics consists of two edges sharing a common endpoint $v$ with $\Gamma_v\subseteq H$, then $\ord_z(P)=0$ on the $\Gamma$-orbit of $\Omega_{v}$. In particular, $P$ is not identically zero.
\end{remark}

\begin{remark}
    Poincaré series can be nonzero even when the geodesics do not intersect. Take $r=2$, and suppose that $[a_1\to b_1]$ and $[a_2\to b_2]$ are at distance $1$ (i.e. they do not intersect, and there exist adjacent vertices $v_1,v_2\in X(\mathcal{T})$ with $v_i\in [a_i\to b_i]$). A similar argument to the proof of Theorem \ref{theoremIntersectionVertex} shows that if $\Gamma_{v_1}\subseteq H$ and $2k_1>k_2$, then $\ord_z(P)=k_2$ on the $\Gamma$-orbit of $\Omega_{v_1}$. In particular, $P$ is not identically zero.
\end{remark}

\section{Poincaré series as modular forms}\label{sectionModularForms}

In this section, $\Gamma$ is either commensurable with $G=\GL_2(A)$, or a discrete cocompact subgroup of $\GL_2(F_\infty)$. In the latter case, we also assume that $Z(\Gamma)$ consists of the scalar matrices in $\Gamma$, and that $\det(\Gamma)\subseteq \mathbb{F}_q^\times$ (these two conditions are always satisfied if $\Gamma$ is commensurable with $G$). Let $d(\Gamma)\coloneqq \#\det(\Gamma)$.

\noindent If $\Gamma$ is commensurable with $G$, the quotient
$\Gamma\backslash\mathcal{T}$ is the union of a finite graph
$(\Gamma\backslash\mathcal{T})^0$ and finitely many half-lines
(for finite-index subgroups of $G$ this is \cite[p.106]{Serre1980}; the general case follows by commensurability). Each half-line corresponds to a $\Gamma$-orbit in $\mathbb{P}^1(F)$; we call this
orbit a \emph{cusp} of $\Gamma$. For instance, the $\Gamma$-orbit of $\infty$ is the cusp of $\Gamma$ corresponding to the image in
$\Gamma\backslash\mathcal{T}$ of the half-line
$\{v_n\}_{n \geqslant N}$ for $N$ sufficiently large. If $\Gamma$ is
cocompact, the quotient $\Gamma\backslash\mathcal{T}$ is a finite
graph, and $\Gamma$ has no cusps. The genus of the compactified modular curve $X_\Gamma$ equals the first Betti number of $\Gamma\backslash\mathcal{T}$ (see \cite[(2.6)]{GekelerReversat1996} if $\Gamma$ is a finite-index subgroup of $G$, and \cite[Proposition 3.2]{Kurihara1979} if $\Gamma$ is cocompact).

\noindent For $k\geqslant 0$, $l\in \mathbb{Z}/d(\Gamma)\mathbb{Z}$, $\gamma\in \GL_2(F_\infty)$, and $f\colon\Omega\to \mathbb{C}_\infty$, define
\begin{equation*}
    (f\slashbar{k}{l}\,\gamma)(z)\coloneqq \frac{(\det\gamma)^l}{j_\gamma(z)^k}f(\gamma z).
\end{equation*}

\begin{definition}\label{defDMF}
    A \textit{Drinfeld modular form of weight $k\geqslant 0$ and type $l\in \mathbb{Z}/d(\Gamma)\mathbb{Z}$ for the group} $\Gamma$ is a rigid holomorphic function $f\colon\Omega\to \mathbb{C}_\infty$ that satisfies $f\slashbar{k}{l}\,\gamma=f$ for all $\gamma\in\Gamma$, and is holomorphic at all cusps. We write $M_{k,l}(\Gamma)$ for the $\mathbb{C}_\infty$-vector space of modular forms of weight $k$ and type $l$ for $\Gamma$.
\end{definition}

We now make precise the condition of holomorphicity at the cusps. If $\Gamma$ is cocompact, this condition is vacuous. Suppose that $\Gamma$ is commensurable with $G$. The subgroup of upper unipotent matrices in $\Gamma$ has the form $\left\{\begin{pmatrix} 1 & b \\ 0 & 1 \end{pmatrix} : b \in \mathfrak{b}\right\}$, where $\mathfrak{b}$ is an additive subgroup of $F$ commensurable with $A$. Define
\[
e_{\mathfrak{b}}(z) \coloneqq  z \sideset{}{'}\prod_{b \in \mathfrak{b}} \left(1 - \frac{z}{b}\right)
\]
and
\[
t(z)=t^\Gamma(z) \coloneqq \frac{1}{e_{\mathfrak{b}}(z)}.
\]
As $e_{\mathfrak{b}}'(z)=1$, we have 
\[
t'(z)=-\frac{e_{\mathfrak{b}}'(z)}{e_{\mathfrak{b}}(z)^2}=-t(z)^2
\]
and 
\[
t(z)=\frac{e_{\mathfrak{b}}'(z)}{e_{\mathfrak{b}}(z)}=\sum_{b \in \mathfrak{b}}\frac{1}{z+b}.
\]
We say that $f$ is \textit{holomorphic at} $\infty$ (with respect to $\Gamma$) if $f$ has a $t^\Gamma$-expansion
\begin{equation}\label{equationTExpansion}
    f(z)=\sum_{n\geqslant 0}c_n(f)\,t^\Gamma(z)^{n}
\end{equation}
with $c_n(f)\in\mathbb{C}_\infty$ and a positive radius of convergence as a power series in $t^\Gamma$.
Let $\mathfrak{c}=\nu \infty$ be a cusp of $\Gamma$ with $\nu\in \GL_2(F)$; then $f_\nu\coloneqq f\slashbar{k}{l}\,\nu$ satisfies $f_\nu\slashbar{k}{l}\,\gamma'=f_\nu$ for all $\gamma'\in \Gamma_\nu\coloneqq \nu^{-1}\Gamma\nu$. 
We say that $f$ is \textit{holomorphic at} $\mathfrak{c}$ if $f_\nu$ is holomorphic at $\infty$ (with respect to $\Gamma_\nu$).

Suppose that $\Gamma$ is commensurable with $G$. We say that a modular form $f$ \textit{vanishes to order $\geqslant n$ at the cusp} $\mathfrak{c}=\nu \infty$ if $c_0(f_\nu)=\ldots=c_{n-1}(f_\nu)=0$ in the  $t^{\Gamma_\nu}$-expansion of $f_\nu$. The order of vanishing does not depend on the representative of $\mathfrak{c}$ in $\mathbb{P}^1(F)$, nor on the matrix $\nu\in\GL_2(F)$ \cite[(2.7.7)]{GekelerReversat1996}. We write $M_{k,l}^n(\Gamma)$ for the subspace of $M_{k,l}(\Gamma)$ consisting of forms vanishing to order $\geqslant n$ at all cusps. A \textit{cusp form} is a form that vanishes to order $\geqslant 1$ at all cusps. If $\det(\Gamma)$ is trivial, then every Drinfeld modular form has type $0$, and we therefore omit the subscript $l$. If $\Gamma$ is cocompact, we set $M_{k,l}^n(\Gamma)=M_{k,l}(\Gamma)$ for all $n$.

Taking $\gamma\in Z(\Gamma)$ in the condition $f\slashbar{k}{l}\,\gamma=f$, one checks easily that $M_{k,l}(\Gamma)$ is zero unless $k\equiv 2l\mod \#Z(\Gamma)$. We now compute the dimension of $M_{2k,l}^n(\Gamma)$ when this condition is satisfied and $\widetilde{\Gamma}=\Gamma/Z(\Gamma)$ has no torsion of order prime to $p$.

\begin{proposition}\label{propDim}
    Assume that $\widetilde{\Gamma}$ has
    no prime-to-$p$ torsion. Let $g$ be the genus of
    $X_\Gamma$ and $c$ its number of cusps (so $c=0$ if $\Gamma$ is
    cocompact). Let $k\geqslant 1$, $n\geqslant 0$, and let
    $l$ be such that $2(k-l)\equiv 0\mod \#Z(\Gamma)$. If $(k-1)(2g-2)+(2k-n)c>0$, then
    \begin{equation*}
        \dim M_{2k,l}^n(\Gamma)=(2k-1)(g-1)+(2k-n)c .
    \end{equation*}
\end{proposition}
\begin{proof}
    Write $\chi(\gamma)\coloneqq(\det\gamma)^{k-l}$. Since
    $\mathrm{d}(\gamma z)=(\det\gamma)j_\gamma(z)^{-2}\mathrm{d}z$, the
    transformation law $f\slashbar{2k}{l}\gamma=f$ is equivalent to
    $\gamma^*\omega_f=\chi(\gamma)\,\omega_f$ for the $k$-form
    $\omega_f\coloneqq f(z)(\mathrm{d}z)^k$. By the assumption on $l$, the
    character $\chi$ is trivial on $Z(\Gamma)$, so it factors through
    $\widetilde{\Gamma}$; moreover $\widetilde{\Gamma}$ acts freely on
    $\Omega$, since a non-scalar $\gamma\in\Gamma$ fixing a point of
    $\Omega$ has finite order prime to $p$.

    \noindent Let $\mathcal{L}_\chi$ be the sheaf on
    $\Gamma\backslash\Omega$ whose sections over an admissible open $V$ are
    the holomorphic functions $u$ on the preimage of $V$ satisfying
    $u(\gamma z)=\chi(\gamma)u(z)$ for all $\gamma\in\Gamma$. It is
    invertible: as $\widetilde{\Gamma}$ acts freely, every point has an
    admissible neighbourhood $V$ whose preimage is a disjoint union
    $\coprod_{[\gamma]}\gamma V_0$ with $V_0\xrightarrow{\sim}V$, and the
    function equal to $\chi(\gamma)$ on $\gamma V_0$ is a nowhere vanishing
    section over $V$. Dividing by such a section identifies the $k$-forms
    $\omega$ on $\Omega$ with $\gamma^*\omega=\chi(\gamma)\,\omega$ with the
    sections of $\Omega_{X_\Gamma}^{\otimes k}\otimes\mathcal{L}_\chi$ over
    $\Gamma\backslash\Omega$. 
    
    \noindent Suppose $c>0$, and let $\mathfrak{c}=\nu\infty$ be a cusp. The stabilizer of $\infty$ in $\Gamma_\nu$ is generated by $Z(\Gamma)$ and its unipotent subgroup, since the quotient by these has order prime to $p$; as $\chi$ is trivial on $Z(\Gamma)$ and on unipotent matrices, the $\chi$-equivariant functions near $\mathfrak{c}$ are the power series in $t=t^{\Gamma_\nu}$. We extend $\mathcal{L}_\chi$ across $\mathfrak{c}$ by taking as sections near $\mathfrak{c}$ the $\chi$-equivariant functions whose expansion in $t$ is a power series.
    Finally, the relation $\mathrm{d}z=-t^{-2}\mathrm{d}t$ shows that $f$
    vanishes to order $\geqslant n$ at $\mathfrak{c}$ if and only if
    $\omega_f$ has a pole of order at most $2k-n$ there. Altogether,
    \begin{equation*}
        M_{2k,l}^n(\Gamma)\cong
        H^0\big(X_\Gamma,\ \Omega_{X_\Gamma}^{\otimes k}
        \otimes\mathcal{L}_\chi\,((2k-n)\Sigma)\big),
    \end{equation*}
    where $\Sigma\coloneqq\sum_{\mathfrak{c}}(\mathfrak{c})$ is the divisor
    of cusps.

    \noindent Since $\det(\Gamma)$ is finite, $\chi$ has finite order $m$,
    so $\mathcal{L}_\chi^{\otimes m}\cong\mathcal{O}_{X_\Gamma}$ and
    $\deg\mathcal{L}_\chi=0$. The line bundle above therefore has degree
    $k(2g-2)+(2k-n)c$, which exceeds $2g-2$ exactly when
    $(k-1)(2g-2)+(2k-n)c>0$. By the Riemann-Roch theorem,
    \begin{align*}
        \dim M_{2k,l}^n(\Gamma)&=k(2g-2)+(2k-n)c+1-g\\
        &=(2k-1)(g-1)+(2k-n)c
    \end{align*}
    as claimed.
\end{proof}

\subsection{Vanishing at cusps}\label{subsectionVanishing}

Let
\begin{equation}\label{equationPoincarSeriesf}
        f(z)=\sum_{[\gamma]\in H\backslash\Gamma}(\det \gamma)^l\prod_{i=1}^r u(z;\gamma^{-1}a_i,\gamma^{-1}b_i)^{k_i}
    \end{equation}
be a Poincaré series that satisfies \textnormal{\textbf{(P)}}. Since $\det(\Gamma)$ is finite, the value of $(\det\gamma)^l$
 depends only on $l \mod d(\Gamma)$. Let
\begin{equation*}
        P(z)=f(z)(\mathrm{d}z)^k=\sum_{[\gamma]\in H\backslash\Gamma}(\det \gamma)^l\prod_{i=1}^r \gamma^*s(z;a_i,b_i)^{k_i}.
\end{equation*}
By \eqref{equationS}, one has $\gamma^*P=(\det\gamma)^{-l}P$ for all $\gamma\in\Gamma$. 

\begin{corollary}\label{coroCuspForm}
    $f$ is a cusp form of weight $2k$ and type $k+l$ for the group $\Gamma$.
\end{corollary}
\begin{proof}
    Since $\mathrm{d}(\gamma z)=(\det\gamma)j_\gamma(z)^{-2}\mathrm{d}z$, the transformation law $\gamma^*P=(\det\gamma)^{-l}P$ translates to $f\slashbar{2k}{k+l}\gamma=f$ for all $\gamma\in\Gamma$. 

    \noindent If $\Gamma$ is cocompact, then $f$ is automatically a cusp form. Suppose that $\Gamma$ is commensurable with $\GL_2(A)$. Let $\mathfrak{c}=\nu\infty$ be a cusp. Since $f$ is holomorphic
    on $\Omega$ and $j_\nu(z)$ is nonvanishing on $\Omega$, the
    function $f_\nu = f\slashbar{2k}{k+l}\nu$ is holomorphic on
    $\Omega$. 
    As $f_\nu$ is periodic under the unipotent subgroup of
    $\Gamma_\nu$, it admits a Laurent expansion $f_\nu(z)=\sum_{n\in\mathbb{Z}}c_n(f_\nu)\,t^{\Gamma_\nu}(z)^{n}$.
    To show that $f_\nu$ vanishes at $\infty$, we check that $f_\nu$ is, up to a multiplicative constant, a Poincaré series for $\Gamma_\nu$. 
    Set $H_\nu\coloneqq \nu^{-1}H\nu$; the map $\gamma'\coloneqq \nu^{-1}\gamma\nu\mapsto\gamma$ induces a bijection $H_\nu\backslash\Gamma_\nu\to H\backslash\Gamma$. As
    $u(\nu z;a,b)=j_\nu(z)^2(\det\nu)^{-1}u(z;\nu^{-1}a,\nu^{-1}b)$, which follows from \eqref{equationS} and $\mathrm{d}(\nu z) = \det(\nu)\,j_\nu(z)^{-2}\,\mathrm{d}z$, we have $f_\nu(z)=(\det\nu)^l\cdot g(z)$ where
        \begin{equation*}
        g(z)\coloneqq \sum_{[\gamma']\in H_\nu\backslash\Gamma_\nu}(\det\gamma')^l\prod_{i=1}^r u(z;(\gamma')^{-1}\nu^{-1}a_i,(\gamma')^{-1}\nu^{-1}b_i)^{k_i}
    \end{equation*}
    is a Poincaré series. One shows easily that condition \textbf{(P)} for $g$ follows from condition \textbf{(P)} for $f$. By Theorem \ref{thmConvergence}, we have $|f_\nu(z)|\leqslant |\det\nu\,|^l\cdot |z|_i^{-k}$ for all $z\in\Omega$, so $f_\nu(z)\to 0$ as $|z|_i\to\infty$. Therefore $c_n(f_\nu)=0$ for $n\leqslant 0$, and $f_\nu$ vanishes at $\infty$.
\end{proof}

\begin{example} Take $\Gamma=\Gamma(\mathfrak{n})$ with $\deg(\mathfrak{n})\geqslant 1$. The geodesics $[\infty\to 0]$ and $[1\to -1]$ intersect at $v_0$, and the stabilizer of $v_0$ in $\Gamma(\mathfrak{n})$ is trivial (see the proof of Lemma \ref{LemmaQuotientGraphGammaN}). Choose $k_1,k_2\geqslant 1$ with $k_1+k_2=k$. By Theorem \ref{theoremIntersectionVertex}, the cusp form associated to
    \begin{equation*}
         P(z)\coloneqq \sum_{\gamma\in\Gamma(\mathfrak{n})}\gamma^* s(z;\infty,0)^{k_1}s(z;1,-1)^{k_2}
    \end{equation*}
        is a nonzero element of $M_{2k}^1(\Gamma(\mathfrak{n}))$.
\end{example}

\begin{remark}
    Assume $k=0\mod q-1$ and $2k<q^2-1$. Then $M_{2k,0}(G)$ is one-dimensional \cite[Proposition 4.3]{Cornelissen1997}, spanned by the Eisenstein series $E_{2k}(z)=\sum_{c,d\in A}'(cz+d)^{-2k}$. In particular, there are no nontrivial cusp forms of weight $2k$ and type $0$ for $G$. Let $k_1,k_2\geqslant 1$ such that $k_1+k_2=k$; then the cusp form associated to
    \begin{align*}
        Q(z)&\coloneqq \sum_{\gamma\in G}\gamma^* s(z;\infty,0)^{k_1}s(z;1,-1)^{k_2}\\
        &=-\sum_{[\gamma]\in Z(G)\backslash G}\gamma^* s(z;\infty,0)^{k_1}s(z;1,-1)^{k_2}
    \end{align*}
    vanishes identically. 
    This example shows that the condition on $\Gamma_v$ in Theorem \ref{theoremIntersectionVertex} cannot be removed: the geodesics $[\infty\to 0]$ and $[1\to -1]$ intersect at $v_0$, whose stabilizer in $G$ is $\GL_2(\mathbb{F}_q)$, which is not contained in $Z(G)=\mathbb{F}_q^\times\cdot I$.
\end{remark}

\section{The space spanned by Poincaré series}\label{SectionSpan}

In this section, we study the space spanned by Poincaré series, first for the principal congruence
subgroups $\Gamma(\mathfrak{n})$ of $G$, then for the cocompact groups
attached to quaternion algebras over $F$ split at $\infty$.

\subsection{The case $\Gamma(\mathfrak{n})$}

Let $k\geqslant 2$, and let $\mathfrak{n}\in A$ be monic with $d\coloneqq \deg(\mathfrak{n})\geqslant 1$. We now construct an explicit linearly independent family of Poincaré series of weight $2k$ for $\Gamma(\mathfrak{n})$. Following Kurihara \cite{Kurihara1994}, we associate linearly independent Poincaré series to vertices and edges of the quotient graph $\Gamma(\mathfrak{n})\backslash\mathcal{T}$.

For the remainder of this subsection, we only consider Poincaré series for the group $\Gamma(\mathfrak{n})$ with $\sum k_i=k$, $r\geqslant 2$, $H$ trivial, and
$\bigcap_{i=1}^r \{a_i,b_i\} =\emptyset$:
\begin{equation}\label{equationPSforGamman}
f(z)=\sum_{\gamma\in\Gamma(\mathfrak{n})}\prod_{i=1}^r u(z;\gamma a_i,\gamma b_i)^{k_i}
\end{equation}
Condition \textnormal{\textbf{(P)}} holds by Theorem \ref{theorem(P)}(ii), and $f\in M_{2k}^1(\Gamma(\mathfrak{n}))$ by Corollary \ref{coroCuspForm}.

\noindent The number of cusps $c=c(X_{\Gamma(\mathfrak{n})})$ and the genus $g=g(X_{\Gamma(\mathfrak{n})})$ are given by \cite[p.89-90]{Gekeler1980}
    \begin{equation}\label{equationCuspGannaN}
        c=\frac{q^{2d}}{q-1}\prod_{\mathfrak{p}|\mathfrak{n}}\left(1-\frac{1}{q^{2\deg(\mathfrak{p})}}\right)
    \end{equation}
    and 
    \begin{align}
    g&=1+\frac{q^{3d}}{q^2-1}\prod_{\mathfrak{p}|\mathfrak{n}}\left(1-\frac{1}{q^{2\deg(\mathfrak{p})}}\right)-c\label{equationGenusGammaN}\\
        &=1+\frac{q^{3d}}{q-1}\left(\frac{1}{q+1}-\frac{1}{q^d}\right)\prod_{\mathfrak{p}|\mathfrak{n}}\left(1-\frac{1}{q^{2\deg(\mathfrak{p})}}\right)\nonumber
    \end{align}
    where the products are over the monic irreducible polynomials $\mathfrak{p}$ dividing $\mathfrak{n}$. Equations \eqref{equationCuspGannaN} and \eqref{equationGenusGammaN} imply that $g=0$ if and only if $d=1$, in which case $c=q+1$.

    \noindent The group $\Gamma(\mathfrak{n})$ has no prime-to-$p$ torsion, and all its elements have determinant $1$. By Proposition \ref{propDim},
    we have 
    \begin{equation}\label{eqdim}
        \dim M_{2k}^n(\Gamma(\mathfrak{n}))=(2k-1)(g-1)+(2k-n)c
    \end{equation}
    as long as $n<2k$ and $d\geqslant 2$. If
    $d=1$, then $g=0$ and $c=q+1$, and the hypothesis
    of Proposition \ref{propDim} reads
    $(2k-n)(q+1)>2(k-1)$.

\begin{theorem}\label{theoremPSspan}
    The Poincaré series
    \eqref{equationPSforGamman} span a subspace of dimension at least $(2k-1)(g-1)+kc$ in $M_{2k}^1(\Gamma(\mathfrak{n}))$.
\end{theorem}

 Our first task is to describe the quotient graph $\Gamma(\mathfrak{n})\backslash\mathcal{T}$ in detail. We say that a vertex $\bar{v}$ (resp. an edge $\bar{e}$) of $\Gamma(\mathfrak{n})\backslash\mathcal{T}$ is \textit{stable} if any lift $v$ of $\bar{v}$ (resp. $e$ of $\bar{e}$) in $\mathcal{T}$ has trivial stabilizer in $\Gamma(\mathfrak{n})$. 

\noindent Since the type of a vertex (defined in \S\ref{subsubsection(P)(iii)}) is constant along $\Gamma(\mathfrak{n})$-orbits, it descends to the quotient: the type of $\bar{v}\in X(\Gamma(\mathfrak{n})\backslash\mathcal{T})$ is the type of any lift $v\in X(\mathcal{T})$. 

 \begin{lemma}\label{LemmaQuotientGraphGammaN}
    The stable vertices in $\Gamma(\mathfrak{n})\backslash\mathcal{T}$
    are the vertices of type $<d$; they are exactly the vertices of
    $(\Gamma(\mathfrak{n})\backslash\mathcal{T})^0$. Moreover, each
    stable vertex has exactly $q+1$ neighbors in
    $\Gamma(\mathfrak{n})\backslash\mathcal{T}$, and the stable edges
    are the edges with an endpoint in
    $(\Gamma(\mathfrak{n})\backslash\mathcal{T})^0$.
    The $c$ half-lines of $\Gamma(\mathfrak{n})\backslash\mathcal{T}$
    are attached as follows:
     \begin{itemize}
         \item if $d=1$, then $(\Gamma(\mathfrak{n})\backslash\mathcal{T})^0$
            consists of a single vertex, to which all $c=q+1$ half-lines
            are attached;
          \item if $d\geqslant 2$, the $c$ half-lines are attached to
            distinct vertices of
            $(\Gamma(\mathfrak{n})\backslash\mathcal{T})^0$.
     \end{itemize}
 \end{lemma}

\noindent For $d\geqslant 2$, Lemma \ref{LemmaQuotientGraphGammaN} follows from the analysis of the covering $\Gamma(\mathfrak{n})\backslash\mathcal{T}\to G\backslash\mathcal{T}$
in \cite[(1.2), 1.7, 1.8, 1.11]{GekelerNonnengardt1995}; we give a direct proof, which also covers the case $d=1$.
 
 \begin{proof}
     Let $n\geqslant 0$. The stabilizer of $v_n$ in $G$ is $\GL_2(\mathbb{F}_q)$ if $n=0$, and 
     \begin{equation*}
         \begin{pmatrix} \mathbb{F}_q^\times & A_{\leqslant n} \\ 0 & \mathbb{F}_q^\times \end{pmatrix}
     \end{equation*}
     if $n\geqslant 1$, where $A_{\leqslant n}$ is the set of polynomials of degree at most $n$ in $A$ \cite[p.87]{Serre1980}. Intersecting with $\Gamma(\mathfrak{n})$, we see that the stabilizer of $v_n$ in $\Gamma(\mathfrak{n})$ is trivial if $n<d$, and equal to
     \begin{equation*}
         H_n\coloneqq\begin{pmatrix} 1 & \mathfrak{n}A_{\leqslant n-d} \\ 0 & 1 \end{pmatrix}
     \end{equation*}
     if $n\geqslant d$. Therefore the stabilizer in $\Gamma(\mathfrak{n})$ of any lift of $\bar{v}$ is trivial if and only if $\type(\bar{v})<d$. Since $\Gamma(\mathfrak{n})$ acts without inversion, the stabilizer of an edge in $\mathcal{T}$ is the intersection of the stabilizers of its endpoints. Therefore if one of the endpoints of an edge $\bar{e}$ in $\Gamma(\mathfrak{n})\backslash\mathcal{T}$ is stable, then the edge is stable as well. Suppose that the endpoints of $\bar{e}$ have type $n$ and $n+1$ with $n\geqslant d$; then by normality of $\Gamma(\mathfrak{n})$ in $G$, the stabilizer of any lift $e$ is conjugate to $H_n$. Thus $\bar{e}$ is not stable.
     
    \noindent The number of neighbors of a vertex $\bar{v}$ in
    $\Gamma(\mathfrak{n})\backslash\mathcal{T}$ equals the number of orbits of $\Gamma(\mathfrak{n})_v$ on the set of $q+1$ neighbors of any lift $v$ of $\bar{v}$ in $\mathcal{T}$, since $\Gamma(\mathfrak{n})$ acts on $\mathcal{T}$ without inversion.
    It follows that a vertex of type $n<d$ has exactly $q+1$ neighbors. We now show that any vertex $\bar{v}$ of type $n\geqslant d$ has exactly two neighbors: one of type $n-1$, and the other of type $n+1$. It suffices to show it for $\bar{v}=\overline{v_n}$. The only neighbor of $v_n$ of type $n+1$ is $v_{n+1}$, which is represented by $\begin{pmatrix} \pi^{-(n+1)} & 0 \\ 0 & 1 \end{pmatrix}$. For all $\gamma\coloneqq\begin{pmatrix} 1 & \mathfrak{n}a \\ 0 & 1 \end{pmatrix}\in H_n$, we have 
    \begin{equation*}
        \gamma\begin{pmatrix} \pi^{-(n+1)} & 0 \\ 0 & 1 \end{pmatrix}=\begin{pmatrix} \pi^{-(n+1)} & \mathfrak{n}a \\ 0 & 1 \end{pmatrix}
    \end{equation*}
    which also represents $v_{n+1}$ since $\mathfrak{n}a\equiv 0\mod \pi^{-(n+1)}\mathcal{O}_\infty$. Thus $H_n$ fixes $v_{n+1}$.
    The $q$ neighbors of $v_n$ of type $n-1$, denoted by $w_\alpha$  for $\alpha\in\mathbb{F}_q$, are represented by $\begin{pmatrix} \pi^{-(n-1)} & \alpha\pi^{-n} \\ 0 & 1 \end{pmatrix}$. We have 
    \begin{equation*}
        \gamma\begin{pmatrix} \pi^{-(n-1)} & \alpha\pi^{-n} \\ 0 & 1 \end{pmatrix}=\begin{pmatrix} \pi^{-(n-1)} & \alpha\pi^{-n}+\mathfrak{n}a \\ 0 & 1 \end{pmatrix}
    \end{equation*}
    Let $\beta\in\mathbb{F}_q$ denote the coefficient of $T^{n-d}$ in $a$; then $\alpha\pi^{-n}+\mathfrak{n}a\equiv (\alpha+\beta)\pi^{-n}\mod{\pi^{-n+1}\mathcal{O}_\infty}$, and $\gamma$ maps $w_\alpha$ to $w_{\alpha+\beta}$. It follows that $H_n$ acts transitively on the $q$ neighbors of type $n-1$. Thus $\overline{v_n}$ has exactly two neighbors in $\Gamma(\mathfrak{n})\backslash\mathcal{T}$, as claimed.

    \noindent Each neighbor of type $d$ of a vertex of type $d-1$ is the first vertex of a half-line, since every vertex of type $\geqslant d$ has exactly two neighbors. These half-lines cannot intersect at a vertex of type $\geqslant d$: otherwise this vertex would have at least three neighbors. It follows that the vertices of 
$(\Gamma(\mathfrak{n})\backslash\mathcal{T})^0$ are exactly the vertices of type $<d$.

    \noindent Suppose $d=1$. Since $A/\mathfrak{n}A\cong\mathbb{F}_q$, the reduction map $G\twoheadrightarrow  \GL_2(\mathbb{F}_q)$ admits a section, so every $g\in G$ can be written $g=\gamma h$ with $\gamma\in \Gamma(\mathfrak{n})$ and $h\in \GL_2(\mathbb{F}_q)$. Let $v\in X(\mathcal{T})$ be a vertex of type $0$; then $v=g v_0=\gamma h v_0=\gamma v_0$. It follows that
    $(\Gamma(\mathfrak{n})\backslash\mathcal{T})^0$ consists of a single vertex $\overline{v_0}$. 
    Since all neighbors of $v_0$ have type $1$, the $q+1$ half-lines are all attached to $\overline{v_0}$.

    \noindent Suppose $d\geqslant 2$, and that two half-lines are attached to the same vertex $\bar{v}$ of type $d-1\geqslant 1$; then $\bar{v}$ would have two neighbors of type $d$, which is a contradiction.
 \end{proof}
 
    Let $V\coloneqq \#X(\Gamma(\mathfrak{n})\backslash\mathcal{T})^0$ denote the number of stable vertices, and let $E$ denote the number of non-oriented edges in $(\Gamma(\mathfrak{n})\backslash\mathcal{T})^0$. Since the first edge of each half-line is stable, the total number of stable edges is $E+c$. Suppose $d\geqslant 2$. By Lemma \ref{LemmaQuotientGraphGammaN}, any vertex $\bar{v}\in X(\Gamma(\mathfrak{n})\backslash\mathcal{T})^0$ has $q$ or $q+1$ neighbors in  $(\Gamma(\mathfrak{n})\backslash\mathcal{T})^0$, depending on whether a half-line is attached to $\bar{v}$ or not. Therefore
    \begin{equation}\label{equationNumberEdges}
        E=\frac{cq+(V-c)(q+1)}{2}=\frac{V(q+1)-c}{2}.
    \end{equation}
    Euler's formula \cite[p.23]{Serre1980} then gives
    \begin{equation}\label{equationNumberVertices}
        g-1=E-V=\frac{V(q+1)-c}{2}-V=\frac{V(q-1)-c}{2}.
    \end{equation}
    If $d=1$, then \eqref{equationNumberEdges} and \eqref{equationNumberVertices} remain valid since $V=1$, $E=0$, $c=q+1$ and $g=0$.

    \noindent We now explain how to associate linearly independent Poincaré series to vertices and edges of $\Gamma(\mathfrak{n})\backslash\mathcal{T}$.

    \begin{definition}\label{defKuriharaDiff}
    Let $k\geqslant 2$. A \textit{Kurihara differential} is a $k$-form on the projective line $\mathbb{P}^1_{\mathbb{F}_q}$ of the form $\bar{s}=\prod_{i=1}^r s(z;\alpha_i,\beta_i)^{k_i}$ with $r\geqslant 2$, $k_i\geqslant 1$, $\sum_{i=1}^r k_i=k$,  $\alpha_i,\beta_i\in\mathbb{P}^1(\mathbb{F}_q)$ with $\alpha_i\neq\beta_i$, and $\bigcap_{i=1}^r \{\alpha_i,\beta_i\}=\emptyset$.
\end{definition}

\noindent Any Kurihara differential $\bar{s}=\prod_{i=1}^r s(z;\alpha_i,\beta_i)^{k_i}$ has poles of order $\leqslant k-1$ at $\alpha_1,\beta_1,\ldots,\alpha_r,\beta_r\in\mathbb{P}^1(\mathbb{F}_q)$. Let $D\coloneqq \sum_{x\in\mathbb{P}^1(\mathbb{F}_q) } (x)$ be the divisor of rational points on the projective line $\mathbb{P}^1_{\mathbb{F}_q}$; then $H^0(\mathbb{P}^1_{\mathbb{F}_q},\Omega^{k}((k-1)D))$ is the $\overline{\mathbb{F}_q}$-vector space of $k$-forms on $\mathbb{P}^1_{\mathbb{F}_q}$ that have poles of order at most $k-1$ at the points of $\mathbb{P}^1(\mathbb{F}_q)$. We denote by $s$ the $k$-form on $\Omega$ obtained by viewing the same expression $\prod_{i=1}^r s(z;\alpha_i,\beta_i)^{k_i}$ over $\Omega$.

The following result was stated without proof in \cite{Kurihara1994}.

\begin{lemma}\label{lemmaKuriharaSpan}
     The $\overline{\mathbb{F}_q}$-vector space $H^0(\mathbb{P}^1_{\mathbb{F}_q},\Omega^{k}((k-1)D))$ has dimension $\allowbreak N\coloneqq k(q-1)-q$, and is spanned by Kurihara differentials.
\end{lemma}
\begin{proof}
    The dimension formula follows from the Riemann-Roch theorem.
    
    \noindent We first show the following claim: for any $h=\prod_{\alpha\in\mathbb{F}_q}(z-\alpha)^{n_\alpha}$ with $0\leqslant n_\alpha\leqslant k-1$
and $\sigma\coloneqq \sum_{\alpha\in\mathbb{F}_q}n_\alpha\geqslant k+1$, the $k$-form
$h(z)^{-1}(\mathrm{d}z)^k$ lies in the span of Kurihara differentials. We proceed by induction on $\sigma$. Since each $n_\alpha\leqslant k-1$ and
$\sigma\geqslant k+1\geqslant 3$, at least two exponents are nonzero; say
$n_{\alpha_1},n_{\alpha_2}\geqslant 1$.

\noindent If $\sigma=k+1$, then
\begin{equation*}
    \frac{1}{h(z)}
    =\frac{1}{\alpha_2-\alpha_1} \left(\frac{1}{z-\alpha_2}-\frac{1}{z-\alpha_1}\right)
    \frac{1}{(z-\alpha_1)^{n_{\alpha_1}-1}(z-\alpha_2)^{n_{\alpha_2}-1}
    \prod_{\alpha\neq \alpha_1,\alpha_2}(z-\alpha)^{n_\alpha}}
\end{equation*}
and $1+(n_{\alpha_1}-1)+(n_{\alpha_2}-1)+\sum_{\alpha\neq \alpha_1,\alpha_2}n_\alpha = \sigma-1 = k$,
so 
\begin{equation*}
    \frac{1}{h(z)}(\mathrm{d}z)^k
    =\frac{1}{\alpha_2-\alpha_1} s(z;\alpha_1,\alpha_2)s(z;\infty,\alpha_1)^{n_{\alpha_1}-1}s(z;\infty,\alpha_2)^{n_{\alpha_2}-1}\prod_{\alpha\neq \alpha_1,\alpha_2}s(z;\infty,\alpha)^{n_{\alpha}}
\end{equation*}
is an $\mathbb{F}_q^\times$-multiple of a Kurihara differential 
since the pairs occurring with positive exponent have empty common
intersection: if $n_{\alpha_1},n_{\alpha_2}\geqslant 2$, this is
$\{\alpha_1,\alpha_2\}\cap\{\infty,\alpha_1\}\cap\{\infty,\alpha_2\}
=\emptyset$; if $n_{\alpha_i}=1$ for some $i$, then
$n_{\alpha_1}+n_{\alpha_2}<k+1=\sigma$ by the assumption
$n_\alpha\leqslant k-1$, so $n_\alpha\geqslant 1$ for some
$\alpha\notin\{\alpha_1,\alpha_2\}$, and $\{\infty,\alpha\}$ is
disjoint from $\{\alpha_1,\alpha_2\}$. 

\noindent If $\sigma>k+1$, then
\[
    \frac{1}{h(z)}(\mathrm{d}z)^k=\frac{1}{\alpha_2-\alpha_1}
\left(\frac{1}{h_1(z)}(\mathrm{d}z)^k-\frac{1}{h_2(z)}(\mathrm{d}z)^k\right), 
\]
where $h_1$ (resp.\ $h_2$) is obtained from $h$ by replacing $n_{\alpha_1}$
by $n_{\alpha_1}-1$ (resp.\ $n_{\alpha_2}$ by $n_{\alpha_2}-1$). Both $h_1$ and $h_2$
satisfy $0\leqslant n_\alpha\leqslant k-1$ and have exponent sum $\sigma-1$. By induction, $h_1^{-1}(\mathrm{d}z)^k$ and $h_2^{-1}(\mathrm{d}z)^k$ are
in the span, and hence so is $h^{-1}(\mathrm{d}z)^k$.

   \noindent Now define
    \begin{equation*}
        \omega_j\coloneqq  \frac{z^j}{(z^q-z)^{k-1}}\,(\mathrm{d}z)^k
    \end{equation*}
    for $0\leqslant j\leqslant N-1$. Let $\alpha\in \mathbb{P}^1(\mathbb{F}_q)$. A simple calculation shows
    \begin{equation*}
        \ord_\alpha(\omega_j)=\begin{cases}
            1-k &\text{if } \alpha\in\mathbb{F}_q^\times,\\
            1-k+j &\text{if } \alpha=0,\\
            N-k-j &\text{if } \alpha=\infty,
        \end{cases}
    \end{equation*}
    so $\ord_\alpha(\omega_j)\geqslant -(k-1)$ for $0\leqslant j\leqslant N-1$. 
    Since the numerators $z^0, z^1, \ldots, z^{N-1}$ are linearly independent, so are the $\omega_j$; they thus form a basis of $H^0(\mathbb{P}^1_{\mathbb{F}_q},\Omega^{k}((k-1)D))$. We show by induction on $j$ that $\omega_j$ lies in the span of Kurihara differentials. By the claim, it is true for $j=0$ since $\omega_0=h^{-1}(\mathrm{d}z)^k$ with $h=\prod_{\alpha\in\mathbb{F}_q}(z-\alpha)^{k-1}$. Suppose that $\omega_0,\ldots,\omega_{j-1}$ lie in the span. 
    Choose integers
    $0\leqslant n_\alpha\leqslant k-1$ for $\alpha\in\mathbb{F}_q$ with
    $\sum n_\alpha = j$ (this is possible since $j\leqslant N-1 < (k-1)q$).
    Set $h_j\coloneqq \prod_{\alpha\in\mathbb{F}_q}(z-\alpha)^{k-1-n_\alpha}$; then $h_j$ has exponent sum $\sigma=(k-1)q-j\geqslant k+1$, so by the claim $h_j^{-1}(\mathrm{d}z)^k$ is in the span of Kurihara differentials. Write
    $\prod_{\alpha\in\mathbb{F}_q}(z-\alpha)^{n_\alpha} = z^j+\beta_{j-1}z^{j-1}+\cdots+\beta_0$ with $\beta_i\in\mathbb{F}_q$. The equality
\begin{equation*}
    h_j(z)^{-1}(\mathrm{d}z)^k
    = \frac{\prod_{\alpha}(z-\alpha)^{n_\alpha}}{(z^q-z)^{k-1}}(\mathrm{d}z)^k
    = \omega_j + \sum_{i=0}^{j-1}\beta_i\,\omega_i
\end{equation*}
    shows that $\omega_j$ is in the span. This completes the induction.
\end{proof}

 Let $R\coloneqq \{z\in\mathbb{C}_\infty\,:\,\ord z\geqslant 0\}$ be the unique valuation ring of $\mathbb{C}_\infty$, and let $\mathfrak{m}\coloneqq \{z\in\mathbb{C}_\infty\,:\,\ord z> 0\}$ be the maximal ideal of $R$. Each element $a\in\mathbb{P}^1(\mathbb{C}_\infty)$ has a representative $(x:y)$ with one coordinate in $R$ and the other in $R^\times$. Reducing $x$ and $y$ modulo $\mathfrak{m}$ defines an element $\overline{a}$ of $\mathbb{P}^1(\overline{\mathbb{F}_q})$. The resulting map $\mathbb{P}^1(\mathbb{C}_\infty)\to\mathbb{P}^1(\overline{\mathbb{F}_q})$, $a\mapsto\overline{a}$ is the \textit{reduction map}. 

\begin{lemma}\label{lemmaLiftParameters}
    Let $r\geqslant 2$, and let  $\alpha_1,\beta_1,\ldots,\alpha_r,\beta_r\in\mathbb{P}^1(\mathbb{F}_q)$ be such that
    \begin{itemize}
        \item[(a)] $\alpha_i\neq\beta_i$
        \item[(b)] $\{\alpha_i,\beta_i\}\neq\{\alpha_j,\beta_j\}$ for $i\neq j$
        \item[(c)] $\bigcap_{i=1}^r\{\alpha_i,\beta_i\}=\emptyset$.
    \end{itemize}
    Then there exist $a_1,b_1,\ldots,a_r,b_r\in\mathbb{P}^1(F)$ such that
    \begin{itemize}
        \item[(1)] The reductions of $a_i$ and $b_i$ are $\alpha_i$ and $\beta_i$ respectively
        \item[(2)] $\bigcap_{i=1}^r[a_i\to b_i]=\{v_0\}$.
    \end{itemize}
\end{lemma}

\begin{proof}
    Acting by a suitable element of $\GL_2(\mathbb{F}_q)$, we may assume $\alpha_1=\infty,\beta_1=0,\alpha_2=1$. Throughout the proof, we set $a_1=\infty,b_1=0,a_2=1$. Let $\mathcal{O}\coloneqq \mathcal{O}_\infty\cap F$. Swapping $\alpha_i$ and $\beta_i$ if necessary, we can assume $\alpha_i \neq 0$ and $\beta_i \neq \infty$; then any lift of $\alpha_i \neq \infty$ lies in $\mathcal{O}^\times$, and any lift of $\beta_i$ lies in $\mathcal{O}$.

    \noindent We first assume $r=2$. Assumptions (a) and (c) imply $\beta_2\neq \infty,0,1$. Let $b_2$ be a lift of $\beta_2$ in $\mathcal{O}^\times$; then $[\infty\to 0]\cap[1\to b_2]=\{v_0\}$.

    \noindent We now assume $r\geqslant 3$.

    \noindent \textbf{First case:} $\beta_2\neq 0,\infty$. Let $b_2$ be a lift of $\beta_2$ in $\mathcal{O}^\times$. By (a),  $\beta_2\in\mathbb{F}_q^\times\setminus\{1\}$. As before, one has $[\infty\to 0]\cap [1\to b_2]=\{v_0\}$.
    
    \noindent Let $i\geqslant 3$. We want to find lifts $a_i$ and $b_i$ of $\alpha_i$ and $\beta_i$ such that $v_0\in [a_i\to b_i]$.
    \begin{itemize}
        \item If $\alpha_i\neq\infty$, choose $a_i\in \mathcal{O}^\times$ and $b_i\in\mathcal{O}$ such that all lifts $a_i,b_i$ are pairwise distinct. If $n\coloneqq -\ord(b_i)<0$, then $[\infty\to 0]\cap [a_i\to b_i]=[v_0\to v_n]$, since the geodesic $[\infty\to 0]$ meets $[a_i\to b_i]$ exactly between the two branch points $v_0$ and $v_n$. If $b_i\in \mathcal{O}^\times$, then by (a) one has $\alpha_i,\beta_i\in \mathbb{F}_q^\times$ with $\alpha_i\neq\beta_i$, and so $[\infty\to 0]\cap [a_i\to b_i]=\{v_0\}$.
        
        \item If $\alpha_i=\infty$, then $\beta_i\neq 0$ by (b). Choose lifts $a_i\in F\setminus\mathcal{O}$ and $b_i\in \mathcal{O}^\times$ such that all lifts $a_i,b_i$ are pairwise distinct. One has $m\coloneqq -\ord(a_i)>0=\ord(b_i)$ so $[\infty\to 0]\cap[a_i\to b_i]=[v_m\to v_0]$.
    \end{itemize}

    \noindent \textbf{Second case:} $\beta_2$ is equal to $0$ or $\infty$. We only deal with the case $\beta_2=0$, the case $\beta_2=\infty$ being analogous. By (c), there exists some $3\leqslant j \leqslant r$ such that $\alpha_j, \beta_j \neq 0$. Choose a lift $b_2$ in $\mathcal{O}$; then $n\coloneqq -\ord(b_2)<0$, and $[\infty\to 0]\cap[1\to b_2]=[v_0\to v_n]$.
    \begin{itemize}
        \item If $\alpha_j\neq\infty$, choose lifts $a_j,b_j\in \mathcal{O}^\times$ distinct from $a_1,b_1,a_2,b_2$. Since $[\infty\to 0]\cap [a_j\to b_j]=\{v_0\}$, one has
    \begin{equation*}
        [\infty\to 0]\cap[1\to b_2]\cap [a_j\to b_j]=[v_n\to v_0]\cap\{v_0\}=\{v_0\}.
    \end{equation*}
    \item If $\alpha_j=\infty$, choose lifts $a_j\in F\setminus\mathcal{O}$ and $b_j\in\mathcal{O}^\times $ distinct from $a_1,b_1,a_2,b_2$. One has $[\infty\to 0]\cap[a_j\to b_j]=[v_m\to v_0]$, where $m\coloneqq -\ord(a_j)>0$. Therefore
    \begin{equation*}
        [\infty\to 0]\cap[1\to b_2]\cap [a_j\to b_j]=[v_n\to v_0]\cap [v_m\to v_0]=\{v_0\}.
    \end{equation*}
    \end{itemize}
    
    \noindent For $3\leqslant i\leqslant r$, $i\neq j$, one chooses as in the first case $a_i\in \mathcal{O}^\times$ and $b_i\in\mathcal{O}$ (or $a_i\in F\setminus\mathcal{O}$ and $b_i\in \mathcal{O}^\times$ if $\alpha_i=\infty$) such that $a_1,b_1,\ldots,a_r,b_r$ are pairwise distinct and $v_0\in [a_i\to b_i]$. Then $\bigcap_{i=1}^r[a_i\to b_i]=\{v_0\}$.
\end{proof}

In what follows, all Kurihara
differentials and Poincaré series are $k$-forms. Recall that
$N=k(q-1)-q=\dim H^0(\mathbb{P}^1_{\mathbb{F}_q},
\Omega^{k}((k-1)D))\geqslant 1$.

\noindent Let $\bar{v}$ (resp.\ $\bar{e}$) be a vertex (resp.\ an open edge)
of the quotient graph $\Gamma(\mathfrak{n})\backslash\mathcal{T}$.
We define $U(\bar{v})$ as the $\Gamma(\mathfrak{n})$-orbit of
$\Omega_v$ for any lift $v$ of $\bar{v}$, and similarly
$U(\bar{e})$ as the $\Gamma(\mathfrak{n})$-orbit of $\Omega_e$
for any lift $e$ of $\bar{e}$. These subsets of $\Omega$ are
pairwise disjoint for distinct $\bar{v}$ and $\bar{e}$.

\begin{lemma}\label{lemmaHPSVertices}
    Let $\bar{v}$ be a vertex of
    $(\Gamma(\mathfrak{n})\backslash\mathcal{T})^0$, and let
    $v=\delta v_0$ be a lift of $\bar{v}$ with
    $\delta\in\GL_2(F)$.
    Let $\overline{s_1},\ldots,\overline{s_N}$ be a basis of
    $H^0(\mathbb{P}^1_{\mathbb{F}_q},\Omega^{k}((k-1)D))$
    consisting of Kurihara differentials
    $\overline{s_j}=\prod_{i} s(z;\alpha_i^{(j)},\beta_i^{(j)})^{k_i^{(j)}}$.
    For each $1\leqslant j\leqslant N$, lift $\alpha_i^{(j)},\beta_i^{(j)}$ to
    $a_i^{(j)},b_i^{(j)}\in\mathbb{P}^1(F)$ with
    $\bigcap_{i}[a_i^{(j)}\to b_i^{(j)}]=\{v_0\}$
    (Lemma \ref{lemmaLiftParameters}), and define
    \begin{equation*}
        P_j^{\bar{v}}(z)
        \coloneqq \sum_{\gamma\in\Gamma(\mathfrak{n})}
        \prod_{i} \gamma^* s(z;\delta a_i^{(j)},
        \delta b_i^{(j)})^{k_i^{(j)}}.
    \end{equation*}
    Then $P_1^{\bar{v}},\ldots,P_N^{\bar{v}}$ are $N$ nonzero
    linearly independent Poincaré series with zero valuation on
    $U(\bar{v})$ and positive valuation outside $U(\bar{v})$.
\end{lemma}

\begin{proof}
    The geodesics $[\delta a_i^{(j)}\to\delta b_i^{(j)}]$
    intersect at $v=\delta v_0$, whose stabilizer in
    $\Gamma(\mathfrak{n})$ is trivial since $\bar{v}$ is a
    stable vertex (Lemma \ref{LemmaQuotientGraphGammaN}).
    By Theorem \ref{theoremIntersectionVertex}, each
    $P_j^{\bar{v}}$ is nonzero, has zero valuation on
    $U(\bar{v})$, and positive valuation outside $U(\bar{v})$.
    Write $f_j\coloneqq P_j^{\bar{v}}/(\mathrm{d}z)^k$ and
    $g_j(\zeta)\coloneqq f_j(\phi_v(\zeta))\cdot(\phi_v'(\zeta))^k$
    for $\zeta\in\Omega_{v_0}$, so that
    $\ord(g_j(\zeta))=\ord_z(P_j^{\bar{v}})=0$.
    By \eqref{equationValuationRPS}, the reduction
    of $g_j$ modulo $\mathfrak{m}$ coincides with $\overline{u_j}$, where $\overline{s_j}=\overline{u_j}(z)(\mathrm{d}z)^k$.
    Suppose $\sum_{j=1}^N c_j P_j^{\bar{v}}=0$ with
    $c_j\in\mathbb{C}_\infty$; after rescaling we may assume
    $c_j\in R$ with at least one $c_j\in R^\times$.
    Then $\sum_j c_j g_j(\zeta)=0$ on $\Omega_{v_0}$, and
    reducing modulo $\mathfrak{m}$ gives
    $\sum_j \overline{c_j}\,\overline{u_j}=0$, contradicting the
    linear independence of
    $\overline{s_1},\ldots,\overline{s_N}$.
\end{proof}

\begin{lemma}\label{lemmaHPSEdges}
    Let $\bar{e}$ be a stable edge of $\Gamma(\mathfrak{n})\backslash\mathcal{T}$ with an endpoint $\bar{v}$ in $(\Gamma(\mathfrak{n})\backslash\mathcal{T})^0$. 
There exists a nonzero   Poincaré series $P^{\bar{e}}$ with zero valuation on $U(\bar{v})\cup U(\bar{e})$, and positive valuation on $\Omega\backslash(U(\bar{v})\cup U(\bar{e})\cup U(\bar{w}))$,
  where $\bar{w}$ is the second endpoint of $\bar{e}$.
\end{lemma}
\begin{proof}
    Choose $r\geqslant 2$, $k_1,\ldots,k_r\geqslant 1$ with $\sum k_i=k$, lifts $e,v$ of $\bar{e},\bar{v}$, and $a_1,b_1,\ldots,a_r,b_r\in\mathbb{P}^1(F)$ such that $\bigcap_{i=1}^r[a_i\to b_i]$ consists of $e$ and its endpoints; such configurations exist, and there are infinitely many of them. For example, for the edge with endpoints $v_0,v_1$ one can take $r=2$ and the geodesics $[\infty\to b_1]$, $[a_2\to b_2]$ with $b_1,a_2,b_2\in F$ of valuations $\ord(b_1)\geqslant 1$, $\ord(a_2)=-1$, $\ord(b_2)=0$; the general case follows by applying an element of $\GL_2(F)$, which acts transitively on the set of non-oriented edges of $\mathcal{T}$.

\noindent The stabilizer of $v$ in $\Gamma(\mathfrak{n})$ is trivial because $\bar{v}$ is stable, so by Theorem \ref{theoremIntersectionEdge} the   Poincaré series 
    \begin{equation*}
        P^{\bar{e}}(z)\coloneqq \sum_{\gamma\in\Gamma(\mathfrak{n})}\prod_{i=1}^r \gamma^*s(z;a_i,b_i)^{k_i}
    \end{equation*}
    is nonzero and has the desired valuation properties.
\end{proof}

We are now ready to prove Theorem \ref{theoremPSspan}. 

\begin{proof}[Proof of Theorem \ref{theoremPSspan}] We have constructed in Lemmas \ref{lemmaHPSVertices} and \ref{lemmaHPSEdges} a family of Poincaré series
\begin{equation*}
    \mathcal{P}\coloneqq \{P_1^{\bar{v}},\ldots,P_N^{\bar{v}}\}_{\bar{v}}\cup\{P^{\bar{e}}\}_{\bar{e}}
\end{equation*}
    where $\bar{v}$ (resp. $\bar{e}$) runs over the vertices of $(\Gamma(\mathfrak{n})\backslash\mathcal{T})^0$ (resp. the stable edges of $\Gamma(\mathfrak{n})\backslash\mathcal{T}$). By \eqref{equationNumberEdges} and \eqref{equationNumberVertices}, one has $\#\mathcal{P}=V\cdot N+E+c=(2k-1)(g-1)+kc$. It remains to show that the elements of $\mathcal{P}$ are linearly independent. Suppose that they satisfy a nontrivial linear relation
    \begin{equation}\label{equationDependenceRelationPS}
        \sum_{\bar{v},j}\mu_j^{\bar{v}} P_j^{\bar{v}}(z)+\sum_{\bar{e}}\mu^{\bar{e}}P^{\bar{e}}(z)=0
    \end{equation}
    with coefficients $\mu_j^{\bar{v}},\mu^{\bar{e}}\in\mathbb{C}_\infty$. After multiplying both sides of \eqref{equationDependenceRelationPS} by an element of $\mathbb{C}_\infty^\times$ if needed, we can assume that all coefficients lie in $R$, and that one of them lies in $R^\times$.

    \noindent For $z\in\Omega$, let $p=\lambda(z)$, $\phi_p:\Omega_{v_0}\xrightarrow{\sim}\Omega_{p}$ and $\zeta=\phi_p^{-1}(z)$ as in  \S\ref{sectionTree}. 
    Let $g_j^{\bar{v}}(\zeta)\coloneqq \phi_p^*P_j^{\bar{v}}(\zeta)/(\mathrm{d}\zeta)^k$ and $g^{\bar{e}}(\zeta)\coloneqq \phi_p^*P^{\bar{e}}(\zeta)/(\mathrm{d}\zeta)^k$. Then \eqref{equationDependenceRelationPS} becomes
    \begin{equation}\label{equationDependenceRelationFunctionPS}
        \sum_{\bar{v},j}\mu_j^{\bar{v}} g_j^{\bar{v}}(\zeta)+\sum_{\bar{e}}\mu^{\bar{e}}g^{\bar{e}}(\zeta)=0.
    \end{equation}
    Suppose that $\ord \mu^{\overline{e_0}}=0$ for some $\overline{e_0}$. Then for $z\in U(\overline{e_0})$, one has $\ord g_j^{\bar{v}}(\zeta)>0$, $\ord g^{\overline{e_0}}(\zeta)=0$ and $\ord g^{\bar{e}}(\zeta)>0$ for $\overline{e}\neq \overline{e_0}$. Reducing both sides of \eqref{equationDependenceRelationFunctionPS} modulo $\mathfrak{m}$, we get $\mu^{\overline{e_0}}g^{\overline{e_0}}(\zeta)\in\mathfrak{m}$. 
    This is a contradiction, since both $\mu^{\overline{e_0}}$ and $g^{\overline{e_0}}(\zeta)$ have valuation $0$.

    \noindent Now suppose that $\ord \mu^{\bar{e}}>0$ for all $\bar{e}$. Then $\ord \mu_i^{\overline{v'}}=0$ for some $\overline{v'}$ and $i$. For $z\in U(\overline{v'})$, one has $\ord g_i^{\overline{v'}}(\zeta)=0$ and $\ord g_j^{\bar{v}}(\zeta)>0$ for $\bar{v}\neq\overline{v'}$ and all $1\leqslant j\leqslant N$. Reducing both sides of \eqref{equationDependenceRelationFunctionPS} modulo $\mathfrak{m}$, we get $\sum_{j}\mu_j^{\overline{v'}}\, g_j^{\overline{v'}}(\zeta)\equiv 0\mod\mathfrak{m}$. The reduced functions $g_i^{\overline{v'}}\mod\mathfrak{m}$ are $\overline{\mathbb{F}_q}$-linearly independent by  Lemma \ref{lemmaHPSVertices}, so $\mu_i^{\overline{v'}}\in\mathfrak{m}$, contradiction.
\end{proof}

\begin{remark}
    As a rigid analytic curve, $\Omega$ admits an \textit{analytic reduction} $\widetilde{\Omega}$. It is a "tree of projective lines" defined over $\mathbb{F}_q$: each irreducible component $E_v\cong \mathbb{P}^1_{\mathbb{F}_q}$ corresponds to a vertex $v\in X(\mathcal{T})$, and $E_v\cap E_w$ is nonempty if and only if $v$ and $w$ are adjacent \cite[(1.2)]{GekelerReversat1996}. In this case, the intersection point corresponds to the open edge connecting $v$ and $w$. The analytic reduction map $\Omega\to\widetilde{\Omega}$ is compatible with the action of $\Gamma(\mathfrak{n})$; in particular, the irreducible components of $\widetilde{\Gamma(\mathfrak{n})\backslash\Omega}$ are the $E_{\bar{v}}\cong \mathbb{P}^1_{\mathbb{F}_q}$ for $\bar{v}\in X(\Gamma(\mathfrak{n})\backslash\mathcal{T})$. The Poincaré series $P_1^{\bar{v}},\ldots,P_N^{\bar{v}}$ constructed in Lemma \ref{lemmaHPSVertices} reduce to Kurihara differentials on 
$E_{\bar{v}}\cong\mathbb{P}^1_{\mathbb{F}_q}$: each $P_j^{\bar{v}}$ has valuation zero on $U(\bar{v})$, and its reduction modulo $\mathfrak{m}$ is a nonzero $k$-form on $E_{\bar{v}}$.
\end{remark}

\begin{remark}\label{rmkAlgo}
    The proof of Theorem \ref{theoremPSspan} yields the following algorithm
    for constructing $(2k-1)(g-1)+kc$ linearly independent Poincaré series
    in $M_{2k}^1(\Gamma(\mathfrak{n}))$.
    \begin{enumerate}
        \item List the stable vertices and edges of $\Gamma(\mathfrak{n})\backslash\mathcal{T}$.
        \item For each stable vertex $\bar{v}$, choose a lift
            $v=\delta v_0$ with $\delta\in\GL_2(F)$.
        \item Choose a basis $\overline{s_1},\ldots,\overline{s_N}$ of
            $H^0(\mathbb{P}^1_{\mathbb{F}_q},
            \Omega^k((k-1)D))$
            consisting of Kurihara differentials
            $\overline{s_j}=\prod_i s(z;\alpha_i^{(j)},\beta_i^{(j)})^{k_i^{(j)}}$. In practice, one can check that  $\overline{s_1},\ldots,\overline{s_N}$
    are linearly independent by expressing each $\overline{s_j}$ in the basis $\{\omega_j\}$ constructed in the proof of
    Lemma \ref{lemmaKuriharaSpan}, and checking that the resulting
    $N\times N$ change of basis matrix is nonsingular.
        \item For each $\overline{s_j}$, use Lemma \ref{lemmaLiftParameters}
            to lift
            $\alpha_i^{(j)},\beta_i^{(j)}\in\mathbb{P}^1(\mathbb{F}_q)$
            to $a_i^{(j)},b_i^{(j)}\in\mathbb{P}^1(F)$ with
            $\bigcap_i[a_i^{(j)}\to b_i^{(j)}]=\{v_0\}$.
            The Poincaré series attached to $\bar{v}$ and $\overline{s_j}$ is
            \begin{equation*}
                P_j^{\bar{v}}(z)
                =\sum_{\gamma\in\Gamma(\mathfrak{n})}
                \prod_i\gamma^*s(z;\delta a_i^{(j)},\delta b_i^{(j)})
                ^{k_i^{(j)}}.
            \end{equation*}
        \item For each stable edge $\bar{e}$ with stable endpoint
            $\bar{v}$, choose a lift $e$ in $\mathcal{T}$ and
            $a_1,b_1,\ldots,a_r,b_r\in\mathbb{P}^1(F)$ such that
            $\bigcap_i[a_i\to b_i]$ consists of $e$ and its endpoints.
            The Poincaré series attached to $\bar{e}$ is
            \begin{equation*}
                P^{\bar{e}}(z)
                =\sum_{\gamma\in\Gamma(\mathfrak{n})}
                \prod_i\gamma^*s(z;a_i,b_i)^{k_i}.
            \end{equation*}
    \end{enumerate}
    The family
    $\mathcal{P}=\{P_j^{\bar{v}}\}_{1\leqslant j\leqslant N,\,
    \bar{v}\text{ stable}}\cup\{P^{\bar{e}}\}_{\bar{e}\text{ stable}}$
    is a linearly independent set of $(2k-1)(g-1)+kc$ Poincaré series. The corresponding cusp forms $f_j^{\bar{v}}$ and $f^{\bar{e}}$ thus span a space of the same dimension in $M_{2k}^1(\Gamma(\mathfrak{n}))$.
    The choices in steps (2)--(5) affect the individual Poincaré series, but the subspace spanned by all Poincaré series is independent of these choices.
\end{remark}

\noindent We illustrate this construction with an explicit example.

\begin{example}\label{exampleGenerationGamma(T)}
   Take $\mathfrak{n}=T$ and $k=2$. We have $g=0$ and $c=q+1$, and $(\Gamma(T)\backslash\mathcal{T})^0$ consists of a single vertex $\overline{v_0}$. The cusps of $\Gamma(T)$ are exactly the elements $\alpha$ of $\mathbb{P}^1(\mathbb{F}_q)$.
    The space $H^0(\mathbb{P}^1_{\mathbb{F}_q},\Omega^{k}((k-1)D))=H^0(\mathbb{P}^1_{\mathbb{F}_q},\Omega^{2}(D))$ has dimension $N=k(q-1)-q=q-2$ over $\overline{\mathbb{F}_q}$. For $\beta\in\mathbb{F}_q\setminus\{0,1\}$, define 
\begin{equation*}
    s_\beta\coloneqq s(z;\infty,0)\cdot s(z;1,\beta)=\frac{\beta-1}{z(z-1)(z-\beta)}(\mathrm{d}z)^2.
\end{equation*}
For $\beta\neq\beta'$, the form $\overline{s_{\beta'}}$ is regular at $\beta$; it follows that $\{\overline{s_\beta}\}_{\beta\in\mathbb{F}_q\backslash\{0,1\}}$ 
is a linearly independent set, hence a basis of 
$H^0(\mathbb{P}^1_{\mathbb{F}_q},\Omega^{2}(D))$. Each $\overline{s_\beta}$ lifts to a Poincaré series 
\begin{align*}
    P^{\overline{v_0}}_\beta(z)&\coloneqq \sum_{\gamma\in\Gamma(T)}\gamma^*
                s_\beta(z)\\
                &=\sum_{\gamma\in\Gamma(T)}               s(z;\gamma\infty,\gamma0)\cdot s(z;\gamma1,\gamma\beta)
\end{align*}
with $[\infty\to 0]\cap [1\to \beta]=\{v_0\}$.

\noindent The stable edges of $\Gamma(T)\backslash\mathcal{T}$ are the $q+1$ edges emanating from $\overline{v_0}$. They are exactly the $\overline{e_\alpha}$ for $\alpha\in\mathbb{P}^1(\mathbb{F}_q)$, where  $e_\alpha$ is the unique open edge emanating from $v_0$ toward the cusp $\alpha$. The Poincaré series corresponding to $\overline{e_\alpha}$ is
\begin{equation}\label{HPSatEdges}
    P^{\overline{e_\alpha}}(z)\coloneqq \begin{cases}
        \sum_{\gamma\in\Gamma(T)}
                s(z;\gamma \infty,\gamma 0)\cdot s(z;\gamma T,\gamma 1)
        &\text{if } \alpha=\infty,\\
        \sum_{\gamma\in\Gamma(T)}
                s(z;\gamma \infty,\gamma \alpha)\cdot s(z;\gamma (\alpha+1),\gamma (\alpha+\pi))
        &\text{if } \alpha\in\mathbb{F}_q.
    \end{cases}
\end{equation}
Note that in both cases, the intersection of the two geodesics consists of $e_\alpha$ and its endpoints. One has $\mathcal{P}=\{P^{\overline{v_0}}_\beta\}_{\beta\in\mathbb{F}_q\backslash\{0,1\}}\cup\{P^{\overline{e_\alpha}}\}_{\alpha\in\mathbb{P}^1(\mathbb{F}_q)}$.

\noindent The cusp form corresponding to $P^{\overline{v_0}}_\beta$ is
\begin{align*}
        f^{\overline{v_0}}_\beta(z)&\coloneqq \sum_{\gamma\in \Gamma(T)} u(z;\gamma\infty,\gamma0)\cdot u(z;\gamma1,\gamma\beta)\\
        &=\sum_{\gamma\in \Gamma(T)}\left(\frac{1}{z-b/d}-\frac{1}{z-a/c}\right)\left(\frac{1}{z-(a\beta+b)/(c\beta+d)}-\frac{1}{z-(a+b)/(c+d)}\right),
\end{align*}
while the cusp form corresponding to $P^{\overline{e_\alpha}}$ is
\begin{align*}
        f^{\overline{e_\infty}}(z)&\coloneqq \sum_{\gamma\in \Gamma(T)} u(z;\gamma \infty,\gamma 0)\cdot u(z;\gamma T,\gamma 1)\\
        &=\sum_{\gamma\in \Gamma(T)}\left(\frac{1}{z-b/d}-\frac{1}{z-a/c}\right)\left(\frac{1}{z-(a+b)/(c+d)}-\frac{1}{z-(aT+b)/(cT+d)}\right),
\end{align*}
if $\alpha=\infty$, and
\begin{align*}
        f^{\overline{e_\alpha}}(z)&=\sum_{\gamma\in \Gamma(T)} u(z;\gamma \infty,\gamma \alpha)\cdot u(z;\gamma (\alpha+1),\gamma (\alpha+\pi))\\
        &=\sum_{\gamma\in \Gamma(T)}\left(\frac{1}{z-(a\alpha+b)/(c\alpha+d)}-\frac{1}{z-a/c}\right)\\
        &\;\;\;\;\times\left(\frac{1}{z-(a(\alpha+\pi)+b)/(c(\alpha+\pi)+d)}-\frac{1}{z-(a(\alpha+1)+b)/(c(\alpha+1)+d)}\right)
\end{align*}
for $\alpha\in\mathbb{F}_q$. The family
    $\{f^{\overline{v_0}}_\beta\}_{\beta\in\mathbb{F}_q\backslash\{0,1\}}\cup \{f^{\overline{e_\alpha}}\}_{\alpha\in\mathbb{P}^1(\mathbb{F}_q)}$
    is a linearly independent set of $(q-2)+(q+1)=2q-1=(2k-1)(g-1)+kc$ cusp forms in $M_{4}^1(\Gamma(T))$.
\end{example}

Computations with SageMath suggest that all forms in the previous example vanish to order exactly $4 - p_\alpha$ at the cusp $\alpha\in\mathbb{P}^1(\mathbb{F}_q)$, where
    \[p_\alpha \coloneqq  \#\{i\in\{1,2\} : \overline{a_i}=\alpha\textup{ or }
    \overline{b_i}=\alpha\}.\] Assuming these vanishing orders and comparing the dimension with \eqref{eqdim}, we see that $\{f^{\overline{v_0}}_\beta\}_{\beta\in\mathbb{F}_q\backslash\{0,1\}}$ is a basis of $M_{4}^3(\Gamma(T))$, and that $\{f^{\overline{v_0}}_\beta\}_{\beta\in\mathbb{F}_q\backslash\{0,1\}}\cup \{f^{\overline{e_\alpha}}\}_{\alpha\in\mathbb{P}^1(\mathbb{F}_q)}$ is a basis of $M_{4}^2(\Gamma(T))$.
    
    \noindent More generally, we make the following two conjectures:

\begin{conjecture}\label{conjectureVanishingExact}
    Let $f$ be a Poincaré series
    \begin{equation*}
        f(z) = \sum_{\gamma\in\Gamma(\mathfrak{n})}
        \prod_{i=1}^r u(z;\gamma a_i,\gamma b_i)^{k_i}
    \end{equation*}
    with $a_i,b_i\in\mathbb{P}^1(\mathbb{F}_q)$ and
    $\bigcap_{i=1}^r [a_i\to b_i] = \{v_0\}$, and for
    $\alpha\in\mathbb{P}^1(\mathbb{F}_q)$ let
    \begin{equation*}
        p_\alpha \coloneqq \sum_{\substack{1\leqslant i\leqslant r \\
        \alpha\in\{a_i,b_i\}}} k_i
    \end{equation*}
    be the pole order at $\alpha$ of the Kurihara differential
    $\prod_{i=1}^r s(z;a_i,b_i)^{k_i}$. If $\mathfrak{c}$ is a cusp of $\Gamma(\mathfrak{n})$ represented by
    $\alpha\in\mathbb{P}^1(\mathbb{F}_q)$, then $f$ vanishes to order at least $2k-p_\alpha$ at $\mathfrak{c}$. If moreover $p_\alpha\leqslant q-1$ for all $\alpha$, then $f$ vanishes to order exactly $2k-p_\alpha$ at $\mathfrak{c}$.
\end{conjecture}

\begin{conjecture}\label{conjectureVanishingOrders}
    Let $\bar{v}$ be a stable vertex, $\bar{e}$ an edge in $(\Gamma(\mathfrak{n})\backslash\mathcal{T})^0$,
    and $\overline{e_\mathfrak{c}}$ the unique stable edge lying on a half-line toward the cusp $\mathfrak{c}$.
    \begin{enumerate}
        \item The vertex Poincaré series $f_1^{\bar{v}},\ldots,f_N^{\bar{v}}$
            vanish to order $\geqslant k+1$ at every cusp, and thus lie in
            $M_{2k}^{k+1}(\Gamma(\mathfrak{n}))$.
        \item The non-cuspidal edge Poincaré series $f^{\bar{e}}$
            vanishes to order $\geqslant k+1$ at every cusp, and thus lies in
            $M_{2k}^{k+1}(\Gamma(\mathfrak{n}))$.
        \item The cuspidal edge Poincaré series $f^{\overline{e_\mathfrak{c}}}$
            vanishes to order exactly $k$ at the cusp $\mathfrak{c}$, and to order
            $\geqslant k+1$ at every other cusp.
            In particular, $f^{\overline{e_\mathfrak{c}}}\in
            M_{2k}^{k}(\Gamma(\mathfrak{n}))\setminus
            M_{2k}^{k+1}(\Gamma(\mathfrak{n}))$.
    \end{enumerate}
\end{conjecture}

\noindent Numerical evidence with SageMath for $q\in\{3,5\}$,
    $\mathfrak{n}\in\{T,\, T^2,\, T(T+1)\}$ and various weights $k$ supports these conjectures. In
    \S\ref{subsectionFamily} we prove Conjecture
    \ref{conjectureVanishingExact}, for an
    explicit family of Poincaré series (Remark
    \ref{remarkConjectureCase}).

\begin{proposition}\label{propConditionalBasis}
    Assume Conjecture \ref{conjectureVanishingOrders}. Then the
    $V\cdot N+E$ vertex and non-cuspidal edge Poincaré series form
    a basis of $M_{2k}^{k+1}(\Gamma(\mathfrak{n}))$, and together
    with the $c$ cuspidal edge Poincaré series they form a basis
    of $M_{2k}^{k}(\Gamma(\mathfrak{n}))$.
\end{proposition}
\begin{proof}
    Combining \eqref{eqdim} with \eqref{equationNumberEdges}
    and \eqref{equationNumberVertices} gives
    \begin{equation*}
        \dim M_{2k}^{k+1}(\Gamma(\mathfrak{n}))=V\cdot N+E,
    \end{equation*}
    and 
    \begin{equation*}
        \dim M_{2k}^{k}(\Gamma(\mathfrak{n}))=V\cdot N+E+c.
    \end{equation*}
    By parts (1) and (2) of
    Conjecture \ref{conjectureVanishingOrders}, the vertex and
    non-cuspidal edge Poincaré series lie in
    $M_{2k}^{k+1}(\Gamma(\mathfrak{n}))$; by part (3), the
    cuspidal edge Poincaré series lie in
    $M_{2k}^{k}(\Gamma(\mathfrak{n}))$. Since the whole family is
    linearly independent (Theorem \ref{theoremPSspan}), both
    claims follow by comparing cardinalities with the dimensions
    above.
\end{proof}

\subsection{The quaternionic case}\label{subsectionQuaternionic}

Let $B$ be a division quaternion algebra over $F$ that splits at $\infty$,
i.e.\ such that there exists an isomorphism of $F_\infty$-algebras
$\iota\colon B\otimes_F F_\infty\xrightarrow{\sim} M_2(F_\infty)$. We denote by
$\mathrm{Nr}$ the reduced norm of $B$; under $\iota$
it corresponds to the determinant on $M_2(F_\infty)$. Fix a
maximal $A$-order $\mathcal{O}$ in $B$, and let
\begin{equation*}
    \mathcal{O}^\times=\{x\in \mathcal{O}:\mathrm{Nr}(x)\in\mathbb{F}_q^\times\}
\end{equation*}
be its group of units. The group $\iota(\mathcal{O}^\times)$ is a discrete
cocompact subgroup of $\GL_2(F_\infty)$ satisfying
$\det(\iota(\mathcal{O}^\times))\subseteq\mathbb{F}_q^\times$, and
$\iota(\mathcal{O}^\times)\backslash\mathcal{T}$ is a finite graph \cite[Lemma 5.1]{Papikian2012}.
Let $\Gamma$ be a finite-index subgroup of $\iota(\mathcal{O}^\times)$. The
center $Z(\Gamma)$ consists of the scalar matrices in $\Gamma$, and
$Z(\Gamma)\subseteq\mathbb{F}_q^\times\cdot I$. Set
$\widetilde{\Gamma}\coloneqq \Gamma/Z(\Gamma)$.

Throughout, we fix $k\geqslant 2$ and an integer $l$ such that $\det(h)^l=1$ for
all $h\in Z(\Gamma)$, i.e.\ such that $2l\equiv 0\bmod\#Z(\Gamma)$, and we put
$m\coloneqq k+l$. We consider Poincaré series of the form
\begin{equation}\label{equationQuaternionicPS}
    f(z)=\sum_{[\gamma]\in \widetilde{\Gamma}}(\det\gamma)^{-l}
    \prod_{i=1}^r u(z;\gamma a_i,\gamma b_i)^{k_i}
\end{equation}
with $\sum k_i=k$, $r\geqslant 2$, $H=Z(\Gamma)$ and $\bigcap_{i=1}^r \{a_i,b_i\}=\emptyset$. Condition \textnormal{\textbf{(P)}} holds by Theorem \ref{theorem(P)}(ii), which applies with $H=Z(\Gamma)$ by the remark following it; by Corollary \ref{coroCuspForm}, $f\in M_{2k,m}(\Gamma)$. The condition on $l$ is the one of \S\ref{sectionModularForms} for weight $2k$ and type $m$, namely $2(k-m)\equiv 0\bmod\#Z(\Gamma)$; without it $M_{2k,m}(\Gamma)$ is zero.

The aim of this subsection is to prove the following result.

\begin{theorem}\label{theoremPSspanQuat}
    The Poincaré series
    \eqref{equationQuaternionicPS} span the space $M_{2k,m}(\Gamma)$.
\end{theorem}

\noindent The strategy of the proof is as follows. We first prove Theorem \ref{theoremPSspanQuat} when $\widetilde{\Gamma}$ is torsion-free. In that case the construction of Theorem \ref{theoremPSspan} applies and produces
$\dim M_{2k,m}(\Gamma)$ linearly independent Poincaré series. The general case follows by a trace argument, which
reduces the statement to the torsion-free case.

Assume that $\widetilde{\Gamma}$ is torsion-free; the stabilizer of any vertex $v\in X(\mathcal{T})$ satisfies
$\Gamma_v=Z(\Gamma)$, and the quotient
$\Gamma\backslash\mathcal{T}$ is a finite $(q+1)$-regular graph. Writing $V$ and $E$ for the number
of vertices and of non-oriented edges of $\Gamma\backslash\mathcal{T}$, we have
\begin{equation}\label{equationQuatVE}
    E=\frac{V(q+1)}{2},\qquad g-1=E-V=\frac{V(q-1)}{2},
\end{equation}
where $g$ is the genus of $X_\Gamma$.

\noindent  As in the proof of Theorem \ref{theoremPSspan}, we associate to each vertex $\bar{v}$ (resp. non-oriented edge $\bar{e}$) of $\Gamma\backslash\mathcal{T}$ Poincaré series    $f^{\bar{v}}_1,\ldots,f^{\bar{v}}_N$ (resp. $f^{\bar{e}}$) of weight $2k$ and type $m$ as in \eqref{equationQuaternionicPS}. Since $\Gamma_v=Z(\Gamma)$ for every vertex $v$, 
the hypothesis $\Gamma_v\subseteq H$ of Theorems
\ref{theoremIntersectionVertex} and \ref{theoremIntersectionEdge}
is satisfied with $H=Z(\Gamma)$, so these theorems apply. Lemmas
\ref{lemmaHPSVertices} and \ref{lemmaHPSEdges} carry over with
$\Gamma(\mathfrak{n})$ replaced by $\Gamma$ and the sum taken over
$Z(\Gamma)\backslash\Gamma$ with the factor $(\det\gamma)^{l}$ as
in \eqref{equationQuaternionicPS}: since $Z(\Gamma)$ acts
trivially on $\Omega$ and on $\mathbb{P}^1(F_\infty)$, and
$(\det h)^{l}=1$ for $h\in Z(\Gamma)$, their proofs go through
unchanged. The proof of Theorem \ref{theoremPSspan}
    then shows that these $V\cdot N+E$ forms are linearly independent. As $V\cdot N+E=(2k-1)(g-1)=\dim M_{2k,m}(\Gamma)$ by Proposition \ref{propDim}, these Poincaré series form a basis of $M_{2k,m}(\Gamma)$, which proves Theorem \ref{theoremPSspanQuat} in the case where $\widetilde{\Gamma}$ is torsion-free.

    We now turn to the general case: $\widetilde{\Gamma}$ is not necessarily torsion-free.

\begin{lemma}\label{lemmaExistsGamma0}
    There exists a normal subgroup $\Gamma_0\leqslant\Gamma$ of index prime to $p$ such
    that $Z(\Gamma_0)=Z(\Gamma)$ and $\widetilde{\Gamma_0}\coloneqq \Gamma_0/Z(\Gamma_0)$ is torsion-free.
\end{lemma}

\begin{proof}
    Since $B$ is a division algebra, it ramifies at some place $\mathfrak{p}$ of
    $F$; put $d\coloneqq\deg(\mathfrak{p})$.
    The completion $B_\mathfrak{p}$ is then a division algebra over
    $F_\mathfrak{p}$ that comes with a valuation $w\coloneqq\frac{1}{2}(\ord_\mathfrak{p}\circ\mathrm{Nr})$. Let $\mathcal{O}_\mathfrak{p}\coloneqq\mathcal{O}
    \otimes_A A_\mathfrak{p}=\{x\in B_\mathfrak{p}:w(x)\geqslant 0 \}$ denote its valuation ring, and let $\mathfrak{P}=\{x\in B_\mathfrak{p}:w(x)> 0 \}$ denote its unique maximal
    two-sided ideal; then $\mathcal{O}_\mathfrak{p}/\mathfrak{P}\cong\mathbb{F}_{q^{2d}}$
    \cite[\S 13.3]{Voight2021}. Let
    \begin{equation*}
        \rho\colon\mathcal{O}\longrightarrow
        \mathcal{O}_\mathfrak{p}/\mathfrak{P}\cong\mathbb{F}_{q^{2d}}
    \end{equation*}
    be the reduction map, and set $\Gamma_0\coloneqq\Gamma\cap \iota(\rho^{-1}(\mathbb{F}_q^\times))$. Note that $\Gamma_0$ is normal in $\Gamma$, since it is the kernel of the natural map $\Gamma\to \mathbb{F}_{q^{2d}}^\times/\mathbb{F}_{q}^\times$. Moreover $\Gamma/\Gamma_0$ injects in $\mathbb{F}_{q^{2d}}^\times/
    \mathbb{F}_q^\times$, so $[\Gamma:\Gamma_0]$ divides
    $(q^{2d}-1)/(q-1)$, which is prime to $p$.

    \noindent  Let $cI=\iota(c)\in Z(\Gamma)$ with $c\in\mathbb{F}_q^\times$. We have $\rho(c)^{q-1}=\rho(c^{q-1})=1$ so $\rho(c)\in \mathbb{F}_q^\times$, i.e., $cI\in\Gamma_0$. Therefore $\Gamma_0$ and $\Gamma$ have the same scalar matrices, i.e., $Z(\Gamma_0)=Z(\Gamma)$.

    \noindent  It remains to show that $\widetilde{\Gamma_0}$ is torsion-free. Let $\iota(x
    )\in\Gamma_0$ represent
    a torsion element of $\widetilde{\Gamma_0}$; then $x$ has finite order in $B^\times$. Consequently
    $\mathbb{F}_q[x]$ is a finite commutative subring of the division algebra
    $B$, hence a finite field. Now
    $\rho(x)\in\mathbb{F}_q^\times=\rho(\mathbb{F}_q^\times)$, so $\rho(x)=\rho(c)$ for some $c\in \mathbb{F}_q^\times$. The
    restriction of $\rho$ to $\mathbb{F}_q[x]$ is injective, so $x=c\in \mathbb{F}_q^\times$, and $\iota(x
    )=cI\in Z(\Gamma_0)$.
\end{proof}

Let $\Gamma_0\leqslant\Gamma$ be as in Lemma \ref{lemmaExistsGamma0}, and let $n\coloneqq [\Gamma:\Gamma_0]$. Fix
representatives $\gamma_1,\ldots,\gamma_n$ of the right cosets
$\Gamma_0\backslash\Gamma$. A
function $f\colon\Omega\to\mathbb{C}_\infty$ lies in $M_{2k,m}(\Gamma)$ if and only if the associated $k$-form $\omega_f\coloneqq f(z)(\mathrm{d}z)^k$
satisfies $\gamma^*\omega_f=(\det\gamma)^{-l}\omega_f$ for all $\gamma\in\Gamma$
(there is no condition at the cusps, as $\Gamma$ is cocompact). For a $k$-form
$\omega$ on $\Omega$ satisfying this identity for all $\gamma\in\Gamma_0$, set
\begin{equation}\label{equationTrace}
    \mathrm{Tr}_{\Gamma/\Gamma_0}(\omega)
    \coloneqq\sum_{i=1}^{n}(\det\gamma_i)^{l}\,\gamma_i^*\omega .
\end{equation}

\noindent  A direct computation shows that $\mathrm{Tr}_{\Gamma/\Gamma_0}(\omega)$ does not depend on the choice of the representatives $\gamma_i$, and that
            $\mathrm{Tr}_{\Gamma/\Gamma_0}(\omega)$ satisfies
            $\gamma^*\mathrm{Tr}_{\Gamma/\Gamma_0}(\omega)
            =(\det\gamma)^{-l}\mathrm{Tr}_{\Gamma/\Gamma_0}(\omega)$ for all
            $\gamma\in\Gamma$. Passing to modular forms, we see that $\mathrm{Tr}_{\Gamma/\Gamma_0}$ defines a
            $\mathbb{C}_\infty$-linear map
            $M_{2k,m}(\Gamma_0)\to M_{2k,m}(\Gamma)$.

\begin{lemma}\label{lemmaTrace}\hfill
    \begin{enumerate}
        \item[\textnormal{(i)}] The map
            $\mathrm{Tr}_{\Gamma/\Gamma_0}\colon M_{2k,m}(\Gamma_0)\to
            M_{2k,m}(\Gamma)$ is surjective.
        \item[\textnormal{(ii)}] Let $s(z)=\prod_{i=1}^r s(z;a_i,b_i)^{k_i}$, let
        \begin{equation*}
    P_\Gamma(z)=\sum_{[\gamma]\in \widetilde{\Gamma}}(\det\gamma)^{l}\cdot \gamma^*s(z)
\end{equation*}
and
\begin{equation*}
    P_{\Gamma_0}(z)=\sum_{[\gamma]\in \widetilde{\Gamma_0}}(\det\gamma)^{l}\cdot \gamma^*s(z)
\end{equation*}
        denote the associated Poincaré series for the groups
            $\Gamma$ and $\Gamma_0$. Then
            \begin{equation*}
                \mathrm{Tr}_{\Gamma/\Gamma_0}(P_{\Gamma_0})=P_{\Gamma}.
            \end{equation*}
    \end{enumerate}
\end{lemma}

\begin{proof} Let $Z\coloneqq Z(\Gamma)=Z(\Gamma_0)$. Any $f\in M_{2k,m}(\Gamma)$ lies in $M_{2k,m}(\Gamma_0)$, and
    $\omega_f=f(z)(\mathrm{d}z)^k$ satisfies
    $\gamma_i^*\omega_f=(\det\gamma_i)^{-l}\omega_f$ for every $i$, so
    $\mathrm{Tr}_{\Gamma/\Gamma_0}(\omega_f)=n\,\omega_f$. Since $p\nmid n$, the integer $n$ is invertible in $\mathbb{C}_\infty$ and $\omega_f=\mathrm{Tr}_{\Gamma/\Gamma_0}(n^{-1}\omega_f)$, which proves (i).
    
    \noindent For (ii), since $\Gamma=\coprod_{i=1}^n\Gamma_0\gamma_i$ and
    $\Gamma_0=\coprod_{[\gamma]\in Z\backslash\Gamma_0}Z\gamma$, the products
    $\gamma\gamma_i$ form a system of representatives of $Z\backslash\Gamma$.
    Using $\gamma_i^*\gamma^*s=(\gamma\gamma_i)^*s$ we obtain
    \begin{equation*}
        \mathrm{Tr}_{\Gamma/\Gamma_0}(P_{\Gamma_0})
        =\sum_{i=1}^n\sum_{[\gamma]\in Z\backslash\Gamma_0}
        \det(\gamma\gamma_i)^{l}(\gamma\gamma_i)^*s
        =\sum_{[\delta]\in Z\backslash\Gamma}(\det\delta)^{l}\,\delta^*s
        =P_\Gamma
    \end{equation*}
    as wanted.
\end{proof}
We can now complete the proof of Theorem \ref{theoremPSspanQuat}.

\begin{proof}[Proof of Theorem \ref{theoremPSspanQuat}]
    Since $\Gamma_0$ has finite index in $\iota(\mathcal{O}^\times)$ and $Z(\Gamma_0)=Z(\Gamma)$, the integer $l$ satisfies the same condition relative to $\Gamma_0$, and the case already treated applies to $\Gamma_0$: every form in $M_{2k,m}(\Gamma_0)$ is a linear combination of Poincaré series of the form \eqref{equationQuaternionicPS} for $\Gamma_0$. Lemma
    \ref{lemmaTrace} shows that
    $\mathrm{Tr}_{\Gamma/\Gamma_0}\colon M_{2k,m}(\Gamma_0)\to M_{2k,m}(\Gamma)$ is surjective, and maps linear combinations of Poincaré series for the group $\Gamma_0$ to linear combinations of Poincaré series for $\Gamma$. Hence $M_{2k,m}(\Gamma)$ is
    spanned by Poincaré series of the form \eqref{equationQuaternionicPS}.
\end{proof}

\begin{example}
   Assume that $B$ ramifies at exactly two places $\mathfrak{p},\mathfrak{q}$
   with $\deg(\mathfrak{p})=2$ and $\deg(\mathfrak{q})=1$, and set
   $\Gamma\coloneqq\iota(\mathcal{O}^\times)$; this is the only case where $X_\Gamma$ is hyperelliptic, see \cite[Theorem 4.1]{Papikian2009}. A nontrivial torsion element of
   $\widetilde{\Gamma}$ would generate a quadratic subfield of $B$ of the form
   $\mathbb{F}_{q^2}F$, in which $\mathfrak{p}$ splits,
   contradicting the fact that a ramified place of $B$ cannot split in a
   subfield; hence $\widetilde{\Gamma}$ is torsion-free. The quotient graph
   $\Gamma\backslash\mathcal{T}$ consists of $V=2$ vertices connected by
   $E=q+1$ edges \cite[Corollary 5.9]{Papikian2012}, so $g=1-2+(q+1)=q$ by
   \eqref{equationQuatVE}. The first vertex of $\Gamma\backslash\mathcal{T}$ is
   $\overline{v_0}$, while the second one is $\overline{v_1}$.

   \noindent Take $k=2$. Since $\det(\Gamma)=\mathrm{Nr}(\mathcal{O}^\times)
   =\mathbb{F}_q^\times$, the types are indexed by
   $\mathbb{Z}/(q-1)$, and the condition
   $2l\equiv 0\mod{q-1}$ of Theorem \ref{theoremPSspanQuat} has
   exactly two solutions, $l=0$ and $l=(q-1)/2$; cf.\
   \cite[Lemma 2.2]{FHP2025}, where the analogous two-type
   decomposition is proved for $\Gamma_0(\mathfrak{n})$. By
   Proposition \ref{propDim}, $\dim M_{4,2+l}(\Gamma)=(4-1)(q-1)=3(q-1)$
   for both types.
   
   \noindent Consider the basis $\{\overline{s_\beta}\}_{\beta\in\mathbb{F}_q
   \backslash\{0,1\}}$ of $H^0(\mathbb{P}^1_{\mathbb{F}_q},\Omega^{2}(D))$ as in
   Example \ref{exampleGenerationGamma(T)}. Each $\overline{s_\beta}$ lifts to a
   Poincaré series
\begin{align*}
    P^{\overline{v_0}}_\beta(z)&\coloneqq \sum_{[\gamma]\in \widetilde{\Gamma}}
    (\det\gamma)^{l}\cdot\gamma^* s_\beta(z)\\
    &=\sum_{[\gamma]\in \widetilde{\Gamma}}(\det\gamma)^{-l}\cdot
    s(z;\gamma\infty,\gamma0)\cdot
    s(z;\gamma1,\gamma\beta).
\end{align*}
The matrix $\delta \coloneqq\begin{pmatrix} T & 0 \\ 0 & 1 \end{pmatrix}$ satisfies
$\delta v_0=v_1$, so the Poincaré series associated to the vertex
$\overline{v_1}$ are
\begin{align*}
    P^{\overline{v_1}}_\beta(z)&\coloneqq \sum_{[\gamma]\in \widetilde{\Gamma}}(\det\gamma)^{-l}\cdot
    s(z;\gamma\delta\infty,\gamma\delta0)\cdot
    s(z;\gamma\delta1,\gamma\delta\beta).\\
    & =\sum_{[\gamma]\in \widetilde{\Gamma}}(\det\gamma)^{-l}\cdot
    s(z;\gamma\infty,\gamma0)\cdot
    s(z;\gamma T,\gamma(\beta T))
\end{align*}
for $\beta\in\mathbb{F}_q\backslash\{0,1\}$.

\noindent Recall that $e_\alpha$ denotes the unique open edge emanating from $v_0$ toward
the end $\alpha\in\mathbb{P}^1(\mathbb{F}_q)\subset \mathbb{P}^1(F_\infty)$. The
edges of $\Gamma\backslash\mathcal{T}$ are the $\overline{e_\alpha}$ for
$\alpha\in\mathbb{P}^1(\mathbb{F}_q)$. The Poincaré series corresponding to
$\overline{e_\alpha}$ is
\begin{equation}\label{HPSatEdgesQuat}
   P^{\overline{e_\alpha}}(z)\coloneqq \begin{cases}
        \sum_{[\gamma]\in \widetilde{\Gamma}}(\det\gamma)^{-l}\cdot
        s(z;\gamma\infty,\gamma0)\cdot s(z;\gamma T,\gamma1)
        &\text{if } \alpha=\infty,\\[2pt]
        \sum_{[\gamma]\in \widetilde{\Gamma}}(\det\gamma)^{-l}\cdot
        s(z;\gamma\infty,\gamma\alpha)\cdot
        s(z;\gamma(\alpha+1),\gamma(\alpha+\pi))
        &\text{if } \alpha\in\mathbb{F}_q.
    \end{cases}
\end{equation}
The set $\{P^{\overline{v_0}}_\beta\}_{\beta\in\mathbb{F}_q
   \backslash\{0,1\}}\cup
\{P^{\overline{v_1}}_\beta\}_{\beta\in\mathbb{F}_q
   \backslash\{0,1\}}\cup
\{P^{\overline{e_\alpha}}\}_{\alpha\in\mathbb{P}^1(\mathbb{F}_q)}$ has
cardinality $\allowbreak 2(q-2)+(q+1)=3(q-1)$, and the corresponding modular forms constitute a
basis of $M_{4,2+l}(\Gamma)$.
\end{example}

\section{Existing cusp forms as Poincaré series}\label{sectionExisting}

In this section, we show that some Drinfeld cusp forms in the literature are
actually Poincaré series.

\subsection{Goss polynomials}\label{subsectionGoss}

For an 
$\mathbb{F}_q$-lattice $\Lambda\subset\mathbb{C}_\infty$, set
\begin{equation*}
    e_\Lambda(z)\coloneqq z\sideset{}{'}\prod_{b\in\Lambda}
    \Big(1-\frac{z}{b}\Big)
\end{equation*}
and $t_\Lambda(z)\coloneqq e_\Lambda(z)^{-1}$.
The function $e_\Lambda$ is entire, $\mathbb{F}_q$-linear and surjective
with kernel $\Lambda$, and
$t_\Lambda(z)=\sum_{b\in\Lambda}(z-b)^{-1}$. For $\Lambda=\mathbb{F}_q$ one has $t_{\mathbb{F}_q}(z)=(z-z^{q})^{-1}$, and
for $\Lambda=\mathfrak{b}$ as in \S\ref{sectionModularForms} we recover
the parameter $t^\Gamma=t_{\mathfrak{b}}$ at the cusp $\infty$. Note that $t^{G}= t_A$ and $t^{\Gamma(\mathfrak{n})}=t_{\mathfrak{n}A}$.

\noindent For every $n\geqslant 1$
there is a unique polynomial
$G_{n,\Lambda}(X)\in\mathbb{C}_\infty[X]$, the \textit{$n$-th Goss
polynomial} of $\Lambda$, such that
\begin{equation}\label{equationGoss}
    \sum_{b\in\Lambda}\frac{1}{(z-b)^{n}}
    =G_{n,\Lambda}\big(t_\Lambda(z)\big)
\end{equation}
for all $z\notin\Lambda$; see \cite[(3.4)]{Gekeler1988}. It satisfies
$G_{n,\Lambda}(X)=X^n$ whenever $1\leqslant n\leqslant q$. 

\subsection{Drinfeld--Poincaré series}\label{subsectionDrinfeldPoincare}

Let $k,n\geqslant 1$, and let $m$ be the image of $n$ in $\mathbb{Z}/(q-1)\mathbb{Z}$. Consider the  \textit{Drinfeld--Poincaré series}
\begin{equation*}
    P_{k,n}(z)\coloneqq \sum_{[\gamma]\in K\backslash G}\frac{\det(\gamma)^m}{j_{\gamma}(z)^{k}}G_{n,A}(t^{G}(\gamma z))
\end{equation*}
where $K\coloneqq \begin{pmatrix} \mathbb{F}_q^\times & A \\ 0 & 1
\end{pmatrix}$ and $G_{n,A}$ is the $n$-th Goss polynomial for the
lattice $A$.

\begin{remark}\label{remarkNormalization}
    The Drinfeld--Poincaré series $P_{k,n}$ were defined by Petrov in
    \cite[Definition 4.1]{Petrov2015}, using the parameter
    $u=\widetilde{\pi}^{-1}t^{G}$ and the Goss polynomials
    $G_{n,\widetilde{\pi}A}$ for the
    lattice $\widetilde{\pi}A$, where $\widetilde{\pi}$ is a
    fundamental period of the Carlitz module. Since
    $G_{n,\widetilde{\pi}A}(\widetilde{\pi}^{-1}X)
    =\widetilde{\pi}^{-n}G_{n,A}(X)$,
    our $P_{k,n}$ equals $\widetilde{\pi}^n$ times Petrov's.
    Petrov shows that $P_{k,n}$ lies in $M_{k,m}(G)$, and is nonzero
    if $k\equiv 2n\bmod{q-1}$ and $n\leqslant k/(q+1)$.
    If $n\leqslant q$, then $G_{n,A}(X)=X^n$, and our $P_{k,n}$ is
    $\widetilde{\pi}^n$ times Gekeler's
    \cite[(5.11)]{Gekeler1988}.
\end{remark}

\noindent We now relate the Poincaré series $\sum_{\gamma\in G}\gamma^* s(z;\infty,0)^{k}$ to $P_{2k,k}$:
\begin{align*}
    \sum_{\gamma\in G}\gamma^* s(z;\infty,0)^{k}&=\sum_{\gamma\in G} s(z;\gamma^{-1}\infty,\gamma^{-1}0)^{k}\\
    &=\sum_{\gamma\in G} s(z;-d/c,-b/a)^{k}\\
    &=\varrho(z)\,(\mathrm{d}z)^k
\end{align*}
with 
\begin{equation*}
        \varrho(z)\coloneqq \sum_{\gamma\in G}\frac{\det(\gamma)^k}{(az+b)^k(cz+d)^k}
\end{equation*}
Condition \textnormal{\textbf{(P)}} holds by Theorem \ref{theorem(P)}(iii), so $\varrho\in M_{2k,l}^1(G)$ by Corollary \ref{coroCuspForm}, where $l$ is the image of $k$ in $\mathbb{Z}/(q-1)\mathbb{Z}$.

\begin{proposition}\label{propDrinfeldPoincare}
    We have $\varrho=-P_{2k,k}$ for $k\geqslant 1$.
\end{proposition}
\begin{proof}
     We have
\begin{align*}
    \varrho(z)
    &=\sum_{\gamma\in G}\left(\frac{cz+d}{az+b}\frac{\det(\gamma)}{(cz+d)^2}\right)^k\\
    &=\sum_{\gamma\in G}\frac{\det(\gamma)^k}{(\gamma z)^kj_\gamma(z)^{2k}}\\
    &=\sum_{[\gamma]\in K\backslash G}\sum_{\gamma'\in K}\frac{\det(\gamma'\gamma)^k}{(\gamma'\gamma z)^{k}j_{\gamma'\gamma}(z)^{2k}}.
\end{align*}
Writing $\gamma'=\begin{pmatrix}
\alpha  & a \\
0 & 1
\end{pmatrix}$, one computes $\det(\gamma'\gamma) = \alpha\det(\gamma)$, $j_{\gamma'\gamma}(z) = j_{\gamma'}(\gamma z) j_\gamma(z)=j_\gamma(z)$, and $\gamma'\gamma z = \alpha\gamma z + a$. Therefore
\begin{align*}
    \varrho(z)&=\sum_{[\gamma]\in K\backslash G}\sum_{\gamma'\in K}\frac{\alpha^k\det(\gamma)^k}{(\alpha\gamma z+a)^{k} j_{\gamma}(z)^{2k}}\\
    &=\sum_{[\gamma]\in K\backslash G}\frac{\det(\gamma)^k}{j_{\gamma}(z)^{2k}}\sum_{\alpha\in \mathbb{F}_q^\times}\sum_{a\in A}\frac{1}{(\gamma z+\alpha^{-1} a)^{k}}\\
    &=-\sum_{[\gamma]\in K\backslash G}\frac{\det(\gamma)^k}{j_{\gamma}(z)^{2k}}\sum_{a\in A}\frac{1}{(\gamma z+a)^{k}}\\
    &=-P_{2k,k}(z)
\end{align*}
as claimed.
\end{proof}

\begin{remark}
    The equality $\varrho=-P_{2k,k}$ does not guarantee that $\varrho$
    is nonzero. Indeed, Petrov's nonvanishing criterion
    \cite[Theorem 4.2]{Petrov2015} requires $n\leqslant k/(q+1)$,
    which for $P_{2k,k}$ amounts to $k\leqslant 2k/(q+1)$. This condition is never satisfied, so the
    criterion does not apply. In fact, $\varrho$ does vanish in some cases: if $k\equiv 0\mod{q-1}$ and $2k<q^2-1$, then
    $M_{2k,0}^1(G)=0$ so $\varrho=0$.
\end{remark}

\subsection{A family of Poincaré series for
$\Gamma(\mathfrak{n})$}\label{subsectionFamily}

Let $\mathfrak{n}\in A$ be monic with $d=\deg(\mathfrak{n})\geqslant 1$,
and set $r\coloneqq (q+1)/2$; since $q\geqslant 3$ is odd, $r\geqslant 2$
is an integer. We fix an exponent $\kappa$ with
$1\leqslant\kappa\leqslant q-1$ and a decomposition
\begin{equation*}
    \mathbb{P}^1(\mathbb{F}_q)=\coprod_{i=1}^{r}\{a_i,b_i\}
\end{equation*}
into $r$ pairs. For $\gamma\in\Gamma(\mathfrak{n})$ put
\begin{equation}\label{equationPsiGamma}
    \Psi_\gamma(z)\coloneqq
    \prod_{i=1}^{r}u(z;\gamma a_i,\gamma b_i)^{\kappa},
\end{equation}
and consider
\begin{equation*}
    f=f_\kappa^{\mathfrak{n}}
    \coloneqq\sum_{\gamma\in\Gamma(\mathfrak{n})}\Psi_\gamma .
\end{equation*}
This is a Poincaré series with
$r\geqslant 2$ pairs, all exponents $k_i=\kappa$ and $H=\{I\}$;
thus $k=\kappa r$ and the weight is $2k=\kappa(q+1)$. As the pairs
$\{a_i,b_i\}$ are pairwise disjoint we have
$\bigcap_{i=1}^{r}\{a_i,b_i\}=\emptyset$, so condition
\textnormal{\textbf{(P)}} holds by Theorem \ref{theorem(P)}(ii), and
$f\in M_{\kappa(q+1)}^{1}(\Gamma(\mathfrak{n}))$ by Corollary
\ref{coroCuspForm}. Moreover $\bigcap_{i=1}^{r}[a_i\to b_i]=\{v_0\}$
and the stabilizer of $v_0$ in $\Gamma(\mathfrak{n})$ is trivial
(Lemma \ref{LemmaQuotientGraphGammaN}), so $f\neq 0$ by
Theorem \ref{theoremIntersectionVertex}.

Let $\Gamma_\infty$ denote the stabilizer of $\infty$ in
$\Gamma(\mathfrak{n})$; it consists of the unipotent
matrices $\begin{psmallmatrix}1&b\\0&1\end{psmallmatrix}$ with
$b\in\mathfrak{n}A$. We abbreviate
$t\coloneqq t^{\Gamma(\mathfrak{n})}=t_{\mathfrak{n}A}$ and write
$f(z)=\sum_{n\geqslant 1}c_n(f)\,t(z)^{n}$ for the $t$-expansion of $f$
at $\infty$ as in \eqref{equationTExpansion}, valid for $|z|_i$ large
enough. Let
\begin{equation*}
    \varsigma\coloneqq
    \begin{cases}
        1  &\text{if }\infty\in\{a_1,\ldots,a_r\}\\[2pt]
        -1  &\text{if } \infty\in\{b_1,\ldots,b_r\},
    \end{cases}
\end{equation*}
and 
\begin{equation*}
    C_0\coloneqq \varsigma\prod_{\substack{1\leqslant i\leqslant r\\
    \infty\not\in\{a_i,b_i\}}}(b_i-a_i)\in\mathbb{F}_{q}^\times.
\end{equation*}
Finally set $\varepsilon\coloneqq e_{\mathfrak{n}A}(1)\in\mathbb{C}_\infty$; note that $\varepsilon\neq 0$ since
$1\not\in \mathfrak{n}A=\ker e_{\mathfrak{n}A}$.

In general, determining the order of vanishing at $\infty$ of a given Poincaré series is a difficult problem. However, it turns out that the order of vanishing at $\infty$ of $f=f_\kappa^{\mathfrak{n}}$ can be computed. The rest of this subsection is devoted to the proof of the following result.

\begin{theorem}\label{theoremFisC0h}
    For $1\leqslant n\leqslant \kappa( q+1)-1$, we have 
    \begin{equation*}
    c_n(f)=
    \begin{cases}
        C_0^\kappa\,\varepsilon^{\kappa(q-1)} &\text{if } n=\kappa q,\\
        0 &\text{otherwise}.
    \end{cases}
    \end{equation*}
    In other words, the $t$-expansion of $f$ at the cusp $\infty$ satisfies
    \begin{equation*}
        f(z)=C_0^\kappa\,\varepsilon^{\kappa(q-1)}\,t(z)^{\kappa q}+O(t(z)^{\kappa( q+1)}).
    \end{equation*}
    In particular, $f$ vanishes to order exactly $\kappa q$ at every cusp of $\Gamma(\mathfrak{n})$ represented by an element of $\mathbb{P}^1(\mathbb{F}_q)$.
\end{theorem}

\begin{remark}\label{remarkConjectureCase}
    Theorem \ref{theoremFisC0h} proves Conjecture
    \ref{conjectureVanishingExact} for the
    Poincaré series $f_\kappa^{\mathfrak{n}}$. Indeed the weight is
    $2k$ with $k=\kappa r$, and every
    $\alpha\in\mathbb{P}^1(\mathbb{F}_q)$ lies in exactly one pair
    $\{a_i,b_i\}$, so $p_\alpha=\kappa$ and  $2k-p_\alpha=\kappa(q+1)-\kappa=\kappa q$ for every $\alpha$, which is precisely the order of vanishing given
    by Theorem \ref{theoremFisC0h}.
\end{remark}

Recall that the map $\gamma=\begin{pmatrix}a&b\\c&d\end{pmatrix}\mapsto (c,d)$ induces a bijection between $\Gamma_\infty\backslash\Gamma(\mathfrak{n})$ and
    \begin{equation*}
        \{(c,d):c\in \mathfrak{n}A, d\in 1+\mathfrak{n}A,\gcd(c,d)=1\}.
    \end{equation*}
The trivial coset $\Gamma_\infty$ is represented by $(0,1)$, while the other cosets are represented by some pair $(c,d)$ with $c\neq 0$.

\noindent We will use repeatedly the following (easy) fact: for $j\geqslant 0$, 
\begin{equation}\label{equationPowerSums}
    \sum_{x\in\mathbb{F}_{q}}x^{\,j}=
    \begin{cases}
        -1 &\text{if } j\neq 0 \text{ and } (q-1)\mid j,\\
        0 &\text{otherwise}.
    \end{cases}
\end{equation}

\subsubsection{Binomial coefficients modulo $p$}

Recall the convention $\binom{a}{b}=0$ for $a<b$.

\begin{lemma}[Lucas' theorem in base $q$]\label{lemmaLucasBaseQ}
    Let $m,n$ be non-negative integers, and let
    \begin{equation*}
        m=\sum_{i\geqslant 0}m_iq^{i},\qquad n=\sum_{i\geqslant 0}n_iq^{i},
        \qquad 0\leqslant m_i,n_i<q,
    \end{equation*}
    be their base-$q$ expansions (with $m_i=n_i=0$ for $i$ large enough). Then
    \begin{equation}\label{equationLucasBaseQ}
        \binom{m}{n}\equiv\prod_{i\geqslant 0}\binom{m_i}{n_i} \bmod p.
    \end{equation}
    In particular if $n_i>m_i$ for some $i$, then $\binom{m}{n}\equiv 0 \bmod p$.
\end{lemma}

\begin{proof}
    For $q=p$, this is \cite[Theorem 1]{Fine1947}. The general case
    follows by writing $q=p^{e}$ and applying it in base $p$, the
    $e$ factors attached to a given $i$ recombining into
    $\binom{m_i}{n_i}$.
\end{proof}

\begin{corollary}\label{corLucas}
    Let $1\leqslant\kappa\leqslant q-1$ and let $n,\ell$ be integers with $\kappa\leqslant n\leqslant\kappa(q+1)-1$, $\ell\geqslant 1$ and $\ell(q-1)<n-\kappa$.
    Then
    \begin{equation*}
        \binom{n-1}{\kappa-1}\binom{n-\kappa}{\ell(q-1)}\equiv 0 \bmod p .
    \end{equation*}
\end{corollary}
\begin{proof} 
    Assume both binomial coefficients $\binom{n-1}{\kappa-1}$ and $\binom{n-\kappa}{\ell(q-1)}$ are
    non-zero modulo $p$; we derive a contradiction.

    \noindent We have $n-1\leqslant \kappa(q+1)-2\leqslant q^2-3<q^{2}$, so we can write $n-1=Aq+B$ with $0\leqslant A,B<q$. We have $\binom{n-1}{\kappa-1}\equiv\binom{A}{0}\binom{B}{\kappa-1}$ by Lemma \ref{lemmaLucasBaseQ}, and $\binom{n-1}{\kappa-1}\not\equiv 0$ forces
    $\kappa-1\leqslant B$. In particular, the base-$q$ expansion of $n-\kappa=(n-1)-(\kappa-1)$ is $Aq+(B-\kappa+1)$.
    Next, $n-1\leqslant\kappa(q+1)-2$ excludes $A\geqslant \kappa$: for $A\geqslant \kappa+1$ this would give $n-1\geqslant(\kappa+1)q$, and for $A=\kappa$ it would give $B\leqslant \kappa-2$, against $\kappa-1\leqslant B$. Hence $A\leqslant \kappa-1\leqslant q-2$, and $n-\kappa=Aq+(B-\kappa+1)\leqslant(q-2)q+(q-1)$. The condition $\ell(q-1)<n-\kappa$ thus forces $\ell\leqslant q$. The base-$q$ expansion of $\ell(q-1)$ is therefore $(\ell-1)q+(q-\ell)$, and $\binom{n-\kappa}{\ell(q-1)}\equiv\binom{A}{\ell-1}\binom{B-\kappa+1}{q-\ell}$. Hence $\ell-1\leqslant A$ and $q-\ell\leqslant B-\kappa+1$.
    
    \noindent Combining, $q-(B-\kappa+1)\leqslant \ell\leqslant A+1\leqslant \kappa$, so $B=q-1$,
    $\ell=\kappa$ and $A=\kappa-1$; but then $\ell(q-1)=\kappa(q-1)=(n-1)-(\kappa-1)=n-\kappa$, contradicting $\ell(q-1)<n-\kappa$.
\end{proof}

\subsubsection{Proof of Theorem \ref{theoremFisC0h}}

Our first task is to determine the partial fraction decomposition of $\Psi_\gamma(z)$. We will need the following lemma.

\begin{lemma}\label{lemmaKey}
    Let $Q(z)=\lambda z^{q}+\mu$ with
    $\lambda,\mu\in\mathbb{C}_\infty$. For $1\leqslant\kappa\leqslant q-1$, we have
    \begin{equation}\label{eqKey}
        \frac{Q(z)^{\kappa}}{(z^{q}-z)^{\kappa}}
        =(-1)^{\kappa}\sum_{x\in\mathbb{F}_q}\frac{Q(x)^{\kappa}}{(z-x)^{\kappa}}
        +\lambda^{\kappa}.
    \end{equation}
\end{lemma}

\begin{proof}
    Let $M\coloneqq Q(z)-\lambda(z^{q}-z)$; using $x^q=x$, one checks that $Q(x)=M-\lambda(z-x)$ for all $x\in\mathbb{F}_q$. Therefore
    \begin{equation*}
        \sum_{x\in\mathbb{F}_q}\frac{Q(x)^{\kappa}}{(z-x)^{\kappa}}
        =\sum_{\ell=0}^{\kappa}\binom{\kappa}{\ell}M^{\kappa-\ell}(-\lambda)^{\ell}
        \sum_{x\in\mathbb{F}_q}\frac{1}{(z-x)^{\kappa-\ell}} .
    \end{equation*}
    For $\ell=\kappa$ the inner sum is $0$. For $\ell<\kappa$ we have
    $1\leqslant\kappa-\ell\leqslant q-1$, so
    $G_{\kappa-\ell,\mathbb{F}_q}(X)=X^{\kappa-\ell}$ and the inner sum is
    \begin{equation*}
        \sum_{x\in\mathbb{F}_q}\frac{1}{(z-x)^{\kappa-\ell}}
        =t_{\mathbb{F}_q}(z)^{\kappa-\ell}
        =\frac{1}{(z-z^q)^{\kappa-\ell}}.
    \end{equation*}
    Hence
    \begin{align*}
        \sum_{x\in\mathbb{F}_q}\frac{Q(x)^{\kappa}}{(z-x)^{\kappa}}&=\frac{1}{(z-z^{q})^{\kappa}}
        \sum_{\ell=0}^{\kappa-1}\binom{\kappa}{\ell}M^{\kappa-\ell}\lambda^{\ell}(z^{q}-z)^{\ell}\\
        &=\frac{(-1)^\kappa}{(z^{q}-z)^{\kappa}}\left(\big(M+\lambda(z^{q}-z)\big)^\kappa-\lambda^\kappa(z^{q}-z)^{\kappa}\right)\\
        &=(-1)^\kappa\left(\frac{Q(z)^{\kappa}}{(z^{q}-z)^{\kappa}}-\lambda^\kappa\right),
    \end{align*}
    which is equivalent to \eqref{eqKey}.
\end{proof}

Recall that $j_\gamma(z)$ is defined as $cz+d$ for $\gamma=\begin{pmatrix}a&b\\c&d\end{pmatrix}\in\GL_2(F_\infty)$ and $z\in\mathbb{C}_\infty$. We extend\footnote{Using projective coordinates, $j_\gamma(z)=cz+d$ can be thought as the value at $(z:1)$ of the homogeneous form $j_\gamma(x:y)=cx+dy$. Note that $c$ is the value at $\infty=(1:0)$ of $j_\gamma(x:y)$.} $j_\gamma$ to a function $\mathbb{P}^1(\mathbb{C}_\infty)\to \mathbb{C}_\infty$ by setting $j_\gamma(\infty)\coloneqq c$. 

\begin{proposition}\label{lemmaPartialFractionDec}
    The partial fraction decomposition of $\Psi_\gamma(z)$ is
    \begin{equation}\label{eqLemmaPartial}
    \Psi_\gamma(z)=\sum_{\substack{x\in\mathbb{P}^1(\mathbb{F}_q)\\\gamma x\neq\infty}}\frac{\rho_x(\gamma)}{(z-\gamma x)^\kappa}
\end{equation}
with
\begin{equation*}
    \rho_x(\gamma)\coloneqq (-1)^\kappa C_0^\kappa\,j_\gamma(x)^{\,\kappa(q-1)}.
\end{equation*}
In particular, 
\begin{equation}\label{eqSumJGamma}
    \sum_{\substack{x\in\mathbb{P}^1(\mathbb{F}_q)\\\gamma x\neq\infty}}j_\gamma(x)^{\,\kappa(q-1)}=0.
\end{equation}
\end{proposition}

\begin{proof}
    Exactly one of the $2r=q+1$ parameters $a_1,b_1,\ldots,a_r,b_r$ equals $\infty$. Therefore
    \begin{align*}
        \Psi_I(z)&=\frac{C_0^\kappa}{\prod_{x\in\mathbb{F}_{q}}(z-x)^\kappa}\\
        &=\frac{C_0^\kappa}{(z^q-z)^\kappa}.
    \end{align*}
Let $w=\gamma^{-1}z\in\Omega$. Applying the formula
\begin{equation}\label{eqnulle}
    u(\gamma w;\gamma a,\gamma b)=\frac{j_\gamma(w)^2}{\det\gamma}u(w;a,b)
\end{equation}
to each of the $r$ factors of $\Psi_\gamma$ and using $\det\gamma=1$ gives
    \begin{align*}
        \Psi_\gamma(z)&=\Psi_\gamma(\gamma w)\\
        &=j_\gamma(w)^{\,\kappa(q+1)}\Psi_I(w)\\
        &=C_0^{\kappa}\,j_\gamma(w)^{\kappa}\,\frac{j_\gamma(w)^{\kappa q}}{(w^q-w)^\kappa}.
    \end{align*}
    As $j_\gamma(w)^{q}=c^{q}w^{q}+d^{q}$, we can apply Lemma \ref{lemmaKey} with $\lambda=c^{q}$ and $\mu=d^q$:
    \begin{align*}
\Psi_\gamma(z)&=C_0^{\kappa}\,j_\gamma(w)^{\kappa}\,\left((-1)^{\kappa}\sum_{x\in\mathbb{F}_q}\frac{j_\gamma(x)^{\kappa q}}{(w-x)^{\kappa}}
        +c^{\kappa q}\right)\\
        &=(-1)^{\kappa}C_0^{\kappa}\sum_{x\in\mathbb{F}_q}
        \frac{j_\gamma(w)^{\kappa}j_\gamma(x)^{\kappa q}}{(w-x)^{\kappa}}
        +C_0^{\kappa}c^{\kappa q}j_\gamma(w)^{\kappa}.
    \end{align*}
    For $x\in\mathbb{F}_q$, we have $w-x=(z-\gamma x)\,j_\gamma(w)\,j_\gamma(x)$; thus
    \begin{align*}
        \Psi_\gamma(z)&=(-1)^{\kappa}C_0^{\kappa}\sum_{x\in\mathbb{F}_q}
        \,\frac{j_\gamma(x)^{\kappa (q-1)}}{(z-\gamma x)^{\kappa}}
        +C_0^{\kappa}c^{\kappa q}j_\gamma(w)^{\kappa}\\
        &=\sum_{x\in\mathbb{F}_q}
        \,\frac{\rho_x(\gamma)}{(z-\gamma x)^{\kappa}}
        +C_0^{\kappa}c^{\kappa q}j_\gamma(w)^{\kappa}.
    \end{align*}
    If $c=0$ the last term vanishes, and the term $x=\infty$ in \eqref{eqLemmaPartial} is excluded since $\gamma\infty=\infty$.
    If $c\neq 0$, then
    $z-\gamma\infty=-\big(c\,j_\gamma(w)\big)^{-1}$, so the last term equals
    \begin{equation*}
        C_0^{\kappa}c^{\kappa(q-1)}\frac{(-1)^{\kappa}}{(z-\gamma\infty)^{\kappa}}=\frac{\rho_\infty(\gamma)}{(z-\gamma\infty)^{\kappa}}.
    \end{equation*}
    In both cases, \eqref{eqLemmaPartial} is satisfied.

    \noindent As $|z|\to\infty$, a factor $u(z;a,b)^{\kappa}$ is $O(z^{-2\kappa})$ when
    $a$ and $b$ are both finite, and $O(z^{-\kappa})$ when one of them is
    $\infty$. At most one of the $q+1$ points $\gamma a_1,\gamma b_1,\ldots,\gamma a_r,\gamma b_r$ equals $\infty$, so at most one of
    the $r$ factors of \eqref{equationPsiGamma} is of the second kind, and
    \begin{equation*}
        \Psi_\gamma(z)=O\big(z^{-2\kappa(r-1)-\kappa}\big)
        =O\big(z^{-\kappa q}\big).
    \end{equation*}
    On the other hand, for $|z|$ large enough, by
    \eqref{eqLemmaPartial},
    \begin{align*}
        z^{\kappa}\Psi_\gamma(z)
        &=\sum_{\substack{x\in\mathbb{P}^1(\mathbb{F}_q)\\\gamma x\neq\infty}}
        \rho_x(\gamma)\left(\frac{z}{z-\gamma x}\right)^{\kappa}\\
        &=\sum_{\substack{x\in\mathbb{P}^1(\mathbb{F}_q)\\\gamma x\neq\infty}}
        \rho_x(\gamma)\sum_{n\geqslant 0}\binom{n+\kappa-1}{n}
        \left(\frac{\gamma x}{z}\right)^{n}\\
        &=\sum_{n\geqslant 0}
        \binom{n+\kappa-1}{n}\left(\sum_{\substack{x\in\mathbb{P}^1(\mathbb{F}_q)\\\gamma x\neq\infty}}
        \rho_x(\gamma)\,(\gamma x)^{n}\right)z^{-n}.
        \end{align*}
    Since $z^{\kappa}\Psi_\gamma(z)=O\big(z^{-\kappa(q-1)}\big)$, the coefficients of $z^{-n}$ vanish for
    $n<\kappa(q-1)$. In particular the coefficient of $z^{0}$ vanishes:
    \begin{equation*}
        \sum_{\substack{x\in\mathbb{P}^1(\mathbb{F}_q)\\\gamma x\neq\infty}}\rho_x(\gamma)=(-1)^\kappa C_0^\kappa\sum_{\substack{x\in\mathbb{P}^1(\mathbb{F}_q)\\\gamma x\neq\infty}}j_\gamma(x)^{\,\kappa(q-1)}=0,
    \end{equation*}
    which implies \eqref{eqSumJGamma}.
\end{proof}

We now express $f$ as a series over $\Gamma_\infty\backslash\Gamma(\mathfrak{n})$ of $\mathfrak{n}A$-periodic functions. We will need the following version of the maximum modulus principle.

\begin{lemma}\label{lemmaGaussNorm}
    Let $\tau\in q^\mathbb{Q}=|\mathbb{C}_\infty^\times|$. If a power series $S(t)=\sum_{n\geqslant 0}a_nt^{n}$ with $a_n\in \mathbb{C}_\infty$ converges on the disc $\{t\in \mathbb{C}_\infty:|t|\leqslant\tau\}$, then
    \begin{equation*}
        \sup_{|t|=\tau}|S(t)|=\sup_{|t|\leqslant\tau}|S(t)|=\max_{n\geqslant 0}|a_n|\tau^{n}.
    \end{equation*}
\end{lemma}

\begin{proof}
    The inequalities $\sup_{|t|=\tau}|S(t)|\leqslant \sup_{|t|\leqslant\tau}|S(t)|\leqslant \max_{n\geqslant 0}|a_n|\tau^{n}$ are clear. Let $\mu\coloneqq\max_n|a_n|\tau^{n}\geqslant 0$; we now show $\sup_{|t|=\tau}|S(t)|\geqslant \mu$. 
    We are going to exhibit $t_1\in \mathbb{C}_\infty$ such that $|S(t_1)|=\mu$ and $|t_1|=\tau$. If $\mu=0$, then $a_n=0$ for all $n\geqslant 0$ and there is nothing to prove. 
    Suppose $\mu>0$. Since $|a_n|\tau^n\to 0$, the set $\{n\geqslant 0: |a_n|\tau^n=\mu\}$ is nonempty and finite; write $\mu=|a_m|\tau^m$ for some $m\geqslant 0$. Fix $t_0\in \mathbb{C}_\infty^\times$ with $|t_0|=\tau$. 
    For $n\geqslant 0$, set
    \begin{equation*}
        \alpha_n\coloneqq\frac{a_nt_0^{\,n}}{a_mt_0^{\,m}}=a_na_m^{-1}t_0^{n-m}\in \mathbb{C}_\infty.
    \end{equation*}
    We have $|\alpha_n|\leqslant 1$, with equality if and only if $|a_n|\tau^{n}=\mu$. Therefore
    \begin{equation*}
        P(x)\coloneqq
        \sum_{\substack{n\geqslant 0\\|a_n|\tau^{n}=\mu}}\overline{\alpha_n}\,x^{n}
    \end{equation*}
    is a nonzero polynomial in $\overline{\mathbb{F}_q}[x]$. Let  $\zeta\in R^\times$ with $P(\overline{\zeta})\neq 0$.
    Taking $t_1\coloneqq\zeta t_0$, so that $|t_1|=\tau$, we get
    \begin{equation*}
        \frac{S(t_1)}{a_mt_0^{\,m}}=\sum_{n\geqslant 0}\alpha_n\zeta^{n}.
    \end{equation*}
    Note that $\alpha_n\zeta^{n}\in R$ for all $n\geqslant 0$; the reduction of $S(t_1)/(a_mt_0^{\,m})$ modulo $\mathfrak{m}$ is $P(\overline{\zeta})\neq 0$. Hence $|S(t_1)/(a_mt_0^{\,m})|=1$, that is $|S(t_1)|=\mu$. 
\end{proof}

\begin{proposition}\label{propC2Decomposition} Let $\gamma\in \Gamma(\mathfrak{n})$.
    \begin{enumerate}
        \item[\textnormal{(i)}] The function $S_\gamma(z)\coloneqq\sum_{b\in \mathfrak{n}A}\Psi_\gamma(z-b)$ only depends on the coset $[\gamma]\in \Gamma_\infty\backslash\Gamma(\mathfrak{n})$, and
            \begin{equation*}
                f(z)=\sum_{[\gamma]\in\Gamma_\infty\backslash\Gamma(\mathfrak{n})}
                S_\gamma(z).
            \end{equation*}
        \item[\textnormal{(ii)}] For $|z|_i$ large
            enough, we have
            \begin{equation*}
                S_\gamma(z)=\sum_{n\geqslant 1} c_n(\gamma)\,t(z)^{n}
            \end{equation*}
            with 
            \begin{equation}\label{eqCnGamma}
                c_n(\gamma)\coloneqq\binom{n-1}{\kappa-1}\sum_{\substack{x\in\mathbb{P}^1(\mathbb{F}_q)\\ \gamma x\neq\infty}}\rho_x(\gamma)\,e_{\mathfrak{n}A}(\gamma x)^{n-\kappa}.
            \end{equation}
            In particular, 
            \begin{equation}\label{eqCn(f)}
            c_n(f)=\sum_{[\gamma]\in\Gamma_\infty\backslash\Gamma(\mathfrak{n})}c_n(\gamma).
            \end{equation}
    \end{enumerate}
\end{proposition}
\noindent Note that $c_n(\gamma)$ depends only on $[\gamma]$, since $e_{\mathfrak{n}A}(\gamma'\gamma x)=e_{\mathfrak{n}A}(\gamma x+b)=e_{\mathfrak{n}A}(\gamma x)$ and $\rho_x(\gamma'\gamma)=\rho_x(\gamma)$ for
    $\gamma'=\begin{pmatrix}1&b\\0&1\end{pmatrix}\in\Gamma_\infty$.

\begin{proof}
    Equation \eqref{eqnulle} implies $\Psi_{\gamma'\gamma}(z)= \Psi_\gamma(z-b)$ for $\gamma'=\begin{pmatrix}1&b\\0&1\end{pmatrix}\in \Gamma_\infty$.
    In particular, $S_\gamma(z)=\sum_{\gamma'\in \Gamma_\infty}\Psi_{\gamma'\gamma}(z)$ satisfies $S_{\delta\gamma}(z)=S_\gamma(z)$ for all $\delta\in\Gamma_\infty$. Moreover,
    \begin{equation*}
        \sum_{[\gamma]\in\Gamma_\infty\backslash\Gamma(\mathfrak{n})}
        S_\gamma(z)=\sum_{[\gamma]\in\Gamma_\infty\backslash\Gamma(\mathfrak{n})}\sum_{\gamma'\in \Gamma_\infty}\Psi_{\gamma'\gamma}(z)=f(z)
    \end{equation*}
    which proves (i).

    \noindent For (ii), since $\kappa\leqslant q-1$, we have
    $G_{\kappa,\mathfrak{n}A}(X)=X^{\kappa}$, i.e.,
    $\sum_{b\in\mathfrak{n}A}(z-b)^{-\kappa}=t(z)^{\kappa}$. Therefore
    \begin{equation*}
        S_\gamma(z)= \sum_{b\in \mathfrak{n}A} \sum_{\substack{x\in\mathbb{P}^1(\mathbb{F}_q)\\ \gamma x\neq\infty}} \frac{\rho_x(\gamma)}{(z-b-\gamma x)^\kappa}
        =\sum_{\substack{x\in\mathbb{P}^1(\mathbb{F}_q)\\ \gamma x\neq\infty}}\rho_x(\gamma)\,t(z-\gamma x)^\kappa.
    \end{equation*}

    \noindent Let
    $R\in q^{\mathbb{Q}}$ with $R> q^{\,d-1}$. We first bound $t$ on the region $\{|z|_i\geqslant R\}$. On the one hand,
    $|e_{\mathfrak{n}A}(y)|\leqslant q^{\,d-1}$ for every $y\in F_\infty$:
    writing $y/\mathfrak{n}=a+w$ with $a\in A$ and $|w|\leqslant q^{-1}$, we
    have $y=\mathfrak{n}a+v$ with $|v|=|\mathfrak{n}w|\leqslant q^{\,d-1}$
    and $e_{\mathfrak{n}A}(y)=e_{\mathfrak{n}A}(v)$ by periodicity, while
    every nonzero $b\in\mathfrak{n}A$ satisfies $|b|\geqslant q^{d}>|v|$,
    hence $|1-v/b|=1$ and
    $|e_{\mathfrak{n}A}(v)|=|v|\leqslant q^{\,d-1}$. On the other hand, let
    $z\in\Omega$ with
    $|z|_i\geqslant R$. Choose $y\in F_\infty$ with $|z-y|=|z|_i$ and set
    $u\coloneqq z-y$, so that $|u|=|u|_i=|z|_i$. For nonzero
    $b\in\mathfrak{n}A\subset F_\infty$ we have $|u-b|\geqslant|u|_i=|u|$,
    so $|1-u/b|=|b-u|/|b|\geqslant 1$ if $|b|\leqslant|u|$, while
    $|1-u/b|=1$ if $|b|>|u|$; hence
    $|e_{\mathfrak{n}A}(u)|\geqslant|u|\geqslant R$. As
    $|e_{\mathfrak{n}A}(y)|\leqslant q^{\,d-1}<R$ and
    $e_{\mathfrak{n}A}(z)=e_{\mathfrak{n}A}(u)+e_{\mathfrak{n}A}(y)$, we conclude $|e_{\mathfrak{n}A}(z)|=|e_{\mathfrak{n}A}(u)|\geqslant R$, that is $|t(z)|\leqslant R^{-1}$.

    \noindent Since
    $\gamma x\in F\subset F_\infty$ whenever $\gamma x\neq\infty$, we get
    $|e_{\mathfrak{n}A}(\gamma x)\,t(z)|\leqslant q^{\,d-1}/R<1$ for $|z|_i\geqslant R$, so the
    geometric series below converges and
    \begin{align*}
        t(z-\gamma x)^\kappa&= \frac{1}{(e_{\mathfrak{n}A}(z)-e_{\mathfrak{n}A}(\gamma x))^\kappa}\\
        &= \frac{t(z)^\kappa}{(1-e_{\mathfrak{n}A}(\gamma x)t(z))^\kappa}\\
        &= \sum_{n\geqslant 1}\binom{n-1}{\kappa-1}\,e_{\mathfrak{n}A}(\gamma x)^{n-\kappa}t(z)^n .
    \end{align*}
    Summing over $x$ gives
    $S_\gamma(z)=\sum_{n\geqslant 1}c_n(\gamma)\,t(z)^{n}$ on
    $\{|z|_i\geqslant R\}$, as wanted. Finally,
    \begin{align*}
        f(z)&=\sum_{[\gamma]\in\Gamma_\infty\backslash\Gamma(\mathfrak{n})}
                S_\gamma(z)
            =\sum_{n\geqslant 1}\left(\sum_{[\gamma]\in\Gamma_\infty\backslash\Gamma(\mathfrak{n})}
                c_n(\gamma)\right)t(z)^n,
    \end{align*}
    which proves \eqref{eqCn(f)}. It remains to justify the interchange of
    the two sums, for which we show that the family
    $\big(c_n(\gamma)\,t(z)^{n}\big)_{n,[\gamma]}$ is summable on
    $\{|z|_i\geqslant R\}$. As each $S_\gamma$ converges there, this amounts
    to showing that $\sup_{n\geqslant 1}|c_n(\gamma)\,t(z)^{n}|\to 0$ as
    $[\gamma]$ ranges over $\Gamma_\infty\backslash\Gamma(\mathfrak{n})$.

    \noindent Write $\gamma=\begin{pmatrix}a&b\\c&d\end{pmatrix}$ and let
    $|z|_i\geqslant R$. Applying Proposition \ref{propositionValuationS} to
    each factor of $\Psi_\gamma(z)(\mathrm{d}z)^{\kappa r}$ and using
    \eqref{equationOrdKForm},
    \begin{equation*}
        \ord\Psi_\gamma(z)
        =\kappa\sum_{i=1}^{r}
        d\big(\lambda(z),[\gamma a_i\to\gamma b_i]\big)
        +\kappa r\log_q|z|_i.
    \end{equation*}
    As $\bigcap_{i}[a_i\to b_i]=\{v_0\}$, we have
    $\bigcap_{i}[\gamma a_i\to\gamma b_i]=\{\gamma v_0\}$, so
    Lemma \ref{lemmaIntersectionConvex} gives
    \begin{equation*}
        \sum_{i=1}^{r}d\big(\lambda(z),[\gamma a_i\to\gamma b_i]\big)
        \geqslant\max_{1\leqslant i\leqslant r}
        d\big(\lambda(z),[\gamma a_i\to\gamma b_i]\big)
        =d\big(\lambda(z),\gamma v_0\big).
    \end{equation*}
    By Lemma \ref{lemmaPhiP}, $\log_q|\cdot|_i$ is affine of slope $\pm1$
    along the edges of $\mathcal{T}(\mathbb{R})$, hence $1$-Lipschitz, so
    $d(\lambda(z),\gamma v_0)\geqslant\log_q|z|_i-\log_q|z'|_i$ for any
    $z'\in\Omega_{\gamma v_0}$. For $\zeta\in\Omega_{v_0}$ one has
    $|\zeta|=|\zeta|_i=1$ and $|j_\gamma(\zeta)|=\max(|c|,|d|)$, so taking
    $z'=\gamma\zeta$ and using \eqref{equationImaginaryPart} gives
    $\log_q|z'|_i=-2\log_q\max(|c|,|d|)$. Combining these with
    $\log_q|z|_i\geqslant\log_q R$, we obtain
    \begin{equation*}
        \ord\Psi_\gamma(z)\geqslant
        2\kappa\log_q\max(|c|,|d|)+\kappa(r+1)\log_q R .
    \end{equation*}
    As $|z-b'|_i=|z|_i$ for $b'\in\mathfrak{n}A$, the same bound holds for
    each $\Psi_\gamma(z-b')$, so the ultrametric inequality applied to
    $S_\gamma(z)=\sum_{b'\in\mathfrak{n}A}\Psi_\gamma(z-b')$ gives
    \begin{equation*}
        B_\gamma\coloneqq\sup_{|z|_i\geqslant R}|S_\gamma(z)|
        \leqslant \frac{R^{-\kappa(r+1)}}{\max(|c|,|d|)^{2\kappa}}.
    \end{equation*}
    For any $M>0$ only finitely many pairs $(c,d)$ with $c\in\mathfrak{n}A$
    and $d\in 1+\mathfrak{n}A$ satisfy $\max(|c|,|d|)\leqslant M$, so
    $B_\gamma\to 0$ as $[\gamma]$ varies.

    \noindent Enlarging $R$ if necessary, we may assume by
    \cite[(2.7.3)]{GekelerReversat1996} that $t$ maps $\{|z|_i\geqslant R\}$
    onto a pointed disc $\{0<|t|\leqslant\tau\}$; by the above,
    $\tau\leqslant R^{-1}<q^{\,1-d}$. The power series
    $\sum_{n\geqslant 1}c_n(\gamma)t^{n}$ then converges on
    $\{|t|\leqslant\tau\}$, since $|c_n(\gamma)|\leqslant
    (\max_x|\rho_x(\gamma)|)\cdot \,q^{\,(d-1)(n-\kappa)}$ and $q^{\,d-1}\tau<1$,
    and its supremum on $\{|t|=\tau\}$ equals $B_\gamma$. Applying 
    Lemma \ref{lemmaGaussNorm} with $S=S_\gamma$ we get
    $|c_n(\gamma)|\,\tau^{n}\leqslant B_\gamma$ for every $n\geqslant 1$, hence
    $|c_n(\gamma)\,t(z)^{n}|\leqslant B_\gamma\big(|t(z)|/\tau\big)^{n}
    \leqslant B_\gamma$ for $|z|_i\geqslant R$. Therefore
    \begin{equation*}
        \sup_{n\geqslant 1}|c_n(\gamma)t(z)^{n}|\leqslant B_\gamma\to 0,
    \end{equation*}
    as wanted.
\end{proof}

The non-trivial cosets of $\Gamma_\infty\backslash\Gamma(\mathfrak{n})$ are parametrized by the pairs $(c,d)$ with $c\in \mathfrak{n}A\backslash\{0\}$, $d\in 1+\mathfrak{n}A$, and $\gcd(c,d)=1$. Equation \eqref{eqCn(f)} then becomes
    \begin{equation}\label{equationcnf2}
        c_n(f)=c_n(I)+
                \sum_{\substack{c\in \mathfrak{n}A\\ c\neq 0}}\sum_{d\in D_c}c_n(\gamma),
    \end{equation}
    where $\gamma$ is any representative of the coset $(c,d)$ and
    \begin{equation*}
        D_c\coloneqq\{d\in 1+\mathfrak{n}A:\gcd(c,d)=1\}.
    \end{equation*}

\begin{lemma}\label{lemmaTrivialCosetH}
    We have, for every $n\geqslant 1$,
    \begin{equation}\label{equationCnI}
        c_n(I)=
        \begin{cases}
            (-1)^{\kappa-1}\dbinom{n-1}{\kappa-1}C_0^{\kappa}\,
            \varepsilon^{\,n-\kappa}
            &\text{if } n>\kappa \text{ and } (q-1)\mid(n-\kappa),\\[6pt]
            0 &\text{otherwise}.
        \end{cases}
    \end{equation}
    In particular $c_n(I)=0$ for $1\leqslant n\leqslant\kappa(q+1)-1$ except for $n=\kappa q$, for which 
    \begin{equation*}
        c_{\kappa q}(I)
    =C_0^{\kappa}\varepsilon^{\kappa(q-1)}.
    \end{equation*}
\end{lemma}

\begin{proof}
    By \eqref{eqCnGamma} and the $\mathbb{F}_q$-linearity of $e_{\mathfrak{n}A}$,
    \begin{equation*}
        c_n(I)=\binom{n-1}{\kappa-1}\sum_{x\in\mathbb{F}_{q}}
        \rho_x(I)\,e_{\mathfrak{n}A}(x)^{\,n-\kappa}
        =(-1)^{\kappa}\binom{n-1}{\kappa-1}C_0^{\kappa}\,
        \varepsilon^{\,n-\kappa}\sum_{x\in\mathbb{F}_{q}}x^{\,n-\kappa} .
    \end{equation*}
    By \eqref{equationPowerSums} the last sum vanishes unless $n-\kappa$ is a
    positive multiple of $q-1$, in which case it equals $-1$. This proves
    \eqref{equationCnI}.

    \noindent Now let $1\leqslant n\leqslant\kappa(q+1)-1$. By \eqref{equationCnI}, $c_n(I)=0$ if $n$ is not of the form $n=\kappa+\ell(q-1)$ with $\ell\geqslant 1$. If $n$ is of this form, then $\ell(q-1)=n-\kappa\leqslant\kappa(q-1)+\kappa-1<(\kappa+1)(q-1)$, the last inequality because $\kappa-1<q-1$. Therefore $1\leqslant \ell\leqslant\kappa$.

     \noindent For $\ell<\kappa$, the base-$q$ expansion of $n-1$ is $\ell q+(\kappa-1-\ell)$ since
    $0\leqslant\kappa-1-\ell<\kappa-1<q$, so Lemma \ref{lemmaLucasBaseQ} gives
    \begin{equation*}
        \binom{n-1}{\kappa-1}\equiv\binom{\ell}{0}\binom{\kappa-1-\ell}{\kappa-1}
        \equiv 0 \bmod p,
    \end{equation*}
    the second factor vanishing since $\kappa-1-\ell<\kappa-1$. 
    
     \noindent For $\ell=\kappa$, we have $n=\kappa q$. The base-$q$ expansion of
    $\kappa q-1$ is $(\kappa-1)q+(q-1)$, so Lemma \ref{lemmaLucasBaseQ} gives
    \begin{equation*}
        \binom{\kappa q-1}{\kappa-1}
        \equiv\binom{\kappa-1}{0}\binom{q-1}{\kappa-1}
        =\binom{q-1}{\kappa-1} \bmod p .
    \end{equation*}
    As $(1+X)^{q-1}=(1+X^{q})/(1+X)=\sum_{j=0}^{q-1}(-1)^{j}X^{j}$ in
    $\mathbb{F}_p[X]$, we have $\binom{q-1}{j}\equiv(-1)^{j}\bmod p$ for
    $0\leqslant j\leqslant q-1$, whence
    $\binom{\kappa q-1}{\kappa-1}\equiv(-1)^{\kappa-1}$. Substituting into
    \eqref{equationCnI} with $n=\kappa q$,
    \begin{equation*}
        c_{\kappa q}(I)=(-1)^{\kappa-1}(-1)^{\kappa-1}C_0^{\kappa}
        \varepsilon^{\,\kappa(q-1)}=C_0^{\kappa}\varepsilon^{\,\kappa(q-1)} .
        \qedhere
    \end{equation*}
\end{proof}

\begin{proposition}\label{propOrbitVanishing}
    Let $c\in\mathfrak{n}A\setminus\{0\}$. Then for every
    $1\leqslant n\leqslant\kappa(q+1)-1$, the inner sum in \eqref{equationcnf2} vanishes: we have
    \begin{equation*}
        \sum_{d\in D_c}c_n(\gamma)=0,
    \end{equation*}
    where $\gamma$ is any representative of the coset $(c,d)$.
\end{proposition}

\begin{proof}
    For $n<\kappa$ the factor $\binom{n-1}{\kappa-1}$ of \eqref{eqCnGamma}
    vanishes, so $c_n(\gamma)=0$ for every coset and there is nothing to prove;
    assume $n\geqslant\kappa$.

     \noindent Consider the action of $\mathbb{F}_q$ on $D_c$ given by
    $x_0\cdot d=cx_0+d$; it is well defined since $cx_0\in\mathfrak{n}A$ and $\gcd(c,cx_0+d)=\gcd(c,d)=1$, and free because $c\neq 0$. Hence $D_c$ is the disjoint union of orbits of cardinality $q$. 
    Fix $d\in D_c$ and a representative $\gamma$ of the coset $(c,d)$. For $x_0\in\mathbb{F}_q$
    put $\gamma_{x_0}\coloneqq\begin{psmallmatrix}1&-x_0\\0&1\end{psmallmatrix}
    \gamma\begin{psmallmatrix}1&x_0\\0&1\end{psmallmatrix}$. It lies in
    $\Gamma(\mathfrak{n})$ (because $\Gamma(\mathfrak{n})$ is normal in $G$) and represents the coset $(c,cx_0+d)$. Moreover, with the convention $\infty+x_0=\infty$, we have $j_{\gamma_{x_0}}(x)=j_\gamma(x+x_0)$ and $\gamma_{x_0}x=\gamma(x+x_0)-x_0$ for all $x\in\mathbb{P}^1(\mathbb{F}_q)$.
    Note that
    \begin{equation}\label{eqSumOrbits}
        \sum_{d\in D_c}c_n(\gamma)=\sum_{[d]\in \mathbb{F}_q\backslash D_c}\sum_{x_0\in \mathbb{F}_q}c_n(\gamma_{x_0}).
    \end{equation}
    We are going to show that the inner sum in the right-hand side of \eqref{eqSumOrbits} vanishes.

     \noindent Since $c\neq 0$, we have $\gamma_{x_0} x\neq\infty$, and \eqref{eqCnGamma} gives
    \begin{align*}
        c_n(\gamma_{x_0})&= (-1)^\kappa C_0^\kappa\binom{n-1}{\kappa-1}\sum_{x\in\mathbb{P}^1(\mathbb{F}_q)}\,j_{\gamma_{x_0}}(x)^{\,\kappa(q-1)}\,e_{\mathfrak{n}A}(\gamma_{x_0} x)^{n-\kappa}\\
        &=(-1)^\kappa C_0^\kappa\binom{n-1}{\kappa-1}\sum_{x\in\mathbb{P}^1(\mathbb{F}_q)}\,j_\gamma(x+x_0)^{\,\kappa(q-1)}\,e_{\mathfrak{n}A}(\gamma(x+x_0)-x_0)^{n-\kappa} \\
        &=(-1)^\kappa C_0^\kappa\binom{n-1}{\kappa-1}\sum_{y\in\mathbb{P}^1(\mathbb{F}_q)}\,j_\gamma(y)^{\,\kappa(q-1)}\,(e_{\mathfrak{n}A}(\gamma y)-x_0\varepsilon)^{n-\kappa}.
    \end{align*}
    For all $y\in\mathbb{P}^1(\mathbb{F}_q)$,
    \begin{equation*}
        \sum_{x_0\in\mathbb{F}_q}(e_{\mathfrak{n}A}(\gamma y)-x_0\varepsilon)^{n-\kappa}=\sum_{j=0}^{n-\kappa}\binom{n-\kappa}{j}e_{\mathfrak{n}A}(\gamma y)^{n-\kappa-j}
        (-\varepsilon)^{j}\sum_{x_0\in\mathbb{F}_q}x_0^{\,j}.
    \end{equation*}
    By \eqref{equationPowerSums} only the terms with $j$ a positive multiple of $q-1$ survive, say $j=\ell(q-1)$ with $\ell\geqslant 1$. Note that since $q-1$ is even, $(-\varepsilon)^{j}=\varepsilon^{j}$. Thus
    \begin{equation*}
        \sum_{x_0\in\mathbb{F}_q}(e_{\mathfrak{n}A}(\gamma y)-x_0\varepsilon)^{n-\kappa}=-\sum_{\ell\geqslant 1}\binom{n-\kappa}{\ell(q-1)}
        \varepsilon^{\,\ell(q-1)}e_{\mathfrak{n}A}(\gamma y)^{\,n-\kappa-\ell(q-1)},
    \end{equation*}
    where the term indexed by $\ell$ vanishes if $\ell(q-1)>n-\kappa$. Hence the inner sum $\sum_{x_0\in\mathbb{F}_q} c_n(\gamma_{x_0})$ in the right-hand side of \eqref{eqSumOrbits} is
    \begin{equation*}
        (-1)^{\kappa+1} C_0^\kappa\sum_{y\in\mathbb{P}^1(\mathbb{F}_q)}\,j_\gamma(y)^{\,\kappa(q-1)}\sum_{\ell\geqslant 1}\binom{n-1}{\kappa-1}\binom{n-\kappa}{\ell(q-1)}
        \varepsilon^{\,\ell(q-1)}e_{\mathfrak{n}A}(\gamma y)^{\,n-\kappa-\ell(q-1)}.
    \end{equation*}
        The terms with $\ell(q-1)>n-\kappa$ vanish, and for $\ell(q-1)<n-\kappa$
    the product $\binom{n-1}{\kappa-1}\binom{n-\kappa}{\ell(q-1)}$ is
    divisible by $p$ by Corollary \ref{corLucas}. The only possibly nonzero term is thus the one with $\ell(q-1)=n-\kappa$, if it occurs. In this case,
    \begin{equation*}
        \sum_{x_0\in\mathbb{F}_q} c_n(\gamma_{x_0})=(-1)^{\kappa+1} C_0^\kappa\binom{n-1}{\kappa-1}
        \varepsilon^{n-\kappa}\sum_{y\in\mathbb{P}^1(\mathbb{F}_q)}\,j_\gamma(y)^{\,\kappa(q-1)}=0
    \end{equation*}
    by \eqref{eqSumJGamma}.
\end{proof}

We now prove Theorem \ref{theoremFisC0h}.

\begin{proof}[Proof of Theorem \ref{theoremFisC0h}]
    Combining \eqref{equationcnf2} and Proposition \ref{propOrbitVanishing}, we have $c_n(f)=c_n(I)$ for every $1\leqslant n\leqslant\kappa(q+1)-1$. The order of vanishing of $f$ at $\infty$ now follows from Lemma \ref{lemmaTrivialCosetH}.
    
    \noindent Let $\mathfrak{c}$ be a cusp of $\Gamma(\mathfrak{n})$ represented by
    $\alpha\in\mathbb{P}^1(\mathbb{F}_q)$, and write $\alpha=\nu\infty$ with
    $\nu\in\GL_2(\mathbb{F}_q)$, which is possible as
    $\GL_2(\mathbb{F}_q)$ acts
    transitively on $\mathbb{P}^1(\mathbb{F}_q)$. Since $\Gamma(\mathfrak{n})$
    is normal in $G=\GL_2(A)$, the map $\gamma\mapsto\nu^{-1}\gamma\nu$ is a
    bijection of $\Gamma(\mathfrak{n})$, and $f_\nu$ is, up to a nonzero
    constant, the same Poincaré series as $f$ but with $a_i,b_i$ replaced by $\nu^{-1}a_i,\nu^{-1}b_i$.    
    Since $\nu^{-1}$ permutes $\mathbb{P}^1(\mathbb{F}_q)$, the pairs
    $\{\nu^{-1}a_i,\nu^{-1}b_i\}$ form again a decomposition of
    $\mathbb{P}^1(\mathbb{F}_q)$. Hence the previous result applies to
    $f_\nu$, and $f$ vanishes to order exactly
    $\kappa q$ at $\mathfrak{c}$.
\end{proof}

\begin{remark}\label{remarkC0is1}
    One can always choose the parameters $a_1,b_1,\ldots,a_r,b_r\in
    \mathbb{P}^1(\mathbb{F}_q)$ so that $C_0=1$. First take $a_1=\infty$ and
    $b_1=0$. Then choose a generator $\theta$ of $\mathbb{F}_q^\times$, and for
    $2\leqslant i\leqslant r$ take $a_i=\theta^{2i-3}$, $b_i=\theta^{2i-2}$, so
    that $a_2,b_2,\ldots,a_r,b_r=\theta,\theta^2,\ldots,\theta^{q-2},\theta^{q-1}$.
    We have
    \begin{align*}
        C_0&=\prod_{i=2}^r (\theta^{2i-2}-\theta^{2i-3})
        =\prod_{i=2}^r \theta^{2i-3}(\theta-1)
        =\theta^{(r-1)^2} (\theta-1)^{r-1}.
    \end{align*}
    The order of $\theta^{r-1}=\theta^{(q-1)/2}$ in $\mathbb{F}_q^\times$ is
    exactly $2$, so $\theta^{r-1}=-1$. Similarly $(\theta-1)^{r-1}$ has order at
    most $2$, so $(\theta-1)^{r-1}=\pm 1$. Therefore
    $C_0= (-1)^{r-1} (\theta-1)^{r-1}=\pm 1$. Swapping $a_2$ and $b_2$ if
    needed, we get $C_0=1$.
\end{remark}

\begin{remark}\label{remarkFrobenius}
    The hypothesis $\kappa\leqslant q-1$ of Theorem \ref{theoremFisC0h} can be relaxed. Indeed, suppose that $\kappa=p^j\kappa'$ with $1\leqslant \kappa'\leqslant q-1$ (recall that $p$ is the characteristic of $\mathbb{F}_q$). Since the Frobenius map is additive and continuous, 
    \begin{align*}
        f_{\kappa}^{\mathfrak{n}}&=\sum_{\gamma\in\Gamma(\mathfrak{n})}\prod_{i=1}^{r}u(z;\gamma a_i,\gamma b_i)^{p^j\kappa'}\\
        &=\left(\sum_{\gamma\in\Gamma(\mathfrak{n})}\prod_{i=1}^{r}u(z;\gamma a_i,\gamma b_i)^{\kappa'}\right)^{p^j}\\
        &=\big(f_{\kappa'}^{\mathfrak{n}}\big)^{p^j}.
    \end{align*}
    Therefore
    \begin{align*}
        f_{\kappa}^{\mathfrak{n}}(z)&=\left(C_0^{\kappa'}\,\varepsilon^{\kappa'(q-1)}\,t(z)^{\kappa' q}
        +O\big(t(z)^{\kappa'(q+1)}\big)\right)^{p^j}\\
        &=C_0^{\kappa}\,\varepsilon^{\kappa(q-1)}\,t(z)^{\kappa q}
        +O\big(t(z)^{\kappa(q+1)}\big).
    \end{align*}
    It follows that Theorem \ref{theoremFisC0h} remains valid for every
    \begin{equation*}
        \kappa\in \big\{p^{j}\kappa'\ :\ j\geqslant 0,\ 1\leqslant \kappa'\leqslant q-1\big\}.
    \end{equation*}
\end{remark}

\subsection{$h$ and $\Delta$ as Poincaré series}\label{subsectionHDelta}

We now take $\mathfrak{n}=T$, so that $g=0$, $c=q+1$, and the cusps of
$\Gamma(T)$ are exactly the $q+1$ points of
$\mathbb{P}^1(\mathbb{F}_q)$ (Example
\ref{exampleGenerationGamma(T)}). We write $t=t^{\Gamma(T)}=t_{TA}$ and
$\varepsilon=e_{TA}(1)$, and we set
\begin{equation*}
    h\coloneqq P_{q+1,1}\in M_{q+1,1}(G),
\end{equation*}
which by Remark \ref{remarkNormalization} equals $\widetilde{\pi}$ times
Gekeler's in \cite[(5.13)]{Gekeler1988}. Recall that the graded algebra of Drinfeld modular forms
for $G$ is the polynomial ring $\mathbb{C}_\infty[g,h]$ with
$g\in M_{q-1,0}(G)$ \cite[(5.13)]{Gekeler1988}, and that
$\Delta\coloneqq-h^{q-1}$ lies in $M_{q^2-1,0}(G)$. We are going to compare $h$ and $\Delta$ with the Poincaré series $f_\kappa^T$ defined in \S\ref{subsectionFamily}.

\noindent Recall that the \textit{Carlitz module} is the Drinfeld module of rank
$1$ defined by $\phi_T(x)=Tx+x^q$; its exponential function is $e_L$,
where $L\coloneqq\widetilde{\pi}A$, and it satisfies the functional
equation $e_{L}(Tz)=\phi_T(e_{L}(z))$.

\begin{lemma}\label{lemmaEone}
    We have $e_{TA}(1)^{\,q-1}=-T^{q}\,\widetilde{\pi}^{\,1-q}$.
\end{lemma}

\begin{proof}
    Since
    $\widetilde{\pi}/T\notin L=\ker e_L$, we have
    $e_L(\widetilde{\pi}/T)\neq 0$.
     For $z=\widetilde{\pi}/T$, the functional equation of $e_L$ gives
    \begin{equation*}
        0=Te_{L}(\widetilde{\pi}/T)+e_{L}(\widetilde{\pi}/T)^q
    \end{equation*}
    i.e., $e_{L}(\widetilde{\pi}/T)^{q-1}=-T$. On the other hand $e_{TA}(z)=T\,e_A(z/T)$
    and $e_A(z)=\widetilde{\pi}^{-1}e_L(\widetilde{\pi}z)$, so that
    $e_{TA}(1)=T\widetilde{\pi}^{-1}e_L(\widetilde{\pi}/T)$ and
    \begin{equation*}
        e_{TA}(1)^{\,q-1}
        =T^{\,q-1}\widetilde{\pi}^{\,1-q}e_L(\widetilde{\pi}/T)^{\,q-1}
        =-T^{\,q}\widetilde{\pi}^{\,1-q}. \qedhere
    \end{equation*}
\end{proof}

We can now expand $h$ in the parameter $t=t^{\Gamma(T)}$.

\begin{lemma}\label{lemmaHorder}
    As a modular form for $\Gamma(T)$, we have
    \begin{equation*}
        h=\varepsilon^{\,q-1}\,t^{q}+O(t^{2q-1}).
    \end{equation*}
    In particular $h$ vanishes to order exactly $q$ at every cusp of
    $\Gamma(T)$, so that $h\in M_{q+1}^{q}(\Gamma(T))$.
\end{lemma}

\begin{proof}
    The functional equation $e_{L}(Tz)=\phi_T(e_{L}(z))$, together with
    $e_A(z)=\widetilde{\pi}^{-1}e_{L}(\widetilde{\pi}z)$, gives
    \begin{equation*}
        e_A(Tz)=Te_A(z)+\widetilde{\pi}^{\,q-1}e_A(z)^{q}.
    \end{equation*}
    Substituting $z/T$ for $z$ and using the identity
    $Te_A(z/T)=e_{TA}(z)$, we obtain
    \begin{equation*}
        e_A(z)=e_{TA}(z)+\widetilde{\pi}^{\,q-1}T^{-q}e_{TA}(z)^{q},
    \end{equation*}
    that is $(t^{G})^{-1}=t^{-1}+\widetilde{\pi}^{\,q-1}T^{-q}t^{-q}$.
    Solving for $t^{G}$ and using
    $T^{q}\widetilde{\pi}^{\,1-q}=-\varepsilon^{\,q-1}$
    (Lemma \ref{lemmaEone}), we get
    \begin{equation*}
        t^{G}=\frac{-\varepsilon^{\,q-1}t^{q}}
        {1-\varepsilon^{\,q-1}t^{\,q-1}}
        =-\varepsilon^{\,q-1}t^{q}+O\big(t^{2q-1}\big)
    \end{equation*}
    for $|z|_i\gg 0$. By \cite[(10.4)]{Gekeler1988} and
    Remark \ref{remarkNormalization} one has
    $h=-t^{G}+O\big((t^{G})^{2}\big)$, whence
    $h=\varepsilon^{\,q-1}t^{q}+O(t^{2q-1})$. Finally $\Gamma(T)$ is
    normal in $G$, so $h$ has the same
    order of vanishing at every cusp of $\Gamma(T)$.
\end{proof}

\begin{theorem}\label{theoremFisC0h2}
    Let $1\leqslant\kappa\leqslant q-1$. Then
    $f_\kappa^{T}\in M_{\kappa(q+1)}^{\kappa q}(\Gamma(T))$ and
    \begin{equation*}
        f_\kappa^{T}=C_0^{\kappa}\,h^{\kappa}.
    \end{equation*}
\end{theorem}

\begin{proof}
    By Theorem \ref{theoremFisC0h}, $f_\kappa^{T}$ vanishes to order
    exactly $\kappa q$ at each of the $q+1$ cusps of $\Gamma(T)$, so
    $f_\kappa^{T}\in M_{\kappa(q+1)}^{\kappa q}(\Gamma(T))$; by
    Lemma \ref{lemmaHorder} the same holds for $h^{\kappa}$. Applying
    \eqref{eqdim} with $2k=\kappa(q+1)$ and $n=\kappa q$, we see that $\dim M_{\kappa(q+1)}^{\kappa q}(\Gamma(T))=1$. Thus $f_\kappa^{T}=C\,h^{\kappa}$ for some
    $C\in\mathbb{C}_\infty^\times$. It remains to compare the
    coefficients of $t^{\kappa q}$: these are
    $C_0^{\kappa}\varepsilon^{\kappa(q-1)}$ for $f_\kappa^{T}$ by
    Theorem \ref{theoremFisC0h}, and $\varepsilon^{\kappa(q-1)}$ for
    $h^{\kappa}$ by Lemma \ref{lemmaHorder}. Hence $C=C_0^{\kappa}$.
\end{proof}

\begin{remark}\label{remarkThmh}
    By Remark \ref{remarkFrobenius}, Theorem \ref{theoremFisC0h2}
    remains valid for every
    $\kappa\in\big\{p^{j}\kappa'\ :\ j\geqslant 0,\
    1\leqslant\kappa'\leqslant q-1\big\}$. If the parameters are normalized so that $C_0=1$ (Remark \ref{remarkC0is1}), then $f_\kappa^{T}=h^{\kappa}$, which shows that $h^{\kappa}$ is a Poincaré series for the group $\Gamma(T)$.
\end{remark}

\begin{corollary}
    For every $\kappa\in\big\{p^{j}\kappa'\ :\ j\geqslant 0,\
    1\leqslant\kappa'\leqslant q-1\big\}$ we have
    $f_\kappa^{T}=\big(f_1^{T}\big)^{\kappa}$. Moreover
    \begin{equation*}
        \Delta=-\big(f_1^{T}\big)^{q-1}=-f_{q-1}^{T}.
    \end{equation*}
\end{corollary}
\begin{proof}
    By Theorem \ref{theoremFisC0h2} and Remark \ref{remarkThmh},
    $f_\kappa^{T}=C_0^{\kappa}h^{\kappa}=(C_0h)^{\kappa}
    =\big(f_1^{T}\big)^{\kappa}$. Since $C_0\in\mathbb{F}_q^\times$ we
    have $C_0^{q-1}=1$, so $f_{q-1}^{T}=h^{q-1}=-\Delta$ and likewise
    $\big(f_1^{T}\big)^{q-1}=(C_0h)^{q-1}=h^{q-1}=-\Delta$.
\end{proof}

\begin{remark}
    Petrov proved in \cite[Theorem 3.16]{Petrov2013} that $h^\kappa$ has an $A$-expansion for all $1\leqslant\kappa\leqslant q$, namely
    $h^\kappa=(-1)^\kappa\sum_{a\in A_+}a^{\kappa q}t_a^{\kappa}$ where
    $t_a(z)=t^{G}(az)$ and $A_+$ is the set of monic elements of $A$ (note that in \cite{Petrov2013}, $h$ is normalized so that the first non-zero coefficient of its $t$-expansion is $1$).
    Combined with Theorem \ref{theoremFisC0h2}, this gives for
    $1\leqslant\kappa\leqslant q$ two rather different expressions for the same modular form: assuming $C_0=1$,
    \begin{equation*}
        \sum_{\gamma\in\Gamma(T)}
        \prod_{i=1}^{r}u(z;\gamma a_i,\gamma b_i)^{\kappa}
        =(-1)^\kappa\sum_{a\in A_+}a^{\kappa q}t_a(z)^{\kappa}.
    \end{equation*}
    It would be interesting to prove this equality directly.
\end{remark}

\section*{AI Declaration}
The authors used Claude Opus 5.5 (Anthropic) to rephrase some sentences in the manuscript in order to improve its language.

\bibliographystyle{plain} 
\bibliography{references}
\end{document}